\documentclass[11pt,a4paper,twoside,reqno]{amsart}
    \usepackage[
      a4paper,
      left=4cm, right=4cm,
      top=4cm,  bottom=4cm,
      includeheadfoot,   
      headheight=14pt,   
      footskip=18pt      
    ]{geometry}
    
    \usepackage[all]{xy}
    \usepackage{amsmath,amssymb,amsfonts,amsthm}
    \usepackage[toc,page]{appendix}
    \usepackage{verbatim} 
    \usepackage{graphicx} 
    
    \usepackage{amsfonts}
    \usepackage{amsthm}
    \usepackage{amsmath}
    \usepackage{stmaryrd}
    \usepackage{amscd}
    \usepackage{mathtools}
    \usepackage[latin2]{inputenc}
    \usepackage{t1enc}
    \usepackage[mathscr]{eucal}
    \usepackage{indentfirst}
    \usepackage{graphicx}
    \usepackage{graphics}
    \usepackage{pict2e}
    \usepackage{epic}
    \usepackage{faktor}
    \usepackage{xfrac}   
    \numberwithin{equation}{section}
    \usepackage{epstopdf} 
    
    \usepackage{hyperref}
    \usepackage{lipsum}
    \usepackage{tikz-cd}
    \usepackage{pict2e}
    \usepackage{epic}
    \usepackage{float}
    \usepackage{caption}
    \usepackage{amssymb}
    \numberwithin{equation}{section}
    \usepackage{epstopdf} 
    \usepackage{comment}
    \usepackage{tikz-cd}
    
    \usepackage{microtype}

    \usepackage[backend=bibtex,
    style=numeric,
    citestyle=numeric-comp,
    sorting=nty]{biblatex}

    \hypersetup{
        colorlinks=true,
        linkcolor=blue,
        citecolor=blue,
        urlcolor=blue
    }
    
    \theoremstyle{plain}
    \newtheorem{Th}{Theorem}[section]
    \newtheorem{Lemma}[Th]{Lemma}
    \newtheorem{Cor}[Th]{Corollary}
    \newtheorem{Prop}[Th]{Proposition}
    
    \newtheorem{claim}[Th]{Claim}
    
    \theoremstyle{definition}
    \newtheorem{Def}[Th]{Definition}
    
    \newtheorem{Rem}[Th]{Remark}
    \newtheorem{?}[Th]{Problem}
    \newtheorem{Ex}[Th]{Example}

    \author{Hoan Nguyen}
    \address{Department of Mathematics, The University of Chicago, Chicago, IL 60615}
    \email{nth@uchicago.edu}

    \numberwithin{equation}{section}
\begin{document}

\title[Min-Max theory and minimal foliations]
{A Min-Max Gap Characterization of Minimal Foliations on the Torus}

\begin{abstract}
We extend an energy introduced by Mather to the setting of Almgren-Pitts min-max theory and obtain a parametric, higher-dimensional analogue of Mather's variational barrier theory for twist maps and geodesics on tori. 
We use this energy to establish a min-max criterion for the existence of foliations of the $n$-torus by minimal hypersurfaces, when $3\leq n\leq 7$. We show that for a generic metric, whenever a lamination by area-minimizing hypersurfaces of the $n$-torus contains a gap, there exists a minimal hypersurface inside the gap that is not area-minimizing. This hypersurface is a higher-dimensional analogue of the secondary minimax orbit appearing in Aubry-Mather theory. 

\end{abstract}
\maketitle
\setcounter{tocdepth}{1}
\tableofcontents

\section{Introduction and main results}

    In this paper, we study minimal hypersurfaces in the $n$-torus $\mathbb{T}^n=\mathbb{R}^n/\mathbb{Z}^n$ equipped with a metric $g$. This problem may be viewed as a parametric version of the Aubry-Mather theory, in which area-minimizing hypersurfaces play a role analogous to that of minimal orbits in convex Hamiltonian systems. In two dimensions, most of the results on Hamiltonian twist maps carry over to the case of minimizing geodesics on a two-torus (cf. \cite{nonrecurrent}).

    In $1986$, Moser \cite{moser} first generalized the Aubry-Mather theory to higher dimensions. He studied minimizing graphs $u: \mathbb{R}^n\to \mathbb{R}$ of a $\mathbb{Z}^n$-periodic variational integral.
    This may be regarded as a PDE analogue of the Aubry-Mather theory. A related problem, concerning minimizers of the elliptic integral on Caccioppoli sets, was studied by Caffarelli and de la Llave in \cite{planelike}. In both settings, the authors established the existence of \textit{plane-like} minimizers, namely minimizers contained in a strip of uniformly bounded width between two parallel affine planes. 

    On a torus with an arbitrary metric, however, area-minimizing hypersurfaces need not be graphs. Auer and Bangert \cite{bangertauer} and Auer \cite{auer} observed that tools from geometric measure theory could provide a natural framework in this setting. Roughly speaking, after passing to the universal cover, each area-minimizing hypersurface with the Birkhoff property (see Section \ref{section4}) admits a unique homological direction $\alpha\in H_{n-1}(\mathbb{T}^n, \mathbb{R})$. Let $\mathcal{M}_\alpha$ denote the family of such hypersurfaces with direction $\alpha$. They proved that hypersurfaces in $\mathcal{M}_\alpha$ with the same asymptotic behavior form a lamination on $\mathbb{R}^n$ whose projection is a minimal lamination on the torus (see Sections \ref{section4} and \ref{section5}). A fundamental question, however, is whether such a lamination is a foliation. This happens in some exceptional cases, such as flat tori, but fails in general for certain metrics constructed by Bangert in \cite{bangertgap}.\\

A main goal of this paper is to use the Almgren-Pitts min-max theory to derive criteria for the existence of minimal foliations by hypersurfaces in $\mathcal{M}_\alpha$. For an integral class $\alpha$, a standard Lusternik-Schnirelmann theory-type argument (see Appendix \ref{appendix:a}) shows that $\mathbb{T}^n$ is foliated by a one-parameter family of closed minimal hypersurfaces in the class $\alpha$ exactly when \[\omega(\alpha)-S(\alpha)=0,\]
where $\omega(\alpha)$ is the min-max width of the class $\alpha$ and $S(\alpha)$ is its stable norm. We refer readers to Sections \ref{section2} and \ref{section4} below for the corresponding definitions. 

We aim to extend this picture to irrational homology classes, where the width is not defined. To overcome this, we study the asymptotic behavior of the quantity $\omega(r)-S(r)$ as the integer class $r$ approaches $\alpha$.

In the case of geodesics on a two-dimensional torus, such a criterion was obtained by Mather in \cite{mather}: given a rational direction $r=[p,q]\in H_1(\mathbb{T}^2, \mathbb{Z})$ with $p,q$ relatively prime, let $\Delta W_r$ be the infimum of all $\epsilon>0$ such that any two periodic geodesics $\gamma_0, \gamma_1 \in \mathcal{M}_r$ can be joined by a continuous family $\gamma_t$ of $(p,q)$-periodic curves satisfying 
\[ L(\pi(\gamma_t))\leq L(\pi(\gamma_0))+\epsilon \quad \text{for all $t$}, \]
where $\pi: \mathbb{R}^2\to \mathbb{T}^2$ is the projection.
Mather proved that for every irrational direction $\alpha\in H_1(\mathbb{T}^2, \mathbb{R})$, $\mathcal{M}_\alpha$ is a foliation if and only if:
\[\Delta W_\alpha:=\lim_{p/q\to \alpha}\Delta W_r=0.\]

 Mather's proof was analytic and was formulated in the context of twist maps on a torus, which has no immediate generalization to higher dimensions. Our first main theorem utilizes the tools from geometric measure theory and min-max theory to generalize it to the case of hypersurfaces on the higher-dimensional $n$-torus, with $3\leq n\leq 7$. The dimensional restriction comes from the regularity theory for minimal hypersurfaces. 
\begin{Def}\label{convergencestable}
    Throughout the paper, we say that a sequence $\{r_k\}_k$ of primitive integer homology classes converges to a class $\alpha\neq 0$ if the normalized sequence $\dfrac{r_k}{S(r_k)}$ converges to $\dfrac{\alpha}{S(\alpha)}$ on the boundary of the unit ball of the stable norm.
\end{Def}

In light of Mather's criterion, we define the following:
\begin{Def}
    For an integer class $r\in H_{n-1}(\mathbb{T}^n, \mathbb{Z})$, we define its \textit{Mather-type energy barrier} as 
    \[\Delta W_r:=\omega(r)-S(r).\]
    
    We extend this energy to irrational classes as follows:
    \[
\Delta W_\alpha
:= \liminf_{\substack{r\in H_{n-1}(\mathbb{T}^n,\mathbb{Z})\\ r\to\alpha}}
\bigl(\omega(r)-S(r)\bigr).
\]

    Here, the finiteness of the right-hand side follows from Proposition \ref{upperbound}.
\end{Def}

    This should be viewed as a higher-dimensional analogue of the $\epsilon$ above. As discussed, this energy conceptually measures the obstruction to moving through the class $r$ by almost area-minimizing hypersurfaces. Our first main result shows that, for totally irrational classes, the vanishing of this barrier asymptotically characterizes minimal foliations. 
\begin{Th}\label{maintheorem1}
    Suppose that $\alpha \in H_{n-1}(\mathbb{T}^n, \mathbb{R})$ is totally irrational (see Definition \ref{totallyirrational}). The following statements are equivalent.
    \begin{enumerate}
        \item $\lim_{k\to \infty} \Delta W_{r_k}=0$ for all sequences of integer classes $\{r_k\}_k$ converging to $\alpha$.
        \item $\lim_{k\to \infty} \Delta W_{r_k}=0$ for some sequence of integer classes $\{r_k\}_k$ converging to $\alpha$.
        \item The elements of $\mathcal{M}_\alpha$ form a foliation of $\mathbb{R}^n$ by area-minimizing hypersurfaces. Moreover, this foliation is invariant under the deck group and descends to a minimal foliation of $\mathbb{T}^n$.
\end{enumerate}
\end{Th}

The following is a direct consequence of Theorem \ref{maintheorem1}.
\begin{Cor}
    For $\alpha\in H_{n-1}(\mathbb{T}^n, \mathbb{R})$ that is totally irrational, the universal cover $\mathbb{R}^n$ is foliated by area-minimizing hypersurfaces in $\mathcal{M}_\alpha$ if and only if $\Delta W_\alpha=0$.
\end{Cor}
We also mention the work of Aubry and Le Daeron \cite{aubry-daeron} in connection with the Hamiltonian twist maps. Motivated by a problem in solid-state physics, they introduced a quantity called the \textit{Peierls's barrier} and obtained a conclusion similar to that of Mather and ourselves. They showed that, for a given frequency, this quantity vanishes if and only if there is an invariant circle. Indeed, the Peierls's barrier is a lower bound of the Mather energy barrier defined above.

We highlight several points in Theorem~\ref{maintheorem1}.
First, the implication $(1)\implies(2)$ is immediate.
Second, the proof in Section~\ref{sec:s7-sufficiency}
shows that $(2)$ implies the existence of a foliation
by elements of $\mathcal{M}_\alpha$ for any nonzero
homology class $\alpha$. When $\alpha$ is totally
irrational, this is equivalent to $(3)$.
For rationally dependent $\alpha$, however, the existence
of such a foliation does not necessarily imply $(1)$,
as demonstrated by the following theorem.

\begin{Th}\label{maintheorem2}
    For any nonzero homology class $\alpha\in H_{n-1}(\mathbb{T}^n, \mathbb{R})$ such that $\alpha^\perp \cap (\mathbb{Z}^n \setminus \{0\})\neq \emptyset$, there exists a Riemannian metric $g$ on $\mathbb{T}^n=\mathbb{R}^n/\mathbb{Z}^n$ and a sequence $\{r_k\}_k$ of integer classes converging to $\alpha$ such that \[\liminf_{k\to \infty}\Delta W_{r_k}>0,\] yet $\mathbb{R}^n$ admits a foliation by area-minimizing hypersurfaces in $\mathcal{M}_\alpha$.
\end{Th}

Nonetheless, for such $\alpha$, the convergence of $\Delta W_{r_k}$ to $0$ does hold under a stronger condition that the foliation in $\mathcal{M}_\alpha$ is unique. Let $\mathcal{M}(\alpha)\subset \mathcal{M}_\alpha$ be the subset consisting of the elements with maximal periodicity (see Section \ref{section4}).

\begin{Th}\label{maintheorem3}
    For any nonzero homology class $\alpha\in H_{n-1}(\mathbb{T}^n, \mathbb{R})$, if $\mathcal{M}(\alpha)$ is a foliation, then for any sequence of integer classes $\{r_k\}_k$ converging to $\alpha$, we have $\lim\limits_{k\to\infty} \Delta W_{r_k}=0$.
\end{Th}
Conversely, if $\Delta W_{r_k}\to0$ along every sequence $r_k\to\alpha$, then $\mathcal{M}(\alpha)$ is a foliation (see Remark \ref{rem:s7-most-periodic}). For integral classes $\alpha$, this is immediate from the constant sequence $r_k=\alpha$. Vanishing along a single sequence, however, does not suffice. Indeed, one can construct a metric on $\mathbb{T}^n$ and a primitive class $\alpha$ such that $\mathcal{M}(\alpha)$ is not a foliation, while for every primitive homology class $r\in H_{n-1}(\mathbb{T}^n, \mathbb{Z})$ other than $\pm\alpha$ there exists a foliation of $\mathbb{T}^n$ by closed, homologically area-minimizing hypersurfaces in $r$. See Example \ref{everydirectionexceptone}.
\medskip

In the case where the minimizing hypersurfaces do not fill all of $\mathbb{R}^n$ (that is, they form a genuine lamination with gaps), we prove that inside every gap, there is a complete minimal hypersurface that is not area-minimizing. This is an analogue of the following fact from the Aubry-Mather theory for the double-well potential shown in \cite{doublewell}: if the set of minimizing solutions (also called the ground states) does not attain all possible values at a point, then there is another non-minimizing critical point in between. In the setting of variational monotone lattice recurrent relations, Mramor and Rink \cite{ghostcircle} proved a similar result: when an Aubry-Mather set has a gap, then this gap must be foliated by minimizers or contain a non-minimizing critical point.  These critical points are often viewed as secondary laminations inside the gap. 

Recall that a Riemannian metric $g$ on a closed manifold $M$ is called \textit{bumpy} if every immersed, closed minimal hypersurface is non-degenerate (having no non-trivial Jacobi field). Brian White proved in \cite{white} that bumpy metrics are generic in the sense of the Baire category.

\begin{Th}\label{nonminimizer}
    Suppose that the metric $g$ on $\mathbb{T}^n$ is bumpy and $\alpha$ is a totally irrational class. Assume further that elements of $\mathcal{M}_\alpha$ do not foliate $\mathbb{R}^n$. Then for any two consecutive minimizers $\Sigma_0$ and $\Sigma_1$ in $\mathcal{M}_\alpha$ that bound a gap $G$, there exists a complete, non-compact, embedded minimal hypersurface lying entirely inside $G$ that is not area-minimizing and has index at most one.
\end{Th}

For geodesics on a two-torus, the existence of a similar non-area-minimizing critical point inside every gap was proved by Mather in \cite{mather}. His construction, in a more general setting of a monotone twist map on tori, used a minimax principle formulated by Birkhoff. Mather interpreted this critical point as a secondary orbit that is \textit{homoclinic} to the original Cantor set (the lamination formed by recurrent minimizers). The minimal hypersurface constructed in Theorem \ref{nonminimizer} is a higher-dimensional analogue of this secondary orbit.

Nevertheless, our approach to constructing a mountain-pass type critical point inside the gap is different from Mather's. Roughly speaking, we construct such hypersurfaces as limits of sequences of index-one minimal hypersurfaces with prescribed boundaries on larger and larger compact exhaustions. This strategy is similar to the one employed by Carlotto and De Lellis \cite{minmaxgeodesic} in their construction of min-max geodesics on asymptotically conical surfaces. \\

Although the minimal hypersurface constructed in Theorem \ref{nonminimizer} is non-compact, we can still describe its area asymptotically inside a large compact ball, in comparison to that of an area-minimizer. Let $B_R=B_R(p)$ be the ball centered at a point $p$ of radius $R$ with respect to a metric $\Tilde{d}$ (see Section \ref{section9} for the definition). 

\begin{Th}\label{areaofnonminimizer}
    Under the assumptions of Theorem \ref{nonminimizer}, there exists a gap $G$ formed by two consecutive minimizers $\Sigma_0$ and $\Sigma_1$ of $\mathcal{M}_\alpha$ such that if $\Gamma$ is a non-area-minimizing minimal hypersurface constructed inside the gap $G$ by Theorem \ref{nonminimizer}, then the following holds.
    \begin{enumerate}
        \item If there is no closed minimal hypersurface inside $G$, then
        \[\limsup_{R\to \infty} \big(\mathbf{M}(\Gamma\llcorner B_R(p)) -\mathbf{M}(\Sigma_0 \llcorner B_R(p))\big)\geq \Delta W_\alpha.\]
        \item Otherwise, there exists a disjoint collection of closed, connected, stable minimal hypersurfaces $\mathcal{S}=\{S_1, S_2,\ldots, S_j\}$ inside $G$ and pairwise disjoint open sets $U_{S_1}, U_{S_2}, \ldots, U_{S_j}$ with $S_i=\partial U_{S_i}$ such that either
        \[\sum_{i=1}^j W_\partial(U_{S_i})\geq \Delta W_\alpha\]
        or 
        \[\limsup_{R\to \infty} \left(\mathbf{M}(\Gamma\llcorner B_R)-\mathbf{M}(\Sigma_0\llcorner B_R)\right)\geq \Delta W_\alpha -\sum_{i=1}^j \mathbf{M}(S_i).\]
        Here, $W_\partial(U_{S_i})$ denotes the width of the manifold $U_{S_i}$ with the strictly stable boundary $S_i$, as defined in Section \ref{section2}.
    \end{enumerate}
\end{Th}

\subsection{Examples and related work}
\subsubsection{Minimal lamination}
On a flat torus, every direction determines a minimal foliation by parallel hyperplanes. The question of whether such foliations exist for general metrics was introduced by Moser in \cite{moser} and Bangert in \cite{bangert1}.

In \cite{stability}, Moser was interested in studying the stability of such minimal foliations. He called a minimal foliation $F_g$ \textit{stable} if there exists a $C^\infty$-neighborhood $U$ of the metric $g$ such that for every $g'\in U$, there exists a foliation $F_{g'}$ that is minimal with respect to $g'$ and diffeomorphic to $F_g$. 

On a flat torus, when $\alpha$ is an integral class, the minimal foliation by hyperplanes $x\cdot \alpha= \beta, \beta\in \mathbb{R}$ is not stable: arbitrarily small perturbations can break it into only finitely many closed area-minimizing hypersurfaces in the same homology class. In fact, this finite behavior is generic for primitive integral classes.

On the other hand, Moser proved that for $\alpha =(\alpha_1, \alpha_2, \ldots, \alpha_n) \in \mathbb{R}^n$ which satisfies the following Diophantine condition
\begin{equation}\label{diophantine}
    \sum_{i, j=1}^{n} \left(\alpha_id_j-\alpha_jd_i\right)^2\geq \gamma \left(\sum_{j=1}^{n} d_j^2\right)^{-\tau}, \forall d=(d_1, d_2, \ldots, d_n)\in \mathbb{Z}^n \setminus \{0\},
\end{equation}
for some constants $\gamma, \tau>0$, the minimal foliation by affine hyperplanes in the direction $\alpha$ is indeed stable. 
Geometrically, this condition prevents nonzero lattice points from approaching the line \(\mathbb R\alpha\) faster than a prescribed polynomial rate.

We also note that Bangert \cite{bangertgap} constructed $\mathbb{Z}^n$-periodic metrics for which no minimal foliation exists. For these metrics, it follows from Theorem \ref{maintheorem1} that there exists a constant $c>0$ such that $\omega(r)-S(r)>c$ for all $r\in H_{n-1}(\mathbb{T}^n, \mathbb{Z})$.

A related inverse question is when a given foliation can be made minimal for some Riemannian metric, that is, when it is \textit{taut}. Sullivan \cite{taut} proved that a codimension-one oriented foliation admits such a metric if and only if every compact leaf is cut by a closed transversal curve.

\subsubsection{Differentiability of the stable norm}
In \cite{stablenorm} and \cite{hannes}, Auer, Bangert, and Junginger-Gestrich investigated the differentiability properties of the stable norm 
\[S: H_{n-1}(M, \mathbb{R})\to \mathbb{R}.\] 

They showed that the restriction of $S$ to the minimal subspace spanned by the integer classes and $\alpha$ (see Definition \ref{totallyirrational}) is differentiable at $\alpha$. 
In particular, it is always differentiable at totally irrational classes. For rationally dependent classes, $S$ is differentiable at $\alpha$ exactly when $\mathcal{M}(\alpha)$ is a minimal foliation. More precisely, if $\mathcal{M}(\alpha)$ is not a foliation, heteroclinic minimizers in different directions in $\mathcal{M}_\alpha$ give rise to different subderivatives of $S$ at $\alpha$.

It is a consequence of these works that we can detect minimal foliations in rationally dependent classes consisting of the most periodic minimizers by examining the differentiability of the stable norm. It remains unclear, however, what the corresponding picture is for totally irrational classes.

\subsubsection{Unstable minimal hypersurfaces on tori}

In a flat three-torus, index-one minimal surfaces have been extensively studied. Ross \cite{ross} showed that the classical Schwarz $\mathcal{P}$ and $\mathcal{D}$ minimal surfaces, as well as A. Schoen's gyroid surface, all have Morse index one on the flat torus given by the quotient of the flat $\mathbb{R}^3$ by the corresponding lattice groups. In \cite{hpr}, the authors constructed (possibly non-orientable) minimal surfaces with index less than or equal to one on any flat torus. On the other hand, Ritor\'{e} and Ros \cite{ros} showed that the moduli spaces of orientable index-one minimal surfaces with genus greater than one in (non-fixed) flat three-tori are compact. As a consequence, on flat tori with prescribed volume, orientable index-one minimal surfaces do not exist when the injectivity radius is sufficiently small. 

All known index-one examples listed above are homologically trivial. Indeed, by a theorem of Meeks \cite{meeks}, any closed two-sided minimal surface that is not totally geodesic in a flat three-torus is a Heegaard surface, hence null homologous (the same statement holds for any closed manifold of nonnegative Ricci curvature). On the other hand, totally geodesic minimal surfaces on a flat three-torus are flat two-tori, and they foliate the torus for each non-trivial homology class.

One can also see the non-existence of a connected, unstable minimal surface in any non-trivial homology class by the maximum principle. 

It is currently unknown if there exist infinitely many connected, embedded, two-sided, index-one minimal surfaces on the cubic flat three-torus. 

\subsection{Plan of the paper}
In Section \ref{section2}, we briefly review the necessary notation and the basic framework of one-parameter Almgren-Pitts min-max theory. We use that to define the \textit{width} associated with an integer homology class. In Section \ref{section3}, we reformulate an interpolation theorem due to Marques-Neves in a form that allows us to construct a mass-continuous map interpolating between two currents that are close in the mass norm. This is a crucial tool in the proofs of Theorems \ref{maintheorem1} and \ref{maintheorem3}. The detailed proofs of these results are carried out in Section \ref{section6}. We briefly summarize the theory of area-minimizing hypersurfaces in tori in Section \ref{section4}, with a continued focus on irrational homological direction in Section \ref{section5}. In Section \ref{section7}, we use a Lusternik-Schnirelmann-type argument to derive a sufficient condition for the existence of a foliation in a given homology class. Examples of metrics that satisfy Theorem \ref{maintheorem2} are constructed in Section \ref{section8}. Section \ref{section9} contains the proof of Theorem \ref{nonminimizer} and Theorem \ref{areaofnonminimizer}, in which we construct and analyze the asymptotic area of complete, non-compact minimal hypersurfaces that are not area-minimizing inside the gaps of an irrational lamination, using a version of min-max theory with boundary developed by Montezuma. We conclude with a discussion of some open problems in the last section.
\\

\textbf{Acknowledgements}. I am deeply grateful to my advisor, Andr\'{e} Neves, for suggesting the research questions that led to this work, and for his many helpful discussions, guidance, and support throughout the project.

\section{Almgren-Pitts min-max theory}\label{section2}
\subsection{The width of an integer homology class}

In this section, we recall some standard definitions from geometric measure theory. Most of these can be found in the book of Simon \cite{simon}. Let $ M^n$  be a closed $ n$-dimensional manifold equipped with a Riemannian metric $g$. We assume that $(M^n,g)$ is isometrically embedded in some Euclidean space $\mathbb{R}^J$. For $0\leq k \leq n$, denote by $\mathcal{I}_{k}(M)$ the space of $k$-dimensional flat chains with integer coefficients in $M$ and let $\mathcal{Z}_{n-1}(M)$ be the space of all $(n-1)$-dimensional integral cycles $T$ in $\mathcal{I}_{n-1}(M)$ satisfying $T=\partial Q$ for some $Q\in \mathcal{I}_{n}(M)$. This is precisely the connected component of the space of integral currents without boundary that contains the zero current. 

For any $T\in \mathcal{I}_k(M)$, we write $|T|$ and $||T||$ for the integral varifold and the Radon measure in $M$ associated with $T$. The mass of a current $T$ is defined as:
\[\mathbf{M}(T)=\sup \{T(\phi): \phi\in \mathcal{D}^k(\mathbb{R}^J), ||\phi||\leq 1\},\]
where $\mathcal{D}^k(\mathbb{R}^J)$ is the space of smooth $k$-forms with compact support. 

The flat metric on $\mathcal{I}_k(M)$ is defined as: 
\[\mathcal{F}(S,T)=\inf\{\mathbf{M}(P)+\mathbf{M}(Q): S-T=P+\partial Q, P\in \mathcal{I}_k(M), Q\in \mathcal{I}_{k+1}(M)\}.\]

The $\mathbf{F}$-metric on the closure $\mathcal{V}_k(M)$ of the space of $k$-dimensional rectifiable varifolds with support contained in $M$ is defined by:
\[\mathbf{F}(V, W)=\sup \{V(f)-W(f): f\in C_c(G_k(\mathbb{R}^J)), |f|\leq 1, \operatorname{Lip}(f)\leq 1 \}.\]

This induces the following $\mathbf{F}$-metric on $\mathcal{I}_k(M)$:
\[\mathbf{F}(S,T)=\mathcal{F}(S,T)+\mathbf{F}(|S|,|T|).\]

A continuous map in the $\mathcal{F}$-topology
\[
\Phi:S^1\to\mathcal{Z}_{n-1}(M),
\qquad \Phi(0)=\Phi(2\pi)=0,
\]
is called a \textit{sweep-out} if it has no concentration
of mass and its image under the Almgren isomorphism
\[
F:\pi_1(\mathcal{Z}_{n-1}(M),\{0\})
\to H_n(M,\mathbb{Z})
\]
is the fundamental class $[M]$. The first condition means
\[
\lim_{\rho\to0}\sup_{\theta\in S^1,\;p\in M}
\|\Phi(\theta)\|(B_\rho(p))=0,
\]
and the second condition means $F([\Phi])=[M]$.
Let $\mathcal{P}$ denote the collection of all such
sweep-outs. The width of $M$ is then defined by
\begin{equation}\label{onewidth}
\omega(M):=\inf_{\Phi\in\mathcal{P}}
\sup_{\theta\in S^1}\mathbf{M}(\Phi(\theta)).
\end{equation}

It can be deduced from the isoperimetric inequality that the width $\omega(M)$ is always positive. Almgren \cite{almgren} established the existence of a stationary integral varifold realizing $\omega(M)$, and Pitts \cite{pitts} later proved that this varifold is actually realized by an embedded minimal hypersurface. Moreover, on a manifold equipped with a generic metric, this min-max minimal hypersurface is unstable and has Morse index one. 

This setting can be applied in a similar manner to find an unstable index-one minimal hypersurface in any integer homology class. Let $\mathcal{Z}_{n-1}(M, r)$ denote the space of all integral cycles that represent the same homology class $r\in H_{n-1}(M, \mathbb{Z})$ in $M$. This is the connected component containing any fixed cycle $S \in r$ of the space of all cycles. Naturally, $\mathcal{Z}_{n-1}(M, r)$ is homeomorphic to $\mathcal{Z}_{n-1}(M)$ via the translation map sending $T$ to $T-S$.

A continuous map
\[
\Psi:S^1\to\mathcal{Z}_{n-1}(M,r)
\]
is called an \textit{$r$-sweep-out} if it has no
concentration of mass and the loop
$\Phi(\theta)=\Psi(\theta)-\Psi(0)$ satisfies
$F([\Phi])=[M]$. Let $\mathcal{P}_r$ denote the
collection of all such $r$-sweep-outs. The \textit{min-max width of the class $r$} is defined as follows:
\[
\omega(r):=\inf_{\Psi\in\mathcal{P}_r}
\sup_{\theta\in S^1}\mathbf{M}(\Psi(\theta)).
\]

If we want to emphasize the dependence on the Riemannian metric $g$, we shall write $\omega_g(r)$ instead.

It follows from the works of Marques-Neves \cite{MNindex} that the following holds:
\begin{Th}\label{minmax}
    Let $(M^{n}, g)$, with $3\leq n\leq 7$, be a closed manifold and $r$ be a fixed primitive homology class. Then there exists a stationary integral varifold $V$ whose support is a closed embedded minimal hypersurface $\Sigma$ of index bounded by one, such that
    \[||V||(M)=\omega(r).\]
    
    Furthermore, we have the decomposition \[\Sigma=n_1\Sigma_1+\cdots+n_k\Sigma_k\] where each $\Sigma_i$ is a closed, connected, and embedded minimal hypersurface and $n_i\in \mathbb{Z}_{>0}$. Any two $\Sigma_i, \Sigma_j$ with $i\neq j$ are disjoint, and:
    \[\operatorname{Ind}(\Sigma)=\sum_{i=1}^{k}\operatorname{Ind}(\Sigma_i)\leq 1.\]
\end{Th}

By the definition of the stable norm (see Definition \ref{definitionstablenorm}), we have $\textbf{M}(\Phi(\theta))\geq S(r)$ for all $\theta\in S^1$ and $\Phi\in \mathcal{P}_r$. As a consequence \[\omega(r) \geq S(r), \quad \forall r\in H_{n-1}(M, \mathbb{Z}).\] 

A standard Lusternik-Schnirelmann theory-type argument shows that equality occurs if and only if $M$ is foliated by area-minimizing hypersurfaces in the class $r$. See Appendix \ref{appendix:a} for more details.

Because we are considering one-parameter families of hypersurfaces lying in the homology class $r$, the support of the min-max limit $\Sigma$ necessarily contains at least one closed hypersurface in $r$. This is made precise in the following proposition:
\begin{Prop}\label{spt}
    In the setup of Theorem \ref{minmax}, the support of the min-max minimal hypersurface $\Sigma$ contains a closed embedded minimal hypersurface (not necessarily connected) that represents the homology class $r$.
\end{Prop}
\begin{proof}
    The statement follows from examining the convergence of a min-max sequence in two different senses. Suppose that $\{|\Phi_i(x_i)|\}_i$ is a min-max sequence converging in the varifold sense to $|\Sigma|=\sum\limits_{i=1}^{l} m_i|\Sigma_i|$.
    
    It follows from the continuity of mass with respect to the varifold convergence that $\mathbf{M}(\Phi_i(x_i))$ is uniformly bounded. Therefore, the weak compactness theorem for integral currents \cite{simon} implies that there is a subsequence (still denoted by $\Phi_i(x_i)$) that converges in the flat topology to some integral current $T$. Varifold convergence also implies that the associated Radon measure $||\Phi_i(x_i)||$ converges to $||\Sigma||$ weakly. Therefore, the support of $T$ must be contained in the support of $\Sigma$, meaning
    \[\operatorname{spt}(T)\subset \operatorname{spt}(\Sigma)= \displaystyle\bigcup_{i=1}^{l} \Sigma_i.\]

    Since $\Phi_i(x_i)$ represents the same homology class $r$ for every $i$, the same is true for $T$. The proposition follows.
\end{proof}

\begin{Rem}
    A basic difficulty that arises when trying to produce an index-one minimal hypersurface in a prescribed homology class $r$ by this method is the following. It may happen that the minimal hypersurface in $r$ realizing the width is not connected and instead consists of several disjoint, connected components. In such a situation, one component may be stable and represent the homology class $r$, while another component may have index one but be homologically trivial.
    
    To illustrate this phenomenon, consider a flat torus to which a large sphere is attached (see Figure \ref{fig:torus} below). For any nontrivial class $r \in H_{n-1}(\mathbb{T}^n, \mathbb{Z})$, the minimal hypersurface that realizes $\omega(r)$ has two connected components: one is an area-minimizing hypersurface in $r$, and the other is an index-one great equator arising from the spherical part. Consequently, the one-parameter family min-max construction does not yield a connected, index-one minimal surface in any integer class. 
    \begin{figure}[H]
    \centering
    \includegraphics[scale = 0.4]{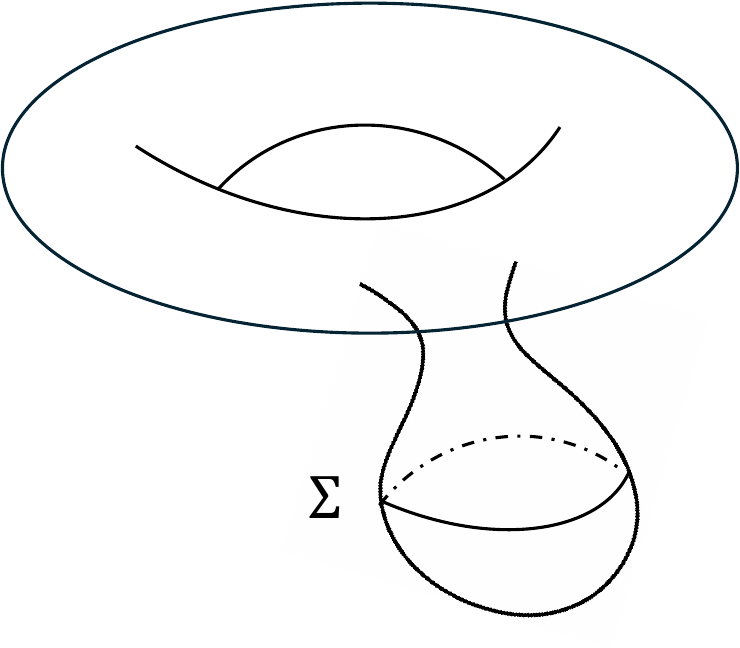}
    \caption{Torus with a sphere attached. The component $\Sigma$ represents an index-one equator arising from the spherical region.}
    \label{fig:torus}
    \end{figure} 
\end{Rem}
Nonetheless, when the quantity $\Delta W_r$ is sufficiently small, there is not enough excess area for an additional closed component. Consequently, the min-max minimal hypersurface in the class $r$ must consist of a single connected component representing $r$. 
\begin{Prop}
    Suppose that $3\leq n \leq 7$ and that $g$ is a Riemannian metric on $\mathbb{T}^n$. There exists a positive constant $C$, depending only on the metric $g$, with the following property: for every primitive homology class $r\in H_{n-1}(\mathbb{T}^n, \mathbb{Z})$ satisfying $\Delta W_r<C$, there exists a connected minimal hypersurface $\Sigma_r$ representing the class $r$ for which the min-max width $\omega(r)$ is achieved by $\Sigma_r$. 
\end{Prop}

\begin{proof}
    By the monotonicity formula for minimal submanifolds, there exists a constant $C=C(g)$ such that for every closed minimal hypersurface $\Sigma$ on $\mathbb{T}^n$, we have: 
    \[\mathbf{M}(\Sigma)>C.\]
    
    We will show that this constant satisfies the proposition. Suppose that $\Delta W_r=\omega(r)-S(r)<C$. By Theorem \ref{minmax}, there exist connected, pairwise disjoint  minimal hypersurfaces $\Sigma_1, \Sigma_2, \ldots, \Sigma_k$ and positive integers $m_1, m_2, \ldots, m_k$ such that 
    \[\omega(r)=m_1\mathbf{M}(\Sigma_1)+\cdots+m_k\mathbf{M}(\Sigma_k). \]
    
    Since $r\in H_{n-1}(\mathbb{T}^n, \mathbb{Z})$ is primitive, it follows from Proposition \ref{spt} and the fact that any two hypersurfaces representing independent homology classes intersect (see Appendix \ref{appendix:c}), at least one of the $\Sigma_i, 1\leq i\leq k$ must belong to the class $r$. Without loss of generality, assume that $\Sigma_1\in r$. From the definition of the stable norm, we have that $S(r)\leq \mathbf{M}(\Sigma_1)$. As a result:
    \begin{equation}
    \begin{split}
    \Delta W_r=\omega(r)-S(r) &=(\mathbf{M}(\Sigma_1)-S(r))+(m_1-1)\mathbf{M}(\Sigma_1)+\sum_{i=2}^{k}m_i\mathbf{M}(\Sigma_i)\\
    &\geq (m_1-1)\mathbf{M}(\Sigma_1)+\sum_{i=2}^{k}m_i\mathbf{M}(\Sigma_i).
    \end{split}
    \end{equation}
    
    By the choice of $C$, we have that $\mathbf{M}(\Sigma_i)>C$ for all $i=1,2, \ldots, k$. Together with the inequality $\Delta W_r<C$, it must happen that $k=m_1=1$ and we conclude that $\omega(r)=\mathbf{M}(\Sigma_1)$. The proof is complete. 
\end{proof}

It is therefore of interest to bound the quantity $\Delta W_r$ from above. Recall that $\omega(M)$ denotes the usual width of $M$ (or, in this context, the width of the trivial class). We have the following upper bound that holds for any Riemannian metric on any Riemannian manifold of any dimension. 
\begin{Prop}\label{upperbound}
    For any integer homology class $r\in H_{n-1}(M, \mathbb{Z})$, we have:
    \[0\leq \omega(r)-S(r) \leq \omega(M).\]
\end{Prop}

\begin{proof}
    For a given $\epsilon>0$, we can choose a sweep-out $\Phi: S^1\to \mathcal{Z}_{n-1}(M)$ such that \[\sup_{t\in S^1} \mathbf{M}(\Phi(t))\leq \omega(M)+\epsilon.\]
    Consider the map $\phi$ defined as 
    \[\phi(t)=\Phi(t)+T_r,\]
    where $T_r\in r$ is a fixed mass-minimizing integer current in $r$. Since $[\Phi(t)]=0 \in H_{n-1}(M, \mathbb{Z})$, we have $\phi(t)\in r$. It follows that $\phi$ is a sweep-out of $M$ by currents in $\mathcal{Z}_{n-1}(M, r)$. 
    Note that 
    \begin{equation}
    \begin{split}
    \textbf{M}(\phi(t)) & =\textbf{M}(\Phi(t)+T_r) \\
    & \leq \textbf{M}(\Phi(t))+\mathbf{M}(T_r) \\ &\leq \sup \textbf{M}(\Phi(t))+S(r)\\ & \leq \omega(M)+\epsilon+S(r)
    \end{split}
    \end{equation}
    for all $t\in [0,1]$. 
    Therefore:
    \[\omega(r) \leq \sup \textbf{M}(\phi(t))\leq \omega(M)+\epsilon+S(r).\]
    Sending $\epsilon$ to $0$, we establish the desired inequality.
\end{proof}

A similar argument yields the following more general inequality: 
\[\omega((k+1)r)-S((k+1)r)\leq \omega(kr)-S(kr)\]
where $kr \in H_{n-1}(M, \mathbb{Z})$ is an integer multiple of the class $r$. 

It follows that for any sequence $\{r_k\}_k$ of integer classes, the sequence 
\[\Delta W_{r_k}=\omega(r_k)-S(r_k)\]
is bounded and thus has a convergent subsequence. Theorem \ref{maintheorem1} states that if $\alpha$ is totally irrational and $\mathcal{M}_\alpha$ is a foliation, then for any sequence of integer classes $r_k\to \alpha$, we have $\Delta W_{r_k} \to 0$. Conversely, if there exists one sequence $r_k\to \alpha$ for which $\Delta W_{r_k}\to 0$, then the same conclusion holds for every such sequence. 

We expect that, more generally, the limit of $\Delta W_{r_k}$ as $r_k\to \alpha$ exists and is independent of the particular sequence, provided that $\alpha$ is totally irrational. In other words, we hope to replace the limit inferior in the definition of the Mather energy barrier with the regular limit.
In dimension two, the analogous result for $\mathbb{T}^2$ with geodesics was proved by Mather in \cite{mather}.

\subsection{A local min-max construction on manifolds with boundary}
We give a brief overview of a local min-max construction that is discussed in \cite{songdichotomy}, Appendix B. It is essentially contained in \cite{MNindex}. Let $N^{n}$ be a compact, connected, orientable Riemannian manifold. Suppose that \[\partial N=\partial_0N\cup \partial_1N\] 
where $\partial_0N$ and $\partial_1N$ (which can be empty) are two disjoint, closed components. 

We are interested in sweep-outs of $N$ that connect $\partial_0N$ to $\partial_1N$:
\[\Phi: [0,1]\to \mathcal{Z}_{n-1}(N;\mathbf{F},\mathbb{Z}), \quad \Phi(0)=\partial_0N,\text{ } \Phi(1)=\partial_1N.\]

Let $\Pi$ be the collection of all continuous sweep-outs $\Phi'$ that are homotopic to $\Phi$  in the flat topology, with fixed endpoints $\partial_0N$ and $\partial_1N$. The \textit{width} of $\Pi$ is defined as follows:
\[L(\Pi)=\inf_{\Phi\in \Pi} \sup_{t\in [0,1]}\mathbf{M}(\Phi(t)).\]

The width of $(N,g)$ is then defined as the infimum of $L(\Pi)$ over all possible homotopy classes of sweep-outs as above.
\[W_\partial(N)=\inf_{\Pi} L(\Pi).\]

 One should not confuse this width with the one-width $\omega_1(N,g)$ defined using relative cycles. 
 
The following min-max theorem is essentially contained in the work of Marques-Neves \cite{MNindex}. See also \cite{songdichotomy}, Appendix B, for a proof. 
\begin{Th}
    Let $3\leq n\leq 7$ and $N^{n}$ be a compact, connected, orientable manifold endowed with a bumpy metric $g$. Assume that $\partial N$ is a strictly stable minimal hypersurface. Then there exists a stationary integral varifold $V$ whose support is a smooth embedded minimal hypersurface $\Sigma\in N$ of index at most one, such that
    \[||V||(N) =W_\partial (N).\]
    Furthermore, at least one of the components of $\Sigma$ is contained in the interior of $N$.
\end{Th}

\subsection{A mountain-pass theorem for minimal hypersurfaces with prescribed boundary}
We recall the prescribed-boundary Almgren--Pitts theory developed by
Montezuma~\cite{rafael}. Its existence theorem applies to a fixed
relative homotopy class of sweep-outs; the endpoints need not be strictly
stable. The distinction between components that carry the prescribed
boundary and closed components will be important in Section~\ref{section9}.

Let $(M^n,g)$, $3\le n\le7$, be a compact, oriented Riemannian
manifold with strictly convex boundary, and let
$\gamma^{n-2}\subset\partial M$ be a smooth, closed, embedded, oriented
submanifold. Write
\[
\mathcal Z_{n-1}(M,\gamma)
:=\{T\in\mathbf I_{n-1}(M):
             \operatorname{spt}(\partial T)\subset\gamma\}.
\]
Fix integral currents $T_0,T_1\in\mathcal Z_{n-1}(M,\gamma)$ with
$\partial T_0=\partial T_1=\llbracket\gamma\rrbracket$, and let
$P_0:\partial I\to\mathcal Z_{n-1}(M,\gamma)$ be given by
$P_0(0)=T_0$ and $P_0(1)=T_1$. We consider a one-parameter discrete
homotopy class
\[
\Pi\in\Pi_1^{\#}\bigl(\mathcal Z_{n-1}(M,\gamma),P_0\bigr)
\]
in the sense of~\cite[Section~2]{rafael}. If
$\{\varphi_i\}\in\Pi$ is a discrete homotopy sequence with
$\varphi_i:I(1,k_i)_0\to\mathcal Z_{n-1}(M,\gamma)$, its width is
\[
\mathbf L(\Pi)
:=\inf_{\{\varphi_i\}\in\Pi}
  \limsup_{i\to\infty}\max_{x\in I(1,k_i)_0}
                         \mathbf M(\varphi_i(x)).
\]
Mass-continuous paths with the same prescribed boundary give such
classes through discretization. The interpolation and discretization
results in~\cite[Section~3]{rafael} provide the connection between the
continuous families used below and this discrete formulation.

\begin{Th}\label{camillo}
In the setting above, suppose
\[
\mathbf L(\Pi)>
\max\{\mathbf M(T_0),\mathbf M(T_1)\}.
\]
Then there exists an integral $(n-1)$-varifold $V$ of mass
$\|V\|(M)=\mathbf L(\Pi)$ whose support is the closure of a smooth
embedded minimal hypersurface $\Gamma$ in $M$. If
$\Gamma_1,\ldots,\Gamma_N$ are its connected components, then:
\begin{enumerate}
\item Each $\Gamma_a$ is smooth up to its boundary. In the dimension
range $3\le n\le7$ there are no interior singularities.
\item Each $\partial\Gamma_a$ is either empty or a union of connected
components of $\gamma$, and
\[
\bigcup_{a=1}^N\partial\Gamma_a=\gamma.
\]
\item There exist positive integers $m_1,\ldots,m_N$ such that
\[
V=\sum_{a=1}^N m_a|\Gamma_a|,
\qquad m_a=1\quad\text{if }\partial\Gamma_a\ne\varnothing.
\]
\item The varifold $V$ belongs to the critical set of a min--max
sequence in $\Pi$, and
\[
\mathbf L(\Pi)=\|V\|(M)
              =\sum_{a=1}^N m_a\mathbf M(\Gamma_a).
\]
\end{enumerate}
The varifold is stationary for interior variations and satisfies the
prescribed-boundary stationarity condition along $\gamma$.
\end{Th}

\begin{Rem}\label{rem:prescribed-use}
In Section~\ref{section9} we apply Theorem~\ref{camillo} with
\[
(M,g,\gamma,T_0,T_1,\Pi)
=(M_k,g_k,\sigma^{R_k},
  \Gamma_0^{R_k},\Gamma_1^{R_k},\Pi_k).
\]
The strict width inequality used there comes from
Proposition~\ref{2plateau}, together with the prescribed-boundary
interpolation needed to identify its continuous competitors with the
discrete class $\Pi_k$. In particular, a selected realization has
the varifold decomposition
\[
V_k=\sum_{\partial\Gamma_{k,a}\ne\varnothing}|\Gamma_{k,a}|
          +\sum_{\partial Z_{k,b}=\varnothing}q_{k,b}|Z_{k,b}|,
\qquad q_{k,b}\in\mathbb N,
\]
where the $Z_{k,b}$ are closed minimal hypersurfaces in
$\operatorname{Int}M_k$, and
\[
L_k=\|V_k\|_{g_k}(M_k).
\]
\end{Rem}

\section{An interpolation lemma}\label{section3}
In this section, we recall and reformulate an interpolation theorem due to Marques-Neves in \cite{willmore}. Its primary purpose is to approximate a discrete map of sufficiently small fineness (a notion defined below) by a continuous map in the mass norm. This is especially useful when one wishes to construct a one-parameter family of currents with controlled mass connecting two given currents $T_1, T_2$. In particular, if $T_1$ and $T_2$ are mass-minimizing currents that are close in the flat topology, we can interpolate from one to the other continuously by a family of currents of similar masses. 

This result will be applied in the proof of the implication $(3)\implies(1)$ in Theorem \ref{maintheorem1}, where we need to construct near-optimal sweep-outs by integral currents in the class $r_k$ to estimate the width $\omega(r_k)$.

We begin by recalling some additional notation that was introduced in \cite{willmore}. For integers $m,j \geq 1$, let $I(m,j)$ be the cell complex on the unit $m$-dimensional cube $I^m$:
\[I(m,j)=I(1,j)\otimes I(1, j)\otimes\cdots\otimes I(1,j) \quad (\text{with } m \text{ factors}).\]
Here $I(1,j)$ is the cube complex on $I^1=[0,1]$ whose $1$-cells and $0$-cells are given by the subdivisions
\[[0, 3^{-j}], \ldots, [1-3^{-j}, 1] \quad \text{and} \quad [0], [3^{-j}],\ldots, [1].\]

The cube complex $X(j)$ is the union of all the cells of $I(m,j)$ whose support is contained in some cell of $X$, and $X(j)_0$ is the set of all $0$-cells (vertices) in $X(j)$. 

Given a discrete map $\phi: X(j)_0\to \mathcal{Z}_{n-1}(M)$, we define the \textit{fineness} of $\phi$ to be 
\[\mathbf{f}(\phi)=\sup \{\mathbf{M}(\phi(x)-\phi(y)): x, y \text{ adjacent vertices in } X(j)_0\}.\]
\begin{Th}\label{discretetocontinuous}
    There exist positive constants $C_0$ and $\delta_0$ depending only on $(M^n, g)$ and $m$ such that for every map
    \[\psi: I(m,j)_0 \to \mathcal{Z}_{n-1}(M)\]
    with $\mathbf{f}(\psi)<\delta_0$, we can find a continuous map in the mass norm 
    \[\Psi: I^m\to \mathcal{Z}_{n-1}(M; \mathbf{M})\]
    such that
    \begin{enumerate}
        \item $\Psi(x)=\psi(x)$ for all $x\in I(m,j)_0$;
        \item for every $\alpha\in I(m, j)_p$, we have the following:
        \[\sup\{\mathbf{M}(\Psi(x)-\Psi(y)): x, y \in \alpha \} \leq C_0\mathbf{f}(\psi).  \]
    \end{enumerate}
\end{Th}
Specializing this to the unit interval $I=[0,1]$, we have the following:
\begin{Cor}\label{interpolation}
    For every $0<\delta<\delta_0$ and for any $T_0, T_1 \in \mathcal{Z}_{n-1}(M)$ with $\mathbf{M}(T_0-T_1)<\delta$, we can find a continuous map in the mass norm
    \[\Psi: [0,1]\to \mathcal{Z}_{n-1}(M, \mathbf{M})\]
    such that
    \begin{enumerate}
        \item $\Psi(0)=T_0, \Psi(1)=T_1$;
        \item for all $t\in [0,1]$ we have
        \[\mathbf{M}(\Psi(t))\leq \mathbf{M}(T_0)+  C_0\delta. \]
    \end{enumerate}
\end{Cor}
\begin{proof}
    Consider the cell complex $I(1, 0)$ on $I^1$ with two vertices $[0]$ and $[1]$. Let $\psi(0)=T_0$ and $\psi(1)=T_1$. From the definition of fineness, we have 
    \[\mathbf{f}(\psi)=\mathbf{M}(T_1-T_0)<\delta_0. \]
    It follows from Theorem \ref{discretetocontinuous} that we can extend $\psi$ to a mass-continuous map \[\Psi: [0,1]\to \mathcal{Z}_{n-1}(M; \mathbf{M})\] so that 
    \[|\mathbf{M}(\Psi(t))-\mathbf{M}(\Psi(0))| \leq |\mathbf{M}(\Psi(t)-\Psi(0))|\leq C_0\delta\]
    for all $t\in [0,1]$. Since $\Psi(0)=T_0$, we are done. 
\end{proof}
\section{Aubry-Mather theory for area-minimizing hypersurfaces on tori}\label{section4}
We now describe the setting in detail. Consider the $n$-dimensional torus $\mathbb{T}^n=\mathbb{R}^n/\mathbb{Z}^n$ equipped with a Riemannian metric $g$ and fix a Riemannian universal covering $\pi: \mathbb{R}^n\to \mathbb{T}^n$. The metric $g$ lifts to a $\mathbb{Z}^n$-periodic metric (still denoted by $g$) on $\mathbb{R}^n$. Under the Hurewicz homomorphism, we have the identification $\pi_1(\mathbb{T}^n)\cong H_1(\mathbb{T}^n, \mathbb{Z})$, which acts as the group of deck transformations of $\pi$. The deck transformation associated with the element $k\in\pi_1(\mathbb{T}^n)$ is denoted by $\tau_k$ and we have $\tau_k(x)=x+k$ for all $x\in \mathbb{R}^n$.

A hypersurface $S$ on $\mathbb{R}^n$ is said to have the \textit{Birkhoff property} or is called \textit{self-intersection free} if its projection to $\mathbb{T}^n$ has no transverse self-intersection. This is equivalent to the fact that for any integer translation $\tau_k$ of $\mathbb{R}^n$, either $S$ is $\tau_k$-invariant or $S$ and $\tau_k(S)$ are disjoint.\\

We are interested in understanding the structure of the following set:
\[
  \mathcal{M}: = \left\{ S=\partial \llbracket E\rrbracket\ \middle\vert \begin{array}{l}
    \text{S is self-intersection free} \\
    \text{and area-minimizing and has connected support}
  \end{array}\right\}
.\]

Here $E \subset \mathbb{R}^n$ is an open set of locally finite perimeter and $S$ is area-minimizing if for all bounded open sets $U\subset \mathbb{R}^n$ we have 
\[\mathbf{M}_U(S)\leq \mathbf{M}_U(S+T)\]
for all $T\in \mathcal{R}_{n-1}(U) \text{ with } \operatorname{spt}(T)\subset U \text{ and }  \partial T=0$. Here, we always choose $E$ consistently with the orientation of $S$.

There is the following natural partial order on $\mathcal{M}$: 
\[\Sigma_0=\partial E_0<\Sigma_1=\partial E_1 \iff E_0\subset E_1.\]

For any $S=\partial \llbracket E\rrbracket \in \mathcal{M}$, it was proved by Moser in \cite{moser} in the non-parametric case and Auer in \cite{auer} in the parametric analogue that there exists a homological direction $\alpha \in H_{n-1}(\mathbb{T}^n, \mathbb{R}) \cong \mathbb{R}^n$, unique up to scalar multiplication, such that the following holds:
\begin{equation}\label{birkoff1}
    \tau_k (E) \subset E \text{ if } k\cdot \alpha<0,
\end{equation}
\begin{equation}\label{birkoff2}
    E \subset \tau_k(E) \text{ if } k\cdot \alpha >0 
\end{equation}
for any $k\in \pi_1(\mathbb{T}^n)=\mathbb{Z}^n$. Let $\mathcal{M}_\alpha$ be the subset of $\mathcal{M}$ consisting of elements having direction $\alpha$. Their results showed even more: an area-minimizing hypersurface in $\mathcal{M}_\alpha$ always stays within a universally bounded distance from an affine hyperplane. 

Since the metric is periodic, there are positive constants $\lambda$ and $\Lambda$ depending only on $g$ such that \[\lambda g_{flat}\leq g \leq \Lambda g_{flat}.\] 
\begin{Th}\label{slab}
    There exists a constant $C>0$ depending only on $\lambda$ and $\Lambda$ such that for every area-minimizing hypersurface $S=\partial E\in \mathcal{M}_\alpha$, there exists a positive constant $K=K(S)$ such that
    \[\operatorname{spt}(S)\subset \left\{x\in \mathbb{R}^n: \text{ } x\cdot \frac{\alpha}{S(\alpha)} \in [K, K+C]\right\}.\]
\end{Th}

We now briefly discuss the existence of area-minimizing hypersurfaces in each $\mathcal{M}_\alpha$. On manifolds with non-trivial homology groups, one can produce minimal submanifolds by minimizing the area functional within a fixed homology class. The existence and regularity for such minimizers were established by Federer in \cite{federerbook}.
\begin{Th}\label{homology}
    Let $(M^{n}, g)$ be a closed Riemannian manifold, and let $\alpha \in H_{n-1}(M, \mathbb{Z})$ be a non-trivial homology class. Then there exists a closed minimal submanifold $\Sigma$ smooth away from a singular set of dimension at most $n-8$ that minimizes the area among all representatives of $\alpha$. Moreover, the support of $\Sigma$ consists of disjoint embedded minimal submanifolds, possibly occurring with multiplicities.
\end{Th}
In this paper, unless otherwise stated, we always work under the classical assumption $3\leq n\leq 7$, so that area-minimizing hypersurfaces are smooth. Under the equivalence between singular homology and the homology of a chain complex of normal currents, see Theorem 5.11 of \cite{normalcurrent}, Federer also proved that an area-minimizing hypersurface $\Sigma$ minimizes the mass among all closed, normal currents representing $\alpha$. 

\begin{Def}\label{definitionstablenorm}
Following \cite{federerbook}, the \textit{stable norm} $S(\alpha)$ of a real homology class $\alpha\in H_{m}(M,\mathbb{R})$ is defined as the infimum of the volume of all real Lipschitz cycles 

\[c=\sum r_i\sigma_i, \quad r_i\in \mathbb{R}, \text{ }\sigma_i: \Delta_i\to M\]
representing $\alpha$. It was shown in \cite{stablenorm} that
\[S(\alpha)=\min \{\mathbf{M}(T)| T\in \mathbf{N}_m(M), \partial T=0, [T]=\alpha\},\]
where $\mathbf{N}_m(M)$ is the space of all normal currents in $M$. Recall that a current $T$ is normal if both $T$ and $\partial T$ have finite mass.
\end{Def}

In the context of the $n$-torus, we have that $H_{n-1}(\mathbb{T}^n, \mathbb{Z})=\mathbb{Z}^n$. Furthermore, it was proved by Bangert and Auer in \cite{bangertauer} that a homologically area-minimizing hypersurface $\Sigma$ in a primitive class (defined as not being an integer multiple of another integer class) is connected. Specifically, they proved that the stable norm on a torus is always strictly convex:
\begin{equation}\label{convex}
    S(\alpha+\beta)<S(\alpha)+S(\beta) 
\end{equation}
for all $ \alpha, \beta \in H_{n-1}(\mathbb{T}^n, \mathbb{R})$ that are linearly independent. If a homologically area-minimizing hypersurface $\Sigma$ in a primitive class $\gamma$ decomposes as $\Sigma_1+\Sigma_2$, then $\Sigma_1$ and $\Sigma_2$ minimize area in their respective homology classes. Let $[\Sigma_1]=\alpha$ and $[\Sigma_2]=\beta$. We have: 
\[S(\alpha+\beta)=\textbf{M}(\Sigma_1+\Sigma_2)=\textbf{M}(\Sigma_1)+\textbf{M}(\Sigma_2)=S(\alpha)+S(\beta),\] 
contradicting (\ref{convex}). A more direct argument showing the connectedness of $\Sigma$ is given in Appendix \ref{appendix:c}.

Therefore, for each $\alpha \in H_{n-1}(\mathbb{T}^n, \mathbb{Z})$, we can lift $\Sigma$ to the universal cover $\mathbb{R}^n$ to obtain a connected hypersurface $\widetilde{\Sigma}$ in the direction $\alpha$:
\begin{Def}
    Assume that $p: \widehat{M}^n\to M^n$ is a normal Riemannian covering and $T$ is an $m$-current of locally finite mass on $M^n$. We can define the lift $p^\#T$ of $T$ as a current of locally finite mass on the cover as follows: for every open set $U\subset \widehat{M}^n$ such that $p|_U$ is injective and for every compactly supported $m$-form $\omega$ on $U$, set
    \[p^\#T(\omega)=\left((p|_U)^{-1}\right)_\#(T)(\omega),\]
    where $(p|_U)^{-1}$ denotes the local inverse of $p$, and $_\#$ is the standard push-forward operation of a current. We can extend this definition by a partition of unity. 
\end{Def}

It follows immediately from the definition that $\widetilde{\Sigma}$ is area-minimizing on every compact set that projects injectively to the torus. In fact, using a result of Federer \cite{federerbook}, one can show that $\widetilde{\Sigma}$ is area-minimizing on every compact subset of $\mathbb{R}^n$:
\begin{Prop}
    Let $U\subset \mathbb{R}^n$ be a bounded open set and $S$ be a current with support contained inside $U$ and $\partial S=0$. Then \[\mathbf{M}_U(\widetilde{\Sigma})\leq \mathbf{M}_U(\widetilde{\Sigma}+S).\] 
\end{Prop}
\begin{proof}
    Since $H_{n-1}(\mathbb{R}^n, \mathbb{Z})=0$, there exists $W$ with support contained inside $U$ and $S=\partial W$. 
    
    Recall that $\widetilde{\Sigma}$ is a lift of a homologically area-minimizing hypersurface $\Sigma$ on $\mathbb{T}^n$. Let $H\subset \mathbb{R}^n$ be a fundamental domain. Since we can choose a sufficiently large integer $k$ such that $U$ is contained inside the union of $k$ isometric copies of $H$, we can assume that $U$ is such a union. It follows from a standard result of Federer \cite{federerbook} that $k\Sigma$ is homologically area-minimizing in the class $k[\Sigma]$. Therefore
    \[\textbf{M}_U(\widetilde{\Sigma})=\textbf{M}(k\Sigma)\leq \textbf{M}(k\Sigma+\partial \pi (W))=\textbf{M}_U(\widetilde{\Sigma}+\partial W)=\textbf{M}_U(\widetilde{\Sigma}+S).\]
    
    This completes the proof.
\end{proof}

The remaining fact that the homological direction of $\Tilde{\Sigma}$ is $\alpha$ up to rescaling follows from the following claim: 

\begin{Lemma} \label{lift}
(\text{\cite{stablenorm}, Lemma 5.4}).
Suppose that $\alpha\in H_{n-1}(\mathbb{T}^n, \mathbb{R})$ is a real homology class, and let $T\in \mathbf{N}_{n-1}(\mathbb{T}^n)$ be a closed normal current representing $\alpha$. Let $\pi: \mathbb{R}^n\to \mathbb{T}^n$ be the universal covering map realizing $H_{n-1}(\mathbb{T}^n, \mathbb{R})\cong \mathbb{R}^n$. Then there exists a function $f$ of locally bounded variation on $\mathbb{R}^n$ such that the following hold:
\begin{enumerate}
    \item for any $(n-1)$-form $\omega$ on $\mathbb{R}^n$ with compact support, we have:
    \[T(\omega)=\int_{\mathbb{R}^n} fd\omega,\]
    \item for all $x\in \mathbb{R}^n$ and $k\in \pi_1(\mathbb{T}^n)\cong \mathbb{Z}^n$, we have:
    \[f(\tau_k x)=f(x)-k\cdot \alpha.\] 
\end{enumerate}   
In particular, every leaf $\Tilde{T}$ of $\pi^\#T$ has the Birkhoff property and satisfies 
    \[ \Tilde{T}<\tau_k \Tilde{T} \quad \text{if and only if} \quad k\cdot \alpha>0,\]
   \[ \Tilde{T}=\tau_k \Tilde{T} \quad \text{if and only if} \quad k\cdot \alpha=0.\]
\end{Lemma}

This shows that $\Tilde{\Sigma}\in \mathcal{M}_\alpha$. In fact, it belongs to the subset consisting of the most periodic elements in the hierarchy described below. 

On the other hand, when the class $\alpha$ is irrational (see the definition below), we can take a sequence of integer classes $\{\alpha_i\}_i$ converging to $\alpha$ and a sequence of corresponding periodic minimizers $\Sigma_i \in \mathcal{M}(\alpha_i)$ that intersect a fixed compact set (this is possible by considering the integer translations of a fixed minimizer). The following standard compactness theorem for area-minimizing hypersurfaces shows that we can take the limit of the sequence $\{\Sigma_i\}_i$ to obtain an element in $\mathcal{M}_\alpha$:

\begin{Th}\label{compactness}
    For $3\leq n\leq 7$, the set of minimizing hypersurfaces $\mathcal{M}$ is compact with respect to the $C^k$ convergence on compact sets, for every $k$. Equivalently, suppose that $\Sigma_i\in \mathcal{M}_{\alpha_i}$ is a sequence of such hypersurfaces, all of which intersect a fixed bounded set $K\subset \mathbb{R}^n$. Then there exists a subsequence (still denoted by $\{\Sigma_i\}_i$) that converges to a minimizing hypersurface $\Sigma\in \mathcal{M}$. 
    
Furthermore, the homological vectors respect this convergence: $\Sigma\in \mathcal{M}_\alpha$ where $\alpha$ is the limit of the sequence $\{\alpha_i\}_i$.
\end{Th}

In particular, $\mathcal{M}_\alpha\neq \emptyset$ for all $\alpha$ and we have the following decomposition:
\[\mathcal{M}=\bigsqcup\limits_{S(\alpha)=1} \mathcal{M}_\alpha.\]

The detailed structure of each $\mathcal{M}_\alpha$ depends heavily on the \textit{rank} of $\alpha$, defined as follows:

\begin{Def}\label{totallyirrational}
   
    For each $\alpha\in H_{n-1}(\mathbb{T}^n, \mathbb{R})$, let $K(\alpha)$ be the set of all $k \in H_1(\mathbb{T}^n, \mathbb{Z})$ such that $I(\alpha, k)=0$, where \[I: H_{n-1}(\mathbb{T}^n, \mathbb{R})\times H_1(\mathbb{T}^n, \mathbb{R})\rightarrow \mathbb{R}\] is the homology intersection pairing: \[I(\alpha, k)=\operatorname{PD}(\alpha)(k),\] with $\operatorname{PD}(\alpha)\in H^1(\mathbb{T}^n, \mathbb{R})$ being the Poincar\'e dual of $\alpha$.
    For the torus, this is just the dot product. 
    
    Let $V(\alpha)$ be the smallest subspace of $H_{n-1}(\mathbb{T}^n, \mathbb{R})$ that is spanned by the integer classes and contains $\alpha$. As in \cite{stablenorm}, this is equivalent to
    \[V(\alpha)=\{\beta\in H_{n-1}(\mathbb{T}^n, \mathbb{R}): \text{ } I(\beta, k)=0 \text{ for all } k\in K(\alpha)\}.\]

   We then define $\operatorname{rank}_\mathbb{Q}(\alpha):=\operatorname{dim} V(\alpha)$. If $\operatorname{rank}_\mathbb{Q}(\alpha)=1$, we say that $\alpha$ is \textit{rational}. Otherwise, $\alpha$ is called \textit{irrational}. In the extreme case where $\operatorname{rank}_\mathbb{Q}(\alpha)=n$, $\alpha$ is said to be \textit{rationally independent} or \textit{totally irrational}.\\
   
   Note that $\alpha$ is rational if and only if it is a scalar multiple of a class in $H_{n-1}(\mathbb{T}^n, \mathbb{Z})$.
\end{Def}

\begin{Def}
    A collection of minimizers $F$ with connected supports is called a \textit{lamination} if:
\begin{enumerate}
    \item $\Sigma_1\in F$ and $\Sigma_2\in F$ implies $\Sigma_1\cap \Sigma_2=\emptyset$ or $\Sigma_1=\Sigma_2$;
    \item $\sqcup_{\Sigma\in F} \Sigma$ is a nonempty, closed subset of $\mathbb{R}^n$.
\end{enumerate}

The lamination $F$ is called a \textit{foliation} if \[\bigsqcup_{\Sigma\in F}\Sigma=\mathbb{R}^n.\]
The connected components of $\mathbb{R}^n\setminus \cup_{\Sigma\in F}\Sigma$ are called the \textit{gaps} of $F$. Each gap of $F$ is an open set of $\mathbb{R}^n$, and each lamination contains at most countably many gaps.
\end{Def}
For example, for every area-minimizing hypersurface $S\in \mathcal{M}$, the closure of the set $\{\tau_k(S): k\in \pi_1(\mathbb{T}^n)\cong \mathbb{Z}^n\}$ is a lamination of $\mathbb{R}^n$.
\\

The following structure theorem is due to Bangert in \cite{bangert1} and \cite{uniquenessbangert} (for the non-parametric case) and Junginger-Gestrich (\cite{hannes}) (in the parametric case).
\begin{Th}\label{structure}
    For every minimizing hypersurface $S\in \mathcal{M}$, there exist a positive integer $t$ and a collection of homology classes $\{\alpha_1, \alpha_2, \ldots, \alpha_t\}$ canonically associated to $S$, defined as follows:
    \begin{enumerate}
        \item Set $\alpha_1=\alpha$ and $\Gamma_1=\mathbb{Z}^n$.
        \item Define inductively the sub-lattices $\Gamma_{s+1}=\alpha_s^{\perp}\cap \Gamma_s$, provided that this intersection is not empty. If $S$ is not $\Gamma_{s+1}$-periodic, then there exists $\alpha_{s+1}\in span (\Gamma_{s+1})$ such that for all $k\in \Gamma_{s+1}$, properties (\ref{birkoff1}) and (\ref{birkoff2}) hold. If $S$ is $\Gamma_{s+1}$-periodic, this procedure terminates. 
    \end{enumerate}
    
    The integer $t=t(S)$ and the set of classes $\alpha_1, \ldots, \alpha_t$ (unique up to multiplying by a scalar) are invariants of $S$, and we write
    \[\mathcal{M}(\alpha_1, \ldots, \alpha_t):=\{S\in \mathcal{M} : \text{ } t(S)=t, \alpha_i(S)=\alpha_i, \forall 1\leq i\leq t\}.\]
 
    Furthermore, for every such $\alpha_1, \alpha_2, \ldots, \alpha_t$, the following union 
    \[\mathcal{M}(\alpha_1)\cup \mathcal{M}(\alpha_1, \alpha_2)\cup \cdots \cup \mathcal{M}(\alpha_1, \alpha_2, \ldots, \alpha_t) \subset \mathcal{M}_{\alpha_1}\]
    is pairwise disjoint and totally ordered (they form a lamination of $\mathbb{R}^n$). In particular, if $\alpha$ is totally irrational, $\mathcal{M}_\alpha$ is totally ordered. 
\end{Th}

\begin{figure}[H]
    \centering
    \includegraphics[scale = 0.4]{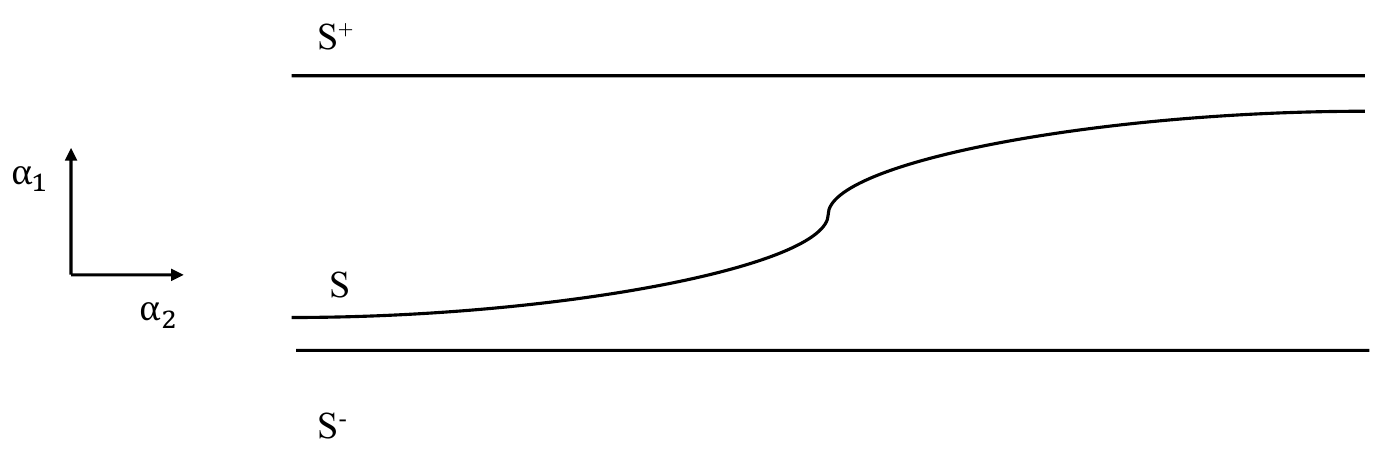}
    \caption{When there is a gap bounded by $S^-, S^+ \in \mathcal{M}(\alpha_1)$ that contains no other element in $\mathcal{M}(\alpha_1)$, for any $\alpha_2\in \alpha_1^\perp\cap \mathbb{Z}^n$, there exists a heteroclinic minimizer $S\in \mathcal{M}(\alpha_1, \alpha_2)$ inside the gap. $S$ is periodic in integral direction orthogonal to both $\alpha_1$ and $\alpha_2$, asymptotic to $S^-$ in the negative $\alpha_2$ direction and asymptotic to $S^+$ in the positive $\alpha_2$ direction.}
    \label{fig:heteroclinic}
\end{figure}

\begin{Rem}
The area-minimizing hypersurfaces in $\mathcal{M}(\alpha_1, \ldots, \alpha_t)$ with $t>1$ are called \textit{heteroclinic minimizers}. They are constructed inside the gaps formed by two consecutive elements in $\mathcal{M}(\alpha_1, \ldots, \alpha_{t-1})$ (see Figure $2$ above). 

It is helpful to keep in mind that the elements in $\mathcal{M}(\alpha_1, \ldots, \alpha_{t-1})$ have higher periodicity than the elements in $\mathcal{M}(\alpha_1, \ldots, \alpha_t)$, with $\mathcal{M}(\alpha)$ consisting of the most periodic element in $\mathcal{M}_\alpha$. When $\alpha$ is an integer class, $\mathcal{M}(\alpha)$ consists of all the lifts of all closed area-minimizing hypersurfaces in $\alpha$. If $\alpha$ is totally irrational, we have $\mathcal{M}(\alpha)=\mathcal{M}_\alpha$. In other cases, generically, the inclusion is strict.
\end{Rem}

\section{Laminations in an irrational class}\label{section5}

In this section, we study area-minimizing hypersurfaces with irrational directions. We begin by establishing the totally ordered property for $\mathcal{M}_\alpha$, in the case where $\alpha$ is irrational. This result is essentially contained in \cite{auer}. The analogue in the non-parametric setting for an elliptic integral was proved earlier by Bangert in \cite{uniquenessbangert}. 

Since the metric $g$ is invariant under integer translations, the actions of elements in $\mathbb{Z}^n$ preserve the area-minimizing property. The fact that the homological vector is preserved is immediate. Therefore, $\mathbb{Z}^n$ acts on each set $\mathcal{M}_\alpha$.  When the homology class $\alpha$ is irrational, we denote by $\mathcal{M}_\alpha^{rec}$ the subset of $\mathcal{M}_\alpha$ consisting of minimizing hypersurfaces that can be approximated by their own translates. More precisely, we say that an element $\Sigma\in \mathcal{M}_\alpha$ can be approximated from above if 
\begin{equation}\label{recurrent1}
    \Sigma=\inf_{k\in \mathbb{Z}^n} \{\tau_{k} \Sigma: \text{ } I(k, \alpha)>0 \}
\end{equation}
and from below if 
\begin{equation}\label{recurrent2}
\Sigma=\sup_{k\in \mathbb{Z}^n} \{\tau_{k} \Sigma: \text{ } I(k, \alpha)<0\}.
\end{equation} 

We now show that the set of recurrent minimizers is totally ordered:
\begin{Prop}\label{lamination}
    For irrational $\alpha$, any two distinct elements in $\mathcal{M}_\alpha^{rec}$ are disjoint. 
\end{Prop}

We first introduce some notation. For $p\in \mathbb{R}^n$ and $r>0$, let $Z(p,r)$ be the cylinder with axis $p+\mathbb{R} \alpha$ and radius $r$. Let $P=\alpha^\perp$ be the hyperplane orthogonal to $\alpha$. For every ball $B(q,r)\subset \mathbb{R}^n$, let $B_P(q,r)=B(q,r)\cap P$. 

Denote by $\mathcal{R}_{n-1}^{loc}(\mathbb{R}^n)$ the space of integral rectifiable currents of locally finite mass on $\mathbb{R}^n$.
For $S\in \mathcal{R}^{loc}_{n-1}(\mathbb{R}^n)$ and $W\subset \mathbb{R}^n$ a bounded open set, define the following:
\begin{equation}
\begin{split}
    D(S,W)&:=\sup\{\mathbf{M}(S\llcorner W)-\mathbf{M}(R): \text{ } R\in \mathcal{R}^{loc}_{n-1}(\mathbb{R}^n), \text{ }\partial R=\partial(S\llcorner W) \} \\
    &=\mathbf{M}(S\llcorner W)-\inf\{\mathbf{M}(R\llcorner W): \text{ }R\in \mathcal{R}^{loc}_{n-1}(\mathbb{R}^n), \text{ }\partial R=\partial(S\llcorner W) \},
\end{split}
\end{equation}
which measures the difference between the mass of $S\llcorner W$ and the mass of the Plateau solution filling $\partial(S\llcorner W)$. It follows from a standard geometric measure theory argument that the function \[S\rightarrow D(S,W)\] is mass lower-semicontinuous with respect to the flat topology on $\mathcal{R}^{loc}_{n-1}(\mathbb{R}^n)$.\\

This quantity is monotonically increasing in the following sense:
\begin{Lemma}\label{monotone}
    If $W_1\subseteq W_2$, then 
    \[D(S,W_1) \leq D(S, W_2).\]
\end{Lemma}
The proof follows by modifying every competitor of $S$ inside $W_1$ to get a competitor of $S$ inside $W_2$. We omit it here.\\

Fix $S_1=\partial E_1$ and $S_2=\partial E_2$ in $\mathcal{M}_\alpha^{rec}$. Assume, for contradiction, that $S_1\neq S_2$ and they intersect. We can assume further that both of them can be approximated from below as follows. Since they are recurrent, there exist two sequences of recurrent minimizers $S_1^i\to S_1$ and $S_2^i\to S_2$ such that every $S_1^i$ and $S_2^i$ can be approximated from below. Since $S_1\neq S_2$ and they intersect, it follows that for sufficiently large $i$, we have $S_1^i\neq S_2^i$ and $S_1^i\cap S_2^i\neq \emptyset$. We can thus work with $S_1^i$ and $S_2^i$ instead of $S_1$ and $S_2$.

Define the following:
\[S^+=\max(S_1, S_2)=\partial(E_1\cup E_2),\]
\[S^-=\min(S_1, S_2)=\partial (E_1\cap E_2).\]

\begin{Lemma}\label{boundinball}
    There exist $\epsilon>0$ and $\rho >0$ such that for all $p\in P$ and $r\geq \rho$, we have 
    \[D(S^+, B(p,r))\geq \epsilon \text{ and }D(S^-, B(p,r))\geq \epsilon. \]
\end{Lemma}

\begin{proof}
    We prove the statement for $S^+$ by contradiction; the one for $S^-$ is completely analogous. Assume the contrary. Then we can find a sequence of points $p_i\in P$ and radii $r_i\to \infty$ such that \[D(S^+, B(p_i,r_i))\to 0.\] 
    
    There exists a corresponding sequence of integer translations $\tau_{k_i}\in \pi_1(\mathbb{T}^n)$ such that 
    \[q_i=\tau_{k_i}(p_i)=p_i+k_i\in [0,1]^n.\]
    
    This choice of translations ensures that the distance from $q_i$ to the plane $P$ is uniformly bounded. As a result, for $j=1,2$, elements of the sequence $\{\tau_{k_i}S_j\}_i$ always lie within a bounded distance from $P$. By the compactness property of the minimizers (Theorem \ref{compactness}), we have $\tau_{k_i}S_j\to \widetilde{S}_j=\partial \widetilde{E}_j$ for $j=1,2$. This means that 
    \[\tau_{k_i}S^+\to \widetilde{S}^+=\partial(\widetilde{E}_1\cup \widetilde{E}_2).\] 
    
    Therefore
    \[\lim_{i\to \infty}D(\tau_{k_i}S^+, B(q_i, r_i))=\lim_{i\to\infty} D(S^+, B(p_i, r_i))=0.\]
    
    For every $r>0$, we can choose $i$ sufficiently large so that $B(0, r)\subset B(q_i, r_i)$. It follows from the monotonicity lemma \ref{monotone} that \[D(\tau_{k_i}S^+, B(0,r))\to 0.\]
    From the lower-semicontinuity of the mass, we get $D(\widetilde{S}^+, B(0,r))=0$ for all $r$. Therefore $\widetilde{S}^+$ is area-minimizing. 

    Note that $\widetilde{E}_1\subset \widetilde{E}_1\cup \widetilde{E}_2$. Since both $\widetilde{S}_1$ and $\widetilde{S}^+$ are area-minimizing, it follows from the maximum principle (cf. \cite{leonsimon}) that $\widetilde{S}_1=\widetilde{S}^+$ or $\widetilde{S}_1\cap \widetilde{S}^+=\emptyset$. Therefore, $\widetilde{S}_1=\widetilde{S}_2$ or $\widetilde{S}_1\cap \widetilde{S}_2=\emptyset$. In both cases we have $\widetilde{S}_2\leq \widetilde{S}_1$. 
    Since $S_1\in \mathcal{M}^{rec}_\alpha$ can be approximated from below, we can find a sequence of integer translations $\{\tau_{h_i}\}_i$ such that 
    \[ \tau_{h_i}\widetilde{S}_1\leq S_1, \text{  }\forall i\in \mathbb{N} \quad  \text{and} \quad \tau_{h_i}\widetilde{S}_1\to S_1. \]
    
    By Lemma 4.7 in \cite{uniquenessbangert}, we have that $\tau_{h_i}\widetilde{S}_2\to S_2$. Since $\widetilde{S}_2\leq \widetilde{S}_1$, it follows that 
    \[\tau_{h_i}\widetilde{S}_2\leq \tau_{h_i}\widetilde{S}_1, \quad \forall i\in \mathbb{N}.\]
    Consequently $S_2\leq S_1$. We get a contradiction. 
\end{proof}
    \begin{Lemma}\label{boundincylinder}
        There exist $\delta>0$ and $r_0>0$ such that
        \[D(S^+, Z(0, r))>\delta r^{n-1} \text{ and } D(S^-, Z(0,r))>\delta r^{n-1}\]
        for all $r>r_0$.
    \end{Lemma}

    \begin{proof}
        Fix the constant $\rho$ as in Lemma \ref{boundinball}. There exists a positive constant $c>0$ such that, for all sufficiently large $r$, the ball $B_P(0,r)$ contains at least $N(r):=c r^{n-1}$ pairwise disjoint open balls of radius $\rho$, which we shall denote by $B_P(p_i, \rho)$ for $i=1,2,\ldots, N(r)$. It follows that the Euclidean balls $B(p_i, \rho)$ are also disjoint and contained in the cylinder $Z(0, r)$ for all $i$. Therefore
        \[D(S^+, Z(0,r))\geq \sum_{i=1}^{N(r)} D(S^+, Z(p_i, \rho))\geq \epsilon c r^{n-1}.\]
        
        The second inequality follows from Lemma \ref{boundinball} and the fact that the balls are contained inside the cylinders. The desired inequality follows by choosing $\delta=\epsilon c$.
    \end{proof}

\begin{proof}[We can now complete the proof of Proposition \ref{lamination}]
By Theorem 
\ref{slab}, we can choose a slab of bounded width $B$ that contains all points of bounded distance $C$ to $P$. In particular, $B$ contains both $S_1$ and $S_2$. 

Choose $r$ sufficiently large. By Lemma \ref{boundincylinder}, we can find competitors $T_r^+, T_r^-$ with $\partial T_r^\pm=\partial (S^\pm \llcorner Z(0,r))$ such that
\[\mathbf{M}(T_r^+) \leq \mathbf{M}(S^+ \llcorner Z(0,r))-\delta r^{n-1} \quad \text{and}\]
\[\mathbf{M}(T_r^-) \leq \mathbf{M}(S^- \llcorner Z(0,r))-\delta r^{n-1}.\]
    Adding these two inequalities, we get
    \begin{equation}\label{contradiction}
    \begin{split}
    \mathbf{M}(T_r^+)+\mathbf{M}(T_r^-) &\leq \mathbf{M}(S^+\llcorner Z(0,r))+\mathbf{M}(S^-\llcorner Z(0,r))-2\delta r^{n-1} \\
    &=\mathbf{M}(S_1\llcorner Z(0,r))+\mathbf{M}(S_2\llcorner Z(0,r))-2\delta r^{n-1}.
    \end{split}
    \end{equation}
    
    Since $S_1$ and $S_2$ are area-minimizing and belong to the bounded slab $B$, their masses satisfy 
    \[\mathbf{M}(S_1\llcorner Z(0,r))\leq \mathbf{M}(T_r^+)+\mathbf{M}(\partial(Z(0,r)\llcorner B))\leq \mathbf{M}(T_r^+)+Cr^{n-2}\]
    for some constant $C>0$. Similarly, 
    \[\mathbf{M}(S_2\llcorner Z(0,r))\leq \mathbf{M}(T_r^-)+Cr^{n-2}.\]
    
    Combining these inequalities with (\ref{contradiction}), we obtain
    \[\mathbf{M}(T_r^+)+\mathbf{M}(T_r^-)\leq \mathbf{M}(T_r^+)+\mathbf{M}(T_r^-) +2Cr^{n-2}-2\delta r^{n-1}.\]
    
    Rearranging gives \[\delta r^{n-1}\leq C r^{n-2},\] which is impossible when $r$ is sufficiently large. We reach a contradiction, and this completes the proof of Proposition \ref{lamination}.
\end{proof}

\begin{Rem}
    A variant of the preceding argument yields an even stronger conclusion in the case where $\alpha$ is totally irrational: any two distinct elements in $\mathcal{M}_\alpha$ are disjoint. In other words, $\mathcal{M}_\alpha$ is totally ordered.
    In addition to that, Bangert (cf. \cite{uniquenessbangert}) proved the following stronger uniqueness result: the set $\mathcal{M}_\alpha^{rec}$ is the unique minimal set of $\mathcal{M}_\alpha$, under the action of the group of integer translations. This gives a dynamical characterization of the subset consisting of recurrent minimizers.
\end{Rem}

It is well-known that when $\alpha$ is irrational, the set of all recurrent minimizers $\mathcal{M}_\alpha^{rec}$ either forms a foliation of $\mathbb{R}^n$ or is homeomorphic to a Cantor set in $\mathbb{R}$. In the latter case, the open intervals belonging to the complement of the Cantor set are called the gaps of the lamination, with its endpoints corresponding to elements in $\mathcal{M}^{rec}_\alpha$ that can only be approximated from above or below, i.e., only one of (\ref{recurrent1}) or (\ref{recurrent2}) holds. The remaining uncountably many other minimizers are those that can be approximated both from above and from below.

The following lemma provides a universal finite upper bound for the volumes of all such gaps in the lamination.

\begin{Lemma}\label{gap}
Suppose that $\alpha$ is totally irrational and $G$ is a gap in the lamination generated by recurrent minimizers in $\mathcal{M}_\alpha$. Then the projection $\pi: \mathbb{R}^n \to \mathbb{T}^n$ is injective when restricted to $G$. Consequently,
\[
\operatorname{Vol}_g(G) \le \operatorname{Vol}_g(\mathbb{T}^n)<\infty.
\]
\end{Lemma}

\begin{proof}
    Let $\Sigma^- <\Sigma^+$ be elements in $\mathcal{M}_\alpha^{rec}$ that form the boundary of $G$. Since $G$ is a gap of the lamination, $\Sigma^\pm$ are consecutive, meaning that no other recurrent minimizers can be found between them. 
    
    Recall that the set $\mathcal{M}_\alpha^{rec}$ is totally ordered and invariant under integer translations. As a result, for every translation $\tau_k\in \mathbb{Z}^n$, $\tau_k(\Sigma^{\pm})$ are again recurrent minimizers. Furthermore, since $\alpha$ is totally irrational, for any nonzero $k\in\mathbb{Z}^n$ we have $I(\alpha, k)\neq 0$. It follows from the Birkhoff property that $\Sigma^\pm$ and $\tau_k(\Sigma^\pm)$ are disjoint. 
    
    Therefore, $\tau_k(\Sigma^\pm)$ bound the translated gap $\tau_k(G)$, which implies that $G$ and $\tau_k(G)$ are disjoint. Since this is true for every integer translation, it follows that the projection $\pi|_G$ is injective. This completes the proof. 
\end{proof}
\begin{Rem}
    In general, for a prescribed irrational class $\alpha$, there exist metrics on $\mathbb{T}^n$ for which non-recurrent minimizers do appear. The following construction of such metrics is due to Bangert in \cite{nonrecurrent}. 
    
    Fix a metric $g_0$ on $\mathbb{T}^n$ and assume that every minimizer in $\mathcal{M}_\alpha(g_0)$ is recurrent. We may further assume that $\mathcal{M}_\alpha(g_0)=\mathcal{M}_{\alpha}^{rec}(g_0)$ does not form a foliation, otherwise we are done. Let $G$ be one of the gaps, and choose a compact set $K$ contained inside $G$. We then decrease the metric $g_0$ inside $K$ to obtain a new metric $g_1$, which agrees with $g_0$ outside $K$, and for which the original minimizers with respect to $g_0$ are no longer minimizers with respect to $g_1$. 
    
    Consider the following family of interpolated metrics: 
    \[g_t=(1-t)g_0+tg_1, \quad t\in [0,1].\] 
    
    These are the affine combinations of $g_0$ and $g_1$. By the compactness property of minimizers, there exists a maximal time $0<T<1$ such that the original minimizers in $\mathcal{M}_\alpha(g_0)$ remain minimizers with respect to $g_T$. By the definition of recurrent minimizers, it follows that \[\mathcal{M}_\alpha^{rec}(g_T)=\mathcal{M}_\alpha^{rec}(g_0).\] 
    
    For every $t>T$, there must exist a minimizer $\Sigma_t\in \mathcal{M}_\alpha(g_t)$ that intersects $K$. Taking the limit as $t\to T^+$,  by the compactness property of area-minimizing hypersurfaces, the sequence $\{\Sigma_t\}$ converges to a minimizer $\Sigma_T\in \mathcal{M}_\alpha(g_T)$ that intersects $K$. This minimizer is non-recurrent. 

    This phenomenon, however, is not generic. In fact, an arbitrarily small perturbation of the metric suffices to eliminate all non-recurrent minimizers in $\mathcal{M}_\alpha$. Equivalently, $\mathcal{M}_\alpha=\mathcal{M}_\alpha^{rec}$ generically. Nonetheless, recurrent minimizing hypersurfaces are limits of periodic minimizers.
\end{Rem}
\begin{Prop}[Bangert, \cite{uniquenessbangert}]
    For any recurrent minimizer $\Sigma \in \mathcal{M}^{rec}_\alpha$ and any sequence of integer classes $r_k$ converging to $\alpha$, there exists a subsequence (still denoted by $r_k$) and $\Sigma_k\in \mathcal{M}^{per}_{r_k}$ such that $\Sigma_k\to \Sigma$. 
\end{Prop}
\begin{proof}
    For each $k$, fix a periodic minimizer $R_k\in \mathcal{M}^{per}_{r_k}$, then consider the discrete set 
    \[\mathcal{O}_k=\{\tau_s R_k: s\in \pi_1(\mathbb{T}^n)\cong \mathbb{Z}^n\}.\]
    Let $\mathcal{O}$ be the closure of the union of all $\mathcal{O}_k$. Choose a sequence of minimizers $S_k \in \mathcal{O}_k$ intersecting a fixed compact set. By Theorem \ref{compactness}, we have \[S_k\to S\in \mathcal{M}_\alpha\cap \mathcal{O}. \] 

    Since the metric is periodic, $\mathcal{O}$ is closed and invariant under integer translations. Therefore $\overline{\{\tau_s S: s\in \mathbb{Z}^n\}}\subseteq \mathcal{O}$. Since $\mathcal{M}_\alpha^{rec}$ is the unique minimal set of $\mathcal{M}_\alpha$ under the action of integer translations, we conclude that \[\mathcal{M}^{rec}_\alpha\subseteq\overline{\{\tau_s S: s\in \mathbb{Z}^n\}}\subseteq \mathcal{O}.\]
    This completes the proof of the proposition.
\end{proof}

It was conjectured by Moser \cite{moserreport} that the above inclusions should be equalities, meaning all limits of sequences of periodic minimizers are recurrent. This is not known to be true yet. The following proposition gives an obstruction for the alternative in terms of the number of closed stable hypersurfaces in a sequence of integer classes approaching $\alpha$.

\begin{Prop}
    Let $r_k\in H_{n-1}(\mathbb{T}^n, \mathbb{Z})$ be a sequence of primitive integer classes converging to a totally irrational class $\alpha$. Suppose that there exist periodic minimizers $\Sigma_k \in \mathcal{M}(r_k)$ such that 
    \[\Sigma_k\to \Sigma\in \mathcal{M}_\alpha \setminus \mathcal{M}_\alpha^{rec} \]
    locally smoothly in $\mathbb{R}^n$.
    Assume further that $\Sigma$ is isolated in $\mathcal{M}_\alpha$. Then there exists $K\in \mathbb{N}$ such that for all $k\geq K$, there exists another connected stable minimal hypersurface $\widetilde{\Sigma}_k$ different from $\pi(\Sigma_k)$, representing the class $r_k$. In particular, if for every positive integer $k$, $\pi(\Sigma_k)$ is the unique connected stable minimal hypersurface representing the class $r_k$, then $\Sigma$ is recurrent or a non-isolated non-recurrent minimizer in $\mathcal{M}_\alpha$.
\end{Prop}

\begin{proof}
Since $\Sigma$ is not recurrent, Bangert \cite{uniquenessbangert} showed that $\Sigma$ lies in a gap $G$ bounded by two recurrent minimizers 
 $\Sigma^- < \Sigma^+$ in $\mathcal{M}_\alpha^{rec}(g)$, where 
 \[\Sigma^+=\inf \{\tau_s \Sigma: s\in \pi_1(\mathbb{T}^n), s\cdot \alpha>0\},\]
 and \[\Sigma^-=\sup \{\tau_s\Sigma: s\in \pi_1(\mathbb{T}^n), s\cdot \alpha<0\}.\]
 
 By Lemma \ref{gap}, the projection $ \pi:G\to \mathbb{T}^n$  is injective, and $G$ has finite $g$-volume. 
Choose a point $p\in \Sigma$. Since $\Sigma$ is isolated in $\mathcal{M}_\alpha(g)$, we may choose a relatively compact open set $ \Omega \subset G $ with smooth boundary such that $p\in \Omega$, $\Omega\cap \Sigma\neq\emptyset$, and no element of $\mathcal{M}_\alpha(g)$ other than $\Sigma$ intersects $\overline{\Omega}$. 

Let $\varphi\in C^\infty(T^n)$ be nonnegative, supported in $\pi(\Omega)$, and positive on $\pi(\Omega)$. This lifts to a $\mathbb{Z}^n$-periodic function on $\mathbb{R}^n$. Define a new metric 
\[ g_\varphi := e^{2\varphi}g \quad \text{on } \mathbb{T}^n. \] 
The lifted metric on $\mathbb{R}^n$ will still be denoted by $g_\varphi$. 
By construction, $g_\varphi\ge g$, and $g_\varphi=g$ outside $\pi(\Omega)$. Since $\Sigma^\pm$ are disjoint from the support of the perturbation, they remain area-minimizing with respect to $g_\varphi$. Indeed, for every compactly supported competitor $T$: 
\[ \mathbf M_{g_\varphi}(T\llcorner W) \ge \mathbf M_g(T\llcorner W) \ge \mathbf M_g(\Sigma^\pm\llcorner W) = \mathbf M_{g_\varphi}(\Sigma^\pm\llcorner W). \]

We claim that no element of $\mathcal{M}_\alpha(g_\varphi)$ intersects the region $\overline{\Omega}$. Suppose, to the contrary, that $\Gamma\in \mathcal{M}_\alpha(g_\varphi)$ intersects this region. If $\Gamma$ does not meet $\Omega$, then it must intersect $\partial \Omega$ and belong to $\mathcal{M}_\alpha(g)$, contradicting the assumption that no other minimizer in $\mathcal{M}_\alpha(g)$ intersects $\overline{\Omega}$. Hence we can assume $\Gamma \cap \Omega\neq \emptyset$.

Since $\Sigma^\pm$ are $g_\varphi$-minimizing and have the same homological direction as $\Gamma$, the Birkhoff property implies that $\Gamma$ lies between $\Sigma^-$ and $\Sigma^+$, hence inside the same gap $G$. Let $B_R=B_R(p)$ denote the Euclidean coordinate ball, as in Section~\ref{section6}.

Since $G$ has finite volume, Lemma~\ref{asymptotic} gives an increasing sequence of radii $R_j\to\infty$, chosen after $\Gamma$ is fixed and transverse to $\Gamma,\Sigma^-,\Sigma^+$, such that
\[
    \mathcal H_g^{n-1}(G\cap\partial B_{R_j})\longrightarrow0.
\]
All the following comparisons are made along this sequence.

For each choice of sign, the restriction formula for the region between $\Gamma$ and $\Sigma^\pm$ gives a spherical closing current carried by $G\cap\partial B_{R_j}$, whose $g$-mass is at most $\mathcal H_g^{n-1}(G\cap\partial B_{R_j})$. Since $g_\varphi=e^{2\varphi}g$, the masses of these currents in both metrics are bounded by
\[
    \eta_j:=e^{(n-1)\|\varphi\|_\infty}
    \mathcal H_g^{n-1}(G\cap\partial B_{R_j})\longrightarrow0.
\]
In particular, since $\Sigma^\pm$ are $g$-area-minimizing, comparison gives
\[
    \mathbf M_g(\Gamma\llcorner B_{R_j})
    \geq\mathbf M_g(\Sigma^\pm\llcorner B_{R_j})-\eta_j.
\]

On the other hand, because $\Gamma$ meets the region where $\varphi>0$, a fixed compact patch of $\Gamma$ has a positive conformal area increase, say $2\lambda>0$. Once $B_{R_j}$ contains this patch,
\[
    \mathbf M_{g_\varphi}(\Gamma\llcorner B_{R_j})
    \geq\mathbf M_g(\Gamma\llcorner B_{R_j})+2\lambda.
\]
Combining these inequalities and using $g_\varphi=g$ along $\Sigma^\pm$, we obtain
\[
    \mathbf M_{g_\varphi}(\Gamma\llcorner B_{R_j})
    \geq\mathbf M_{g_\varphi}(\Sigma^\pm\llcorner B_{R_j})
       +2\lambda-\eta_j.
\]

The same spherical closing currents provide competitors for $\Gamma\llcorner B_{R_j}$. Its $g_\varphi$-area-minimality therefore gives
\[
    \mathbf M_{g_\varphi}(\Gamma\llcorner B_{R_j})
    \leq\mathbf M_{g_\varphi}(\Sigma^\pm\llcorner B_{R_j})+\eta_j.
\]
Hence $2\lambda\leq2\eta_j$, a contradiction for all sufficiently large $j$. This proves the claim.

Now, for each $k$, let $\widetilde{\Sigma}_k$ be a closed area-minimizing hypersurface in the class $r_k$ with respect to $g_\varphi$, and choose a lift 
\[ \Sigma_k'\in \mathcal{M}^{per}_{r_k}(g_\varphi). \]

We argue that, for all sufficiently large $k$, no lift of $\widetilde{\Sigma}_k$ intersects $\Omega$. Otherwise, after passing to a subsequence and choosing appropriate integer translates, we would obtain a sequence $\Sigma_k'$ intersecting $\Omega$. By the compactness theorem \ref{compactness}, a subsequence converges locally smoothly to  $\Gamma\in \mathcal{M}_\alpha(g_\varphi)$ which intersects $\overline{\Omega}$, contradicting the claim. Therefore, for all sufficiently large $k$, the hypersurface $\widetilde{\Sigma}_k$ is disjoint from $\Omega$. Hence $\widetilde{\Sigma}_k$ is not only minimal and stable with respect to $g_\varphi$, but also with respect to the original metric $g$. Since $r_k$ is primitive, the minimizing representative $\widetilde{\Sigma}_k$ is connected.

Finally, because $\Sigma_k\to \Sigma$ and $\Sigma$ intersects the region where $\varphi>0$, the projected hypersurface $\pi(\Sigma_k)$ also intersects this region for all sufficiently large $k$. By contrast, $\widetilde\Sigma_k$ is disjoint from it. Hence $ \widetilde\Sigma_k \neq \pi(\Sigma_k).$ This completes the proof. \end{proof}

\section{Necessary condition for minimal foliation}
\label{section6}\label{sec6:section}

In this section, we prove the implication $(3)\Rightarrow(1)$ in Theorem~\ref{maintheorem1}, together
with Theorem~\ref{maintheorem3}. Throughout, $3\leq n\leq7$, the metric $g$ is fixed,
and convergence of primitive integral classes is understood in the
normalized sense on the unit ball of the stable norm. 

\begin{Th}\label{sec6:main}
Suppose that $\alpha$ is totally irrational and that
$\mathcal M_\alpha$ foliates $\mathbb R^n$. For every sequence of primitive classes $r_k\to\alpha$,
let $M_k<N_k$ be successive lifts of the same homologically
area-minimizing representative of $r_k$. Then for every $\varepsilon>0$ and all sufficiently
large $k$ there is a flat-continuous family of integral cycles
$\Phi_k:[0,1]\to\mathcal Z_{n-1}(X_k)$, with no concentration of mass,
such that
\begin{equation}\label{sec6:target}
\Phi_k(0)=M_k,\qquad \Phi_k(1)=N_k,\qquad
S(r_k)\leq\mathbf M(\Phi_k(t))\leq S(r_k)+\varepsilon.
\end{equation}
Here $X_k=\mathbb R^n/(r_k^\perp\cap\mathbb Z^n)$, and
$\mathcal Z_{n-1}(X_k)$ denotes compactly supported integral cycles. The projected family is an
$r_k$-sweep-out of $\mathbb T^n$. Consequently,
\[
\lim_{k\to\infty}\bigl(\omega(r_k)-S(r_k)\bigr)=0.
\]
For rationally dependent $\alpha$, the same conclusions hold if
$\mathcal M(\alpha)$ foliates $\mathbb R^n$.
\end{Th}

We want to interpolate from $M_k$ to $N_k$ inside $X_k$, sweeping through the region $H_k$ between them. The construction below partitions $H_k$ into two types of domains. The Type B domains contain the
portions where the region between the boundaries $M_k, N_k$ is not
sufficiently thin, and there we transfer interpolations from the limiting
minimizing foliation. On the Type A remainder we replace small pieces
in a prescribed order and apply Corollary~\ref{interpolation}. The mass of each partial
replacement is controlled by its lateral boundary, whereas the mass
of each individual increment controls the cost of interpolation.

\subsection{The cyclic cover and lateral area comparison}
\label{sec6:cover}
Write
\[
\Lambda_k=r_k^\perp\cap\mathbb Z^n,\qquad
X_k=\mathbb R^n/\Lambda_k,\qquad
\pi_k:X_k\longrightarrow\mathbb T^n,
\]
and let $q_k:\mathbb R^n\to X_k$ be the quotient map. The hypersurfaces
$M_k,N_k$ are compact in $X_k$. They are successive lifts of one fixed
minimizing representative in $r_k\in H_{n-1}(\mathbb{T}^n, \mathbb{Z})$. With consistent orientations, the
region $H_k$ between them satisfies
\begin{equation}\label{sec6:strip}
\begin{aligned}
\partial\llbracket H_k\rrbracket&=N_k-M_k,
&\mathbf M(M_k)=\mathbf M(N_k)&=S(r_k),\\
\operatorname{Vol}_g(H_k)&=V:=\operatorname{Vol}_g(\mathbb T^n),
&(\pi_k)_\#\llbracket H_k\rrbracket&=\llbracket\mathbb T^n\rrbracket.
\end{aligned}
\end{equation}
Both boundary components $M_k, N_k$ minimize mass in their homology class in
$X_k$. Denote their full lifts and the full
periodic strip in $\mathbb R^n$ by
$\widetilde M_k,\widetilde N_k,\widetilde H_k$, respectively.

We write $d_0$ for the flat distance and $d_g$ for the
Riemannian distance, on either a cover or the universal cover as
appropriate. Thus
\[
B^0_r(p)=\{x:d_0(x,p)<r\},\qquad
B_g(p,r)=\{x:d_g(x,p)<r\}
\]
are ambient balls. We omit the superscript $0$ when no confusion can arise.  All masses and Hausdorff measures below are
computed using $g$, even when a ball is defined using $d_0$.
The two distances are uniformly comparable on all these covers, since $g$ is periodic.

\begin{Lemma}\label{sec6:finite-volume}
Assume that $\alpha$ is totally irrational. Suppose that, after
integer translations and passing to a subsequence, the full lifts
converge locally smoothly to $M<N$. Then the open region $G$ between
$M,N$ projects injectively to $\mathbb T^n$. In particular,
$\operatorname{Vol}_g(G)\leq V$.
\end{Lemma}

\begin{proof}
For each fixed $z\in\mathbb Z^n\setminus\{0\}$, total irrationality
implies $r_k\cdot z\neq0$ for all sufficiently large $k$. Therefore:
\[
\widetilde H_k\cap\tau_z\widetilde H_k=\varnothing,
\qquad \tau_z(x)=x+z.
\]
The bounding sets converge locally with their oriented sides. If
$G\cap\tau_zG$ were nonempty, a small compact ball in this intersection
would belong to both strips for all sufficiently large $k$, a
contradiction. Hence $G$ is disjoint from each of its nontrivial
integer translates. This proves injectivity and the corresponding volume bound.
\end{proof}

We use the following decay lemma, proved in
Appendix~\ref{appendixd}.

\begin{Lemma}\label{asymptotic}\label{sec6:slices}
Let $G$ be a finite-volume region between the ordered minimizing
hypersurfaces under consideration, and fix $p\in\mathbb R^n$.
For every positive sequence $\eta_j\to0$, there are radii
$R_j\to\infty$ such that
\begin{equation}\label{sec6:decay}
\mathcal H_g^{n-1}(G\cap\partial B_{R_j}(p))\leq\eta_j.
\end{equation}
The spheres can be chosen transverse to both boundary hypersurfaces $M,N$.
\end{Lemma}

For a relatively compact domain $B\subset X_k$ with piecewise smooth boundary meeting $M_k, N_k$ transversely, define 
\[
w_k(B):=\mathcal H_g^{n-1}(H_k\cap\partial B).
\]
 The
restriction formula, see~\cite{simon}, gives
\begin{equation}\label{sec6:restriction}
\partial\llbracket H_k\cap B\rrbracket
=N_k\llcorner B-M_k\llcorner B+W_{k,B},
\qquad \mathbf M(W_{k,B})\leq w_k(B).
\end{equation}
Let $P\subset H_k$ be a set of finite perimeter, and let
$M_k^P,N_k^P$ be the restriction of $P$ on $M_k, N_k$. Then
\[
\partial\llbracket P\rrbracket=N_k^P-M_k^P+W(P),
\qquad \operatorname{spt}W(P)\subset\overline{H_k},
\]
where the $\textit{lateral boundary}$ $W(P)$ has no mass on $M_k$ and $N_k$. Since $N_k$ is area-minimizing, we have:
\[
\mathbf M(N_k^P)\leq\mathbf M(M_k^P)+\mathbf M(W(P)).
\]
 Cancelling $M_k^P$ from $M_k$
gives
\begin{equation}\label{sec6:lateral}
\mathbf M\bigl(M_k+\partial\llbracket P\rrbracket\bigr)
\leq S(r_k)+2\mathbf M(W(P)).
\end{equation}
The cycle on the left retains $M_k$ outside the replacement, inserts
the corresponding portion of $N_k$, and joins the two by $W(P)$. We later want to make sure that this lateral boundary has a small area.

\subsection{A localized interpolation lemma}
\label{sec6:scales}
We use the codimension-one, integer-coefficient form of the
Marques--Neves interpolation theorem stated in Theorem~\ref{discretetocontinuous} and
Corollary~\ref{interpolation} (compare~\cite[Theorem~14.1]{willmore}).
The following localization keeps its constants independent of the
varying cover $X_k$.

\begin{Lemma}
\label{sec6:localized}
After enlarging $C_0$ and decreasing $\delta_0$ from Theorem~\ref{discretetocontinuous}, there
is $0<a_0<1/8$, depending only on $g$, with the following property.
Suppose that an integral $n$-current $Q$ is supported in
$\overline{B^0_{2a}(p)}\subset X_k$, where $a<a_0$, and
\[
\mathbf M(\partial Q)<\delta<\delta_0.
\]
Then for any compact integral cycle $T$ in $X_k$, there is a
mass-continuous path $\Psi$ from $T$ to $T+\partial Q$ such that
\begin{equation}\label{sec6:interpolation}
\sup_{t\in[0,1]}\mathbf M(\Psi(t)-T)\leq C_0\delta.
\end{equation}
In particular,
\[
\sup_{t\in[0,1]}\mathbf M(\Psi(t))
\leq\mathbf M(T)+C_0\delta.
\]
Moreover, $\Psi(t)=T+\partial Q(t)$, where $Q(t)$ is continuous in
mass, is supported in the same closed ball, and satisfies
$Q(0)=0$, $Q(1)=Q$. The constants are independent of $k$, $p$, $a$,
and the initial cycle $T$.
\end{Lemma}

\begin{proof}
Choose $a_0$ sufficiently small so that $\overline{B^0_{4a}(p)}$ projects
injectively to $\mathbb T^n$. By applying Corollary~\ref{interpolation} to $0$ and
$(\pi_k)_\#\partial Q$, we get a mass-continuous path $A(t)$ of
null-homologous cycles on the fixed torus with
$\mathbf M(A(t))\leq C_0\delta$.

Let \(F:\mathbb T^n\to\overline{B^0_{2a}(p)}\) be a radial Lipschitz map that agrees with the lifted identity on \(B^0_{2a}(\pi_k(p))\) and collapses the complement of \(B^0_{4a}(\pi_k(p))\) to \(p\). Its Lipschitz constant is bounded independently of \(a,p,k\). Then \(\widehat A(t):=F_\#A(t)\) is a mass-continuous path from \(0\) to \(\partial Q\), supported in \(\overline{B^0_{2a}(p)}\), with
\[
\sup_t\mathbf M(\widehat A(t))\le C_0\delta,
\]
after enlarging \(C_0\).

Following Simon \cite{simon}, denote by $C_p(A)$ the cone over $A$ in the flat coordinates:
\[
C_p(A)
=
H_\#\bigl(\llbracket(0,1)\rrbracket\times A\bigr),
\qquad
H(t,x)=p+t(x-p).
\]
For
cycles supported in the target ball,
\[
\partial C_p(A)=A,\qquad
\mathbf M\bigl(C_p(A)-C_p(A')\bigr)
\leq C_ga\,\mathbf M(A-A').
\]
Thus $Q(t)=C_p(\widehat A(t))$ is mass-continuous, has the required
support, and satisfies $\partial Q(t)=\widehat A(t)$ and $Q(0)=0$.
The current $Q(1)-Q$ is an $n$-cycle compactly supported in a
coordinate neighborhood. Extend it by zero to $\mathbb R^n$, the
constancy theorem and compact support give $Q(1)=Q$. Set
$\Psi(t)=T+\widehat A(t)$. The last inequality follows
from~\eqref{sec6:interpolation} and the triangle inequality.
\end{proof}

We will use two local geometric bounds, with a fixed constant $C_g$.
Area comparison in the universal cover gives
\begin{equation}\label{sec6:area}
\begin{aligned}
\mathbf M(M_k\llcorner B^0_{2a}(p))
+\mathbf M(N_k\llcorner B^0_{2a}(p))&\leq C_ga^{n-1},\\
\mathbf M(M_k\llcorner B^0_L(p))&\leq C_g(L+1)^{n-1}.
\end{aligned}
\end{equation}
For these upper bounds, write a lifted minimizing hypersurface as
$S=\partial\llbracket E\rrbracket$ and compare it with
$S-\partial\llbracket E\cap B_r^0\rrbracket$. For almost every $r$,
\[
\mathbf M(S\llcorner B_r^0)
\leq\mathbf M\bigl((\partial\llbracket B_r^0\rrbracket)\llcorner E\bigr)
\leq C_gr^{n-1}.
\]
Using a slightly larger admissible radius gives the estimates for
every radius.

We also have the following volume bound
\begin{equation}\label{sec6:tube}
\operatorname{Vol}_g\bigl(
B^0_{3L}(p)\cap\{d_g(\cdot,M_k)<h\}\bigr)
\leq C_gh(L+1)^{n-1},
\end{equation}
for all $L$ and sufficiently small $h<h_0$. Indeed, choose a maximal disjoint collection of small balls of radius
comparable to $h$ centered on $M_k$ in a slightly larger ball. The
hypersurface density bound and~\eqref{sec6:area} bound their number by
$C_g(L+1)^{n-1}h^{1-n}$. Their fixed enlargements cover the indicated
$h$-neighborhood, and each has volume at most $C_gh^n$. This proves
\eqref{sec6:tube}. All these estimates follow from the local
minimizing and compactness arguments in Section~\ref{section4}. 

Fix $0<\varepsilon<1$, and choose the following constants:
\begin{equation}\label{sec6:choices}
\begin{gathered}
\delta_\varepsilon
=
\min\left\{
\frac{\delta_0}{2},
\frac{\varepsilon}{4(2+C_0)}
\right\},
\\[6pt]
0<a_\varepsilon<a_0,
\qquad
C_g a_\varepsilon^{n-1}
<
\frac{\delta_\varepsilon}{8},
\\[6pt]
L_\varepsilon>1,
\qquad
\frac{C_g V}{L_\varepsilon}
<
\frac{\delta_\varepsilon}{16},
\\[6pt]
0<h_\varepsilon
<
\min\left\{
h_0,\,
\frac{a_\varepsilon}{10},\,
\frac{\delta_\varepsilon a_\varepsilon}
     {16C_g(L_\varepsilon+1)^{n-1}}
\right\}.
\end{gathered}
\end{equation}

\subsection{The Type B covering}
\label{sec6:typeB-cover}
We first treat totally irrational $\alpha$. In that case,
\begin{equation}\label{sec6:escaping}
\min\{|z|:0\neq z\in\Lambda_k\}\longrightarrow\infty.
\end{equation}

\medskip

\noindent Consider the compact \textit{thick} set
\[
H_{k}^{\text{thick}}=\{x\in\overline{H_k}:d_g(x,M_k)\geq h_\varepsilon\}.
\]
There is a constant $c_g>0$, depending only on $g$, such that
for every $x\in H_{k}^{\text{thick}}$:
\begin{equation}\label{sec6:core-density}
\operatorname{Vol}_g\bigl(H_k\cap B_g(x,h_\varepsilon/2)\bigr)
\geq c_g h_\varepsilon^n.
\end{equation}
Indeed, if $d_g(x,N_k)\geq h_\varepsilon/4$, then
$B_g(x,h_\varepsilon/4)\subset H_k$, and the uniform
small-ball volume bound gives \eqref{sec6:core-density}.
Otherwise, choose $y\in N_k$ with
$d_g(x,y)<h_\varepsilon/4$. Since
$d_g(y,M_k)>3h_\varepsilon/4$, the ball
$B_g(y,h_\varepsilon/8)$ misses $M_k$ and is contained in
$B_g(x,h_\varepsilon/2)$. The standard one-sided
volume-density estimate for locally area-minimizing
boundaries, applied to the side of $N_k$ facing $H_k$, gives
\[
\operatorname{Vol}_g\bigl(H_k\cap B_g(y,h_\varepsilon/8)\bigr)
\geq c_g h_\varepsilon^n,
\]
after possibly decreasing $c_g$. This proves
\eqref{sec6:core-density} in this case as well.
\medskip

We now cover $H_{k}^{\text{thick}}$ by a uniformly bounded number of small balls of bounded radii. All balls are taken below a fixed small radius on the torus,
so the volume and density constants are uniform over the
covers $X_k$.

\medskip
\noindent
Choose a maximal collection of ball centers
$p_{k,1},\ldots,p_{k,m_k}\in H_{k}^{\text{thick}}$ such that the balls
$B_g(p_{k,\ell},h_\varepsilon/2)$ are pairwise disjoint. We call
these selected small balls \emph{cores}. By maximality,
\[
H_{k}^{\text{thick}}\subset\bigcup_{\ell=1}^{m_k}B_g(p_{k,\ell},h_\varepsilon).
\]
Each core contains at least $c_gh_\varepsilon^n$ of the volume of
$H_k$, by~\eqref{sec6:core-density}. Since the cores are disjoint, we have
$m_kc_gh_\varepsilon^n\leq V$. Therefore, the number of cores is at most
\begin{equation}\label{sec6:number}
m_k\leq N(\varepsilon):=
\max\left\{1,\left\lceil\frac{V}{c_gh_\varepsilon^n}\right\rceil\right\}.
\end{equation}
If $H_{k}^{\text{thick}}=\varnothing$, we
put $U_k=\varnothing$ and omit the Type B construction.

\medskip
\noindent
 Since $m_k$ are uniformly bounded, we can pass to a subsequence
on which $m_k=m\geq1, \forall k$, and fix the labels $1,\ldots,m$ for the centers. We can choose a further common subsequence of the centers so that every pairwise distance between them either remains
uniformly bounded or tends to infinity. On this common subsequence define an equivalence relation on the labels by
\[
\ell\sim q
\quad\Longleftrightarrow\quad
\sup_k d_0(p_{k,\ell},p_{k,q})<\infty.
\]
Let $I_1,\ldots,I_J$ be the equivalence classes,
which we call the groups of cores. There are at most $N(\varepsilon)$ groups. Since the number of labels is
finite, each group has diameter bounded independently of $k$, and
\[
\min_{\substack{\ell\in I_j, q\in I_{j'}}}
 d_0(p_{k,\ell},p_{k,q})\longrightarrow\infty
\qquad(j\neq j').
\]

The purpose of this grouping is to later translate a whole group into a fixed compact bounded set, then enlarge it uniformly so that the lateral area is small, then use foliation in the limit to construct a sweep-out inside. This is done below and in Step 1.

\medskip

Fix a group $I_j$, lift a bounded neighborhood of its cores using
\eqref{sec6:escaping}, and choose an integer translation $z_{k,j}$
so that, after applying $\tau_{-z_{k,j}}$, one representative center
lies in a fixed fundamental cube. All centers of that group then
remain in a fixed bounded set. The recentering depends on both $k$
and $j$.

Both endpoint hypersurfaces meet a uniformly bounded neighborhood
of a core center. Theorem~\ref{compactness} gives a common
subsequence on which, for every group $j$,
\[
\tau_{-z_{k,j}}\widetilde M_k\longrightarrow M^{(j)},\qquad
\tau_{-z_{k,j}}\widetilde N_k\longrightarrow N^{(j)}
\quad\text{locally smoothly in }\mathbb R^n,
\]
with $M^{(j)},N^{(j)}\in\mathcal M_\alpha$.
The local smooth convergence, together with the consistent
orientations, implies
\[
\chi_{\tau_{-z_{k,j}}\widetilde H_k}
\longrightarrow \chi_{G_j}
\quad\text{in }L^1_{\mathrm{loc}}(\mathbb R^n),
\]
where $G_j$ is the region between $M^{(j)}$ and $N^{(j)}$.
If $M^{(j)}=N^{(j)}$, the intervening volume in every
fixed compact set converges to zero. This contradicts~\eqref{sec6:core-density},
since a core of the recentered group remains in a fixed compact
set and carries a fixed positive volume. Thus
$M^{(j)}<N^{(j)}$. Different groups may have different limiting
pairs. By Lemma~\ref{sec6:finite-volume}, the intervening region
$G_j$ satisfies $\operatorname{Vol}_g(G_j)\leq V$.

\medskip
For each $G_j$, we first choose a radius containing all the $h_\varepsilon$-covering balls of the group $j$ for
large $k$.
By Lemma~\ref{sec6:slices} we can enlarge the radius to $R_j$ for which
$B_j=B_{R_j}$ has boundary transverse to $M^{(j)},N^{(j)}$ and
\begin{equation}\label{sec6:wall-budget}
\sigma_j:=\mathcal H_g^{n-1}(G_j\cap\partial B_j)
<\frac{\delta_\varepsilon}{64N(\varepsilon)}.
\end{equation}

On a fixed neighborhood of $\overline{B_j}$, local smooth
convergence gives diffeomorphisms $F_{k,j}\to\mathrm{id}$ in $C^1$
mapping $M^{(j)},N^{(j)}$ to $\tau_{-z_{k,j}}\widetilde{M}_k, \tau_{-z_{k,j}}\widetilde{N}_k$.
To construct $F_{k,j}$, write $\tau_{-z_{k,j}}\widetilde{M}_k, \tau_{-z_{k,j}}\widetilde{N}_k$ as small
normal graphs over $M^{(j)},N^{(j)}$ and extend their
displacements with cutoffs in disjoint tubular neighborhoods.

Return to the original cover by setting
\[
B_{k,j}=q_k\bigl(\tau_{z_{k,j}}(F_{k,j}(B_j))\bigr).
\]
For sufficiently large $k$, this map is injective on the fixed
neighborhoods under consideration, and the oriented matching gives
\[
(q_k\circ\tau_{z_{k,j}}\circ F_{k,j})(G_j\cap B_j)
=H_k\cap B_{k,j}.
\]

The margin preserves the covering of $H_{k}^{\text{thick}}$. Since $R_j$ is fixed and
$F_{k,j}\to\mathrm{id}$, each $B_{k,j}$ lies in $B^0_{D_j}(p_{k,\ell_j})$
for a representative label $\ell_j\in I_j$ and a constant $D_j$
independent of $k$. Because the centers in different groups separate,
the closures $\overline{B_{k,j}}$ are pairwise disjoint for all
sufficiently large $k$. Since
$U_k=\bigcup_{j=1}^J B_{k,j}$, we have:
\begin{equation}\label{sec6:typeB-wall}
w_k:=\mathcal H_g^{n-1}(H_k\cap\partial U_k) \leq \sum_{j=1}^{J}\mathcal{H}^{n-1}_g(H_k\cap \partial B_{k,j})
<\frac{\delta_\varepsilon}{32}
\end{equation}
for all sufficiently large $k$. The complementary thin region satisfies
\begin{equation}\label{sec6:thin-remainder}
H_{k}^{\text{thin}}:=H_k\setminus U_k
\subset\{d_g(\cdot,M_k)<h_\varepsilon\}.
\end{equation}

\subsection{Step 1: interpolation inside the Type B domains}
\label{sec6:typeB}
We interpolate from $M_k$ to $M_k+\partial\llbracket H_k\cap U_k\rrbracket $. 
Fix one group $I_j$ and omit its index for convenience.
Write
\[
M=M^{(j)},\quad N=N^{(j)},\quad
G=G_j,\quad B=B_j,\quad B_k=B_{k,j},
\]
and set
\[
F_k=F_{k,j},\qquad
\widehat F_k=q_k\circ\tau_{z_{k,j}}\circ F_k.
\]
Since $\mathcal{M}_\alpha$ is a foliation,
the region $G$ is filled by an ordered family of minimizing leaves
\[
\varphi_t=\partial\llbracket E_t\rrbracket,\qquad
E_0\subset E_t\subset E_1,\qquad
\varphi_0=M,\quad\varphi_1=N.
\]
The ordered leaf interval is compact in the local smooth topology:
every intermediate leaf separates the endpoints and meets a fixed
compact arc between them, so Theorem~\ref{compactness} applies. The map
\[
E\longmapsto\operatorname{Vol}_g(E\setminus E_0)
\]
is continuous and strictly increasing with image
$[0,\operatorname{Vol}_g(G)]$. Normalize this to obtain a family with
$0\leq t\leq1$.

For the fixed cutting domain $B$, put
\[
P_t=\llbracket(E_t\setminus E_0)\cap B\rrbracket,
\qquad w(B)=\mathcal H_g^{n-1}(G\cap\partial B).
\]
For $s<t$, nestedness and the chosen parametrization give
\[
\mathbf M(P_t-P_s)
\leq\operatorname{Vol}_g(E_t\setminus E_s)
=(t-s)\operatorname{Vol}_g(G).
\]
In particular, the fillings are continuous in mass. As a result, the boundary family $\partial P_t$ is continuous in the flat topology.

At admissible slicing parameters, we have
\[
\partial P_t=\varphi_t\llcorner B-M\llcorner B+W_t,
\qquad \mathbf M(W_t)\leq w(B).
\]
Since each leaf $\varphi_t$ is area-minimizing, we have
\begin{equation}\label{sec6:leaf-comparison}
\mathbf M(\varphi_t\llcorner B)
\leq\mathbf M(M\llcorner B)+w(B).
\end{equation}
Thus the interior replacement and its lateral boundary cost at most \(2w(B)\). Since the exceptional slicing parameters are at most countable, flat continuity and lower semicontinuity of mass extend this estimate to every \(t\), also after push-forward by the matching map.

In the quotient $X_k$, consider
\[
T_{k,B}(t)=M_k+\partial(\widehat{F}_{k\#}P_t).
\]
Since $\widehat{F}_{k\#}(M\llcorner B)=M_k\llcorner B_k$, the portion of $M_k$ inside $B_k$ is removed and we have:
\[
T_{k,B}(t)
=M_k-M_k\llcorner B_k
+\widehat{F}_{k\#}(\varphi_t\llcorner B)+\widehat{F}_{k\#}W_t.
\]
The endpoints are:
\[
T_{k,B}(0)=M_k,\qquad
T_{k,B}(1)=M_k+\partial\llbracket H_k\cap B_k\rrbracket.
\]
Since $F_k\to\mathrm{id}$ in $C^1$ and
$q_k\circ\tau_{z_{k,j}}$ is a local isometry,
there exists $\theta_{k,B}\to0$ such that
\[
(1-\theta_{k,B})\mathbf M(\Sigma)
\leq \mathbf M(\widehat F_{k\#}\Sigma)
\leq (1+\theta_{k,B})\mathbf M(\Sigma).
\]
for every rectifiable $(n-1)$-current $\Sigma$
supported in $B$. Writing
$A_B=\mathbf M(M\llcorner B)$, we obtain
\begin{equation}\label{sec6:typeB-local-mass}
\begin{aligned}
\mathbf M(T_{k,B}(t))-S(r_k)
&\leq(1+\theta_{k,B})(A_B+w(B))
-(1-\theta_{k,B})A_B+(1+\theta_{k,B})w(B)\\
&\leq2w(B)+2\theta_{k,B}(A_B+w(B)).
\end{aligned}
\end{equation}

We now restore the group index and denote the last term in
\eqref{sec6:typeB-local-mass} by $e_{k,j}$. With all domains fixed,
take $k$ large enough that $\sum_j e_{k,j}<\varepsilon/8$.
Repeat the process on all disjoint domains $B_{k,j}$, one at a time. The construction arrives at
\[
T_k^{\mathrm{far}}
=M_k+\partial\llbracket H_k\cap U_k\rrbracket
\]
and gives a flat-continuous family $\Phi_k^{\mathrm{far}}$ from $M_k$ to $T_k^{\mathrm{far}}$ with
\begin{equation}\label{sec6:typeB-mass}
\sup_t\mathbf M(\Phi_k^{\mathrm{far}}(t))
\leq S(r_k)+2\sum_j\sigma_j+\sum_j e_{k,j}
<S(r_k)+\frac{\varepsilon}{4}.
\end{equation}
This family has no
concentration of mass: its pieces lie on $M_k, N_k$
and finitely many cutting walls $H_k\cap \partial B_{k,j}$, or are images of minimizing
leaves in $\mathcal{M}_\alpha$, with uniform local small-ball area bounds.

\subsection{Step 2: Interpolation inside Type A domain using Marques--Neves interpolation}
\label{sec6:typeA}
We now construct the discrete replacements inside the thin region $H_{k}^{\text{thin}}$. We need two requirements: every partial replacement must have
small lateral boundary, and every individual increment must have
small full boundary mass. If $H_{k}^{\text{thin}}$ has zero volume, then
$T_k^{\mathrm{far}}=N_k$ and we are done. Otherwise, we will partition $H_{k}^{\text{thin}}$ into small balls.

\medskip
\noindent Choose a maximal $L_\varepsilon$-separated set of centers $\{x_{k,i}\}_i\subset H_{k}^{\text{thin}}$.
Equivalently, $H_{k}^{\text{thin}}\subset \cup_{i} B^0_{L_\varepsilon}(x_{k,i})$, and every two $B^0_{2L_\varepsilon}(x_{k,i}), B^0_{2L_\varepsilon}(x_{k,j})$ have uniformly
bounded overlap. 

Since the metric is periodic, it follows from the coarea formula and bounded overlap that we can choose radii
$s_{k,i}\in(L_\varepsilon,2L_\varepsilon)$ such that
\begin{equation}\label{sec6:coarse-cut}
\sum_i\mathcal H_g^{n-1}
\bigl(H_{k}^{\text{thin}}\cap\partial B^0_{s_{k,i}}(x_{k,i})\bigr)
\leq\frac{C_gV}{L_\varepsilon}<\frac{\delta_\varepsilon}{16}.
\end{equation}

\medskip
Define the following pairwise disjoint sets
\[
D_{k,i}
:=
B^0_{s_{k,i}}(x_{k,i})
\setminus
\bigcup_{q<i}B^0_{s_{k,q}}(x_{k,q}),
\]
and cover $H_{k}^{\text{thin}}\cap D_{k,i}$ by flat $a_\varepsilon$-balls with
$a_\varepsilon$-separated centers. Choose radii in $(a_\varepsilon,2a_\varepsilon)$ by coarea, and
subtract preceding balls. Bounded overlap and~\eqref{sec6:tube}
bound the sum of the new cutting areas inside $H_{k}^{\text{thin}}\cap D_{k,i}$ by
\begin{equation}\label{sec6:fine-cut}
\frac{C_g}{a_\varepsilon}\operatorname{Vol}_g
\bigl(H_{k}^{\text{thin}}\cap B^0_{3L_\varepsilon}(x_{k,i})\bigr)
\leq\frac{C_gh_\varepsilon(L_\varepsilon+1)^{n-1}}{a_\varepsilon}
<\frac{\delta_\varepsilon}{16}.
\end{equation}
The $3L_\varepsilon$-ball contains all relevant fine balls, because $2a_\varepsilon<L_\varepsilon$.
The last inequality follows from~\eqref{sec6:choices}.
Each resulting piece has closure in a flat ball of radius $2a_\varepsilon$.
Carry out this subdivision in each set $H_{k}^{\text{thin}}\cap D_{k,i}$, ordering
the resulting pieces so that one such set is completely filled
before the next is started. Denote the ordered partition by $Q_{k,1},\ldots,Q_{k,J_k}$.

\medskip

Set
\[
P_{k,0}=H_k\cap U_k,\qquad
P_{k,j}=P_{k,0}\cup\bigcup_{\ell=1}^jQ_{k,\ell},\qquad
T_{k,j}=M_k+\partial\llbracket P_{k,j}\rrbracket.
\]
Then $T_{k,0}=T_k^{\mathrm{far}}$ and $T_{k,J_k}=N_k$.
At each stage, the lateral boundary of $P_{k,j}$ lies on the
Type B cutting walls, the cutting boundaries of the initial
partition, and the subdivision boundaries inside at most one
partially filled set $H_{k}^{\text{thin}}\cap D_{k,i}$. Indeed, once all pieces in
$H_k^{\mathrm{thin}}\cap D_{k,i}$ have been added,
additivity implies that all internal
subdivision faces cancel in pairs with the opposite
orientations.

Consequently, by~\eqref{sec6:typeB-wall}, \eqref{sec6:coarse-cut},
and~\eqref{sec6:fine-cut},
\begin{equation}\label{sec6:partial-wall}
\mathbf M(W(P_{k,j}))
<\frac{\delta_\varepsilon}{32}
+\frac{\delta_\varepsilon}{16}
+\frac{\delta_\varepsilon}{16}
<\frac{\delta_\varepsilon}{4}.
\end{equation}
It follows from the lateral comparison~\eqref{sec6:lateral} that
\begin{equation}\label{sec6:discrete-mass}
S(r_k)\leq\mathbf M(T_{k,j})
<S(r_k)+\frac{\delta_\varepsilon}{2}
\qquad(0\leq j\leq J_k).
\end{equation}

\medskip
\noindent
The full boundary of a fine piece consists of portions of $M_k,N_k$
in its $2a_\varepsilon$-ball and portions of the same three collections of
cutting walls. Thus~\eqref{sec6:area} gives
\begin{equation}\label{sec6:fineness}
\begin{aligned}
\mathbf M(T_{k,j}-T_{k,j-1})
&=\mathbf M(\partial\llbracket Q_{k,j}\rrbracket)\\
&\leq C_ga_\varepsilon^{n-1}+\frac{\delta_\varepsilon}{32}
+\frac{\delta_\varepsilon}{16}
+\frac{\delta_\varepsilon}{16}
<\frac{\delta_\varepsilon}{2}<\delta_0.
\end{aligned}
\end{equation}

We now apply Lemma~\ref{sec6:localized} to each
$\llbracket Q_{k,j}\rrbracket$, with $\delta=\delta_\varepsilon$.
The interpolating path has the exact endpoints
\[
\Psi_{k,j}(0)=T_{k,j-1},\qquad
\Psi_{k,j}(1)=T_{k,j},
\]
and its quantitative mass estimate gives
\[
\sup_{t\in[0,1]}\mathbf M(\Psi_{k,j}(t))
\leq\mathbf M(T_{k,j-1})+C_0\delta_\varepsilon.
\]
Concatenating these paths yields a mass-continuous family $\Phi_k^{\mathrm{close}}$ from
$T_k^{\mathrm{far}}$ to $N_k$. Since every prescribed discrete
cycle satisfies~\eqref{sec6:discrete-mass},
\begin{equation}\label{sec6:typeA-mass}
\begin{aligned}
\sup_t\mathbf M(\Phi_k^{\mathrm{close}}(t))
&\leq\max_{0\leq j<J_k}\mathbf M(T_{k,j})+C_0\delta_\varepsilon\\
&<S(r_k)+\left(\frac12+C_0\right)\delta_\varepsilon\\
&\leq S(r_k)+(2+C_0)\delta_\varepsilon
\leq S(r_k)+\frac{\varepsilon}{4}.
\end{aligned}
\end{equation}

\noindent\emph{Completion of the totally irrational case.}
Concatenating $\Phi_k^{\text{far}}$ and $\Phi_k^{\text{close}}$ gives a family $\Phi_k$ from
$M_k$ to $N_k$ with maximum mass at most $S(r_k)+\varepsilon$.
The fillings constructed in Step~1 and Lemma \ref{sec6:localized} yield
compact integral $n$-currents $A_k(t)$, continuous in mass,
such that
\begin{equation}\label{sec6:fillings}
\Phi_k(t)=M_k+\partial A_k(t),
\qquad
A_k(0)=0,
\qquad
A_k(1)=\llbracket H_k\rrbracket.
\end{equation}
Hence $\Phi_k$ is flat-continuous. It has no concentration
of mass: this was verified in Step~1, while the family in
Step~2 is continuous in mass.
For a fixed $k$, the construction is supported in a common
compact subset of $X_k$, so projection preserves this
property.

The projected endpoints agree, so $(\pi_k)_\#\Phi_k$ is a loop in
$\mathcal Z_{n-1}(\mathbb T^n,r_k)$. In particular, each projected cycle has
mass at least $S(r_k)$. For a fine subdivision, the projected increments of
$A_k$ are the small fillings of the increments of the loop, so its Almgren
class is
\[
\bigl[(\pi_k)_\#(A_k(1)-A_k(0))\bigr]
=\bigl[(\pi_k)_\#\llbracket H_k\rrbracket\bigr]
=[\mathbb T^n]
\]
by \eqref{sec6:strip}. Hence $(\pi_k)_\#\Phi_k\in\mathcal P_{r_k}$.

Since the construction applies to a subsequence of every
subsequence, we conclude that
\[
0\leq\omega(r_k)-S(r_k)\leq\varepsilon
\]
for all sufficiently large $k$. Letting $\varepsilon\to 0$
proves Theorem \ref{sec6:main} when $\alpha$ is totally irrational.

\subsection{Rationally dependent directions}
\label{sec6:rational}

Suppose that $\mathcal M(\alpha)$ foliates $\mathbb R^n$.
By Theorem~\ref{structure}, $\mathcal M_\alpha=\mathcal M(\alpha)$. Thus every limiting leaf is invariant
under $\Gamma=K(\alpha)$. The Type A construction is uniform
over the covers $X_k$ and only the treatment of the thick regions
requires modification.

\medskip

Suppose that $r_k\notin V(\alpha)$ along a subsequence.
Passing to a further subsequence, choose a fixed vector
$\beta$ from an integral basis of $\Gamma$ such that
$r_k\cdot\beta\neq0$ for every $k$. Then
\[
\widetilde H_k\cap\tau_\beta\widetilde H_k=\varnothing.
\]
If thick cores persisted, the recentered compactness argument
and \eqref{sec6:core-density} would give distinct limiting
leaves bounding a nonempty region $G$. Since these leaves
are $\Gamma$-invariant, $\tau_\beta G=G$, contradicting the
displayed disjointness and local convergence of the oriented
regions. Hence $H_{k}^{\text{thick}}=\varnothing$ for all sufficiently large
$k$. We can repeat Step~2 with $U_k=\varnothing$, followed by the filling
argument, to obtain the required sweep-out.

\medskip
Suppose that $r_k\in V(\alpha)$, so $\Gamma\subset\Lambda_k$,
and work in the fixed intermediate cover
\[
Y=\mathbb R^n/\Gamma
\cong
\bigl(\operatorname{span}_{\mathbb R}\Gamma/\Gamma\bigr)
\times V(\alpha),
\qquad
Y\longrightarrow X_k.
\]
The nonzero periods in $\Lambda_k/\Gamma$ escape
every compact subset of $Y$: otherwise, reducing modulo
$\Gamma$ would give bounded integral representatives
$z_k\notin\Gamma$ with $r_k\cdot z_k=0$. A constant
subsequence would satisfy $\alpha\cdot z=0$, hence
$z\in\Gamma$, a contradiction. Thus fixed compact subsets
of $Y$ embed in $X_k$ for sufficiently large $k$.

Repeat the previous core selection, grouping, and recentering in $Y$.
The limiting leaves and the foliation between them descend
to $Y$. Each limiting region $G\subset Y$ satisfies
$\operatorname{Vol}_g(G)\leq V$, by local convergence,
the preceding embeddings, and compact exhaustion.

Let $f$ be the flat distance from the origin in the
noncompact factor of $Y$. Coarea gives
\[
\int_0^\infty
\mathcal H_g^{n-1}(G\cap\{f=s\})\,ds
\leq C_g\operatorname{Vol}_g(G)<\infty.
\]
Hence the domains $\{f<R\}$ provide arbitrarily large
admissible cuts satisfying \eqref{sec6:wall-budget} and
containing the corresponding group of cores with a margin.
Their radii are fixed before taking $k$ large, and separation
of the groups makes their images in $X_k$ disjoint.

  In Step~1 we compare leaves inside domains of $Y$, which requires the
  following quotient version of area-minimizing property.

\begin{Lemma}\label{sec6:quotient-min}
Let $\Gamma\subset\mathbb Z^n$ be a sublattice
and $Y=\mathbb R^n/\Gamma$. Suppose
$\varphi=\partial\llbracket E\rrbracket$
is area-minimizing in $\mathbb R^n$ and
$\tau_\gamma E=E$ for every $\gamma\in\Gamma$.
Let $\varphi_Y$ be its quotient in $Y$.
Then, for every compactly supported integral
$n$-current $P$ in $Y$ and every bounded open
set $U\supset\operatorname{spt}P$, we have
\[
\mathbf M(\varphi_Y\llcorner U)
\leq
\mathbf M((\varphi_Y+\partial P)\llcorner U).
\]
\end{Lemma}

\begin{proof}
We may assume that $d=\operatorname{rank}\Gamma\geq1$.
Suppose, to the contrary, that there exist \(P\) and \(U\) as above such that
\[
\mathbf M((\varphi_Y+\partial P)\llcorner U)
\le \mathbf M(\varphi_Y\llcorner U)-\eta
\]
for some \(\eta>0\).
Write $\mathbb R^n=\operatorname{span}_{\mathbb R}\Gamma\oplus W$
orthogonally, fix a fundamental domain $D_0$ of $\Gamma$ in
$\operatorname{span}_{\mathbb R}\Gamma$, translated so that the cell
boundaries carry no mass of $\varphi$ or of the lift of $\partial P$, and let
$\widetilde P$, $\widetilde U$ be the $\Gamma$-periodic lifts of $P$, $U$, both
lie in $\{|x_W|\leq R_0\}$ for some $R_0$. Put
$Z_\rho=\{|x_{\operatorname{span}\Gamma}|<\rho\}$. Slicing $\widetilde P$ by the
$1$-Lipschitz function $|x_{\operatorname{span}\Gamma}|$, for every large $R$
there is $\rho\in(R,R+1)$ with
\[
\partial(\widetilde P\llcorner Z_\rho)
=(\partial\widetilde P)\llcorner Z_\rho+S_\rho,
\qquad
\mathbf M(S_\rho)\leq\|\widetilde P\|(Z_{R+1}\setminus Z_R)\leq CR^{d-1}.
\]
The current $\varphi+\partial(\widetilde P\llcorner Z_\rho)$ is a compactly
supported modification of $\varphi$ inside the bounded set
$\widetilde U\cap Z_{\rho+1}$. On each of the at least $cR^d$ cells
$(\gamma+D_0)\times W$ contained in $Z_\rho$ it agrees with the lift of
$\varphi_Y+\partial P$, so its mass there is smaller than that of $\varphi$
by $\eta$. At most $CR^{d-1}$ cells meet $\partial Z_\rho$, and on each of
them the masses differ by at most
$\mathbf M(\partial P)+\mathbf M(S_\rho\llcorner\text{cell})$. Hence we have:
\[
\mathbf M\bigl((\varphi+\partial(\widetilde P\llcorner Z_\rho))\llcorner
(\widetilde U\cap Z_{\rho+1})\bigr)
-\mathbf M\bigl(\varphi\llcorner(\widetilde U\cap Z_{\rho+1})\bigr)
\leq-c\eta R^d+CR^{d-1}<0
\]
for large $R$, contradicting the minimality of $\varphi$ in $\mathbb R^n$.
\end{proof}

We can now perform the construction and Step~1 in these
compact domains in $Y$, and Step~2 in $X_k$.   Since the leaves are $\Gamma$-invariant, Lemma~\ref{sec6:quotient-min}
  shows that the leaf comparison \eqref{sec6:leaf-comparison} holds in $Y$. The previous
mass estimates and filling identity
\eqref{sec6:fillings} give \eqref{sec6:target} and the
fundamental projected Almgren class.
This
completes the proof of Theorem~\ref{sec6:main}.

\begin{Rem}
    It follows from this result that $\Delta W_\alpha>0$ is an obstruction to the existence of a minimal foliation in $\mathcal{M}_\alpha$. A more direct geometric obstruction to minimal foliation on $\mathbb{T}^n$ is the existence of a closed, contractible minimal hypersurface. This follows from the maximum principle and the fact that the lift of a closed contractible hypersurface to $\mathbb{R}^n$ is closed. A related inverse question on $\mathbb{T}^2$ was analyzed in \cite{zeroentropy}, where the authors proved the existence of foliations by minimizing geodesics in any direction under a stronger assumption of zero topological entropy. The absence of a contractible geodesic under zero topological entropy was proved in \cite{contractiblegeodesic}.
\end{Rem}

\section{Sufficient condition for minimal foliation}
\label{section7}\label{sec:s7-sufficiency}

We prove the implication $(2)\Rightarrow(3)$ in
Theorem~\ref{maintheorem1}. The idea is to use the
lattice inequality to replace nearly minimizing
sweep-outs by finite nested families of nearly
minimizing cycles with prescribed volume increments.
Passing to the limit yields a nested family of
area-minimizing hypersurfaces, and the volume
condition ensures that these hypersurfaces form
a foliation. This construction does not require
total irrationality. For rationally dependent
directions, the resulting foliation need
not consist of the most periodic minimizers. We first restate the equivalent form of the theorem.

\begin{Th}\label{suffient}\label{thm:s7-sufficiency}
Let $3\leq n\leq7$, let $0\neq\alpha\in H_{n-1}(\mathbb T^n,\mathbb R)$, and let
$\{r_k\}$ be primitive integer classes converging to $\alpha$ in the sense
of Definition \ref{convergencestable}. If
\[
\omega(r_k)-S(r_k)\longrightarrow0,
\]
then $\mathbb R^n$ admits a deck-invariant foliation by area-minimizing
hypersurfaces in $\mathcal M_\alpha$. This foliation descends to a minimal
foliation of $\mathbb T^n$. The construction preserves common integer periods:
if $r_k\cdot z=0$ for all sufficiently large $k$, then every leaf can be
chosen $\tau_z$-invariant, simultaneously for all such $z$.
If $\alpha$ is totally irrational, the foliation consists of all the
elements of $\mathcal M_\alpha$.
\end{Th}

Fix a primitive class $r$ and an integral minimizing cycle $T_r$ with
$\mathbf{M}(T_r)=S(r)$. We use the codimension-one equality between
integral and real homological mass minimization recalled in
Section~\ref{section4}. There is an
integer-valued function
$u_r\in BV_{\mathrm{loc}}(\mathbb R^n)$ such that
\begin{equation}\label{eq:s7-quasiperiod}
\partial\llbracket u_r\rrbracket=\pi^\#T_r,
\qquad
u_r(x+z)=u_r(x)-r\cdot z
\quad(z\in\mathbb Z^n)
\end{equation}
for almost every $x$. For any integer-valued $u\in BV_{\mathrm{loc}}(\mathbb R^n)$ with the
same additive periods, denote by $T_u$ the integral cycle on $\mathbb T^n$ whose
full lift is $\partial\llbracket u\rrbracket$.
The following lattice inequality will be useful later:
\begin{equation}\label{eq:s7-lattice}
\mathbf{M}(T_{u\wedge v})+\mathbf{M}(T_{u\vee v})
\leq \mathbf{M}(T_u)+\mathbf{M}(T_v),
\end{equation}
where $u\wedge v=\min\{u,v\}$ and $u\vee v=\max\{u,v\}$. See, for example, in \cite{DePhilippisMaggi2015}.

We first establish a quantitative nesting lemma
for nearly minimizing sweep-outs.

\begin{Lemma}\label{lem:s7-ordering}
Suppose that an $r$-sweep-out $\Psi$ satisfies
\[
\sup_t\mathbf{M}(\Psi(t))\leq S(r)+\varepsilon.
\]
Then for every positive integer $N$, there are sets of finite perimeter:
\[
\varnothing=A_0\subset A_1\subset\cdots\subset A_N=\mathbb T^n
\]
in $\mathbb T^n$ such that for $1\leq j\leq N$, we have
\begin{equation}\label{eq:s7-ordered-mass}
\operatorname{Vol}_g(A_j\setminus A_{j-1})=\frac VN,
\qquad
\mathbf{M}\bigl(T_r+\partial\llbracket A_j\rrbracket\bigr)
\leq S(r)+N\varepsilon.
\end{equation}
\end{Lemma}

\begin{proof}
Regard $\Psi$ as a flat-continuous loop on $[0,1]$. It lifts to a
mass-continuous path of integral $n$-currents $Q(t)$ on the torus with
\[
\partial Q(t)=\Psi(t)-\Psi(0),\qquad Q(0)=0,
\qquad Q(1)=\llbracket \mathbb T^n\rrbracket.
\]
Choose an integral filling of $\Psi(0)-T_r$ and use its density,
together with the density of $Q(t)$, to obtain potentials $u(t)$ with
\[
T_{u(t)}=\Psi(t),\qquad u(1)=u(0)+1.
\]
Their differences vary continuously in $L^1(\mathbb T^n)$. We can extend them to all
$t\in\mathbb R$ by $u(t+m)=u(t)+m$, $m\in\mathbb Z$. 
Define the truncation:
\[
w(t)=\max\{u_r,\min\{u(t),u_r+1\}\}.
\]
This truncation does not increase the mass. Because the potentials are
integer-valued, there is a finite-perimeter set $A(t)\subset \mathbb T^n$ with
\[
w(t)=u_r+\chi_{\pi^{-1}A(t)},\qquad
\mathbf{M}\bigl(T_r+\partial\llbracket A(t)\rrbracket\bigr)
\leq S(r)+\varepsilon.
\]
The truncation is pointwise $1$-Lipschitz in $u(t)$, so
$t\mapsto\chi_{A(t)}$ is $L^1$-continuous. Along integer values of $t$,
these characteristic functions tend to $0$ as $t\to-\infty$ and to $1$
as $t\to+\infty$, by the dominated convergence theorem.

For any finite-perimeter sets $A,B\subset \mathbb T^n$, it follows from
\eqref{eq:s7-lattice} and the lower bound $S(r)$ for the intersection
term that
\begin{equation}\label{eq:s7-union}
\begin{split}
\mathbf{M}\bigl(T_r+\partial\llbracket A\cup B\rrbracket\bigr)-S(r)
\leq{}&\mathbf{M}\bigl(T_r+\partial\llbracket A\rrbracket\bigr)-S(r)\\
&+\mathbf{M}\bigl(T_r+\partial\llbracket B\rrbracket\bigr)-S(r).
\end{split}
\end{equation}
Suppose that $A_{j-1}$ has volume $(j-1)V/N$, where $1\leq j<N$.
The continuous function
\[
t\longmapsto\operatorname{Vol}_g\bigl(A(t)\setminus A_{j-1}\bigr)
\]
has values approaching $0$ and $V-\operatorname{Vol}_g(A_{j-1})$.
By the intermediate value theorem, we can choose $t_j$ such that $\operatorname{Vol}_g(A(t_j)\setminus A_{j-1})=V/N$. Put $A_j=A_{j-1}\cup A(t_j)$. Starting with
$A_0=\varnothing$ and using \eqref{eq:s7-union} inductively gives
\[
\operatorname{Vol}_g(A_j)=\frac{jV}{N},\qquad
\mathbf{M}\bigl(T_r+\partial\llbracket A_j\rrbracket\bigr)
\leq S(r)+j\varepsilon.
\]
Finally set $A_N=\mathbb T^n$. Its last volume increment is $V/N$, and its
associated cycle is $T_r$. The proof is complete.
\end{proof}

\begin{Lemma}\label{lem:s7-local}
Let $u_k\in BV_{\mathrm{loc}}(\mathbb R^n;\mathbb Z)$ such that 
\[u_k(x+z)=u_k(x)-r_k\cdot z, \quad \forall x\in \mathbb{R}^n, z\in \mathbb{Z}^n.\]
Assume that
\[
\mathbf{M}(T_{u_k})\leq S(r_k)+\delta_k,
\qquad \delta_k\longrightarrow0^+.
\]
For arbitrary integers $\ell_k$, put
\[
E_k=\{u_k\geq\ell_k\},\qquad S_k=\partial\llbracket E_k\rrbracket.
\]
The masses of $S_k$ are uniformly bounded on compact subsets of
$\mathbb R^n$. Taking a subsequence,
$\chi_{E_k}\to\chi_E$ in $L^1_{\mathrm{loc}}$, and
$S=\partial\llbracket E\rrbracket$ is area-minimizing. Furthermore, if $S\neq0$, then $\operatorname{spt}S$ is connected and smooth when $3\leq n\leq7$.
If, in addition, $r_k\to\alpha\neq0$ in the sense of
Definition \ref{convergencestable}, then $S\in \mathcal{M}_\alpha$.
\end{Lemma}

\begin{proof}
The coarea formula gives
\[
\partial\llbracket u_k\rrbracket
=\sum_{\ell\in\mathbb Z}\partial\llbracket\{u_k\geq\ell\}\rrbracket,
\qquad
\bigl\|\partial\llbracket u_k\rrbracket\bigr\|
=\sum_{\ell\in\mathbb Z}
 \bigl\|\partial\llbracket\{u_k\geq\ell\}\rrbracket\bigr\|.
\]
Fix a bounded ball $B\subset\mathbb R^n$. We can choose an integer $m$ so large
that $B$ can be embedded in
\[
p_m:T_m^n=\mathbb R^n/m\mathbb Z^n\longrightarrow \mathbb T^n,
\]
and let $d=m^n$ be the degree. Denote the full finite-cover lift of
$T_{u_k}$ by $\widehat T_k$, and identify $B$ with its image in $T_m^n$.
Then $\mathbf{M}(\widehat T_k)=d\mathbf{M}(T_{u_k})$ and
$(p_m)_\#\widehat T_k=dT_{u_k}$. Inside $B$, the coarea decomposition
implies
\[
\mathbf{M}\bigl((\widehat T_k-S_k)\llcorner B\bigr)
=\mathbf{M}(\widehat T_k\llcorner B)-\mathbf{M}(S_k\llcorner B).
\]
Consequently, for every integral $n$-current $Q$ compactly supported in
$B$, the triangle inequality gives
\[
\begin{split}
\mathbf{M}(\widehat T_k+\partial Q)
\leq{}&d\mathbf{M}(T_{u_k})-\mathbf{M}(S_k\llcorner B)\\
&+\mathbf{M}\bigl((S_k+\partial Q)\llcorner B\bigr).
\end{split}
\]
The projected current represents the homology class $dr_k$. Since projection cannot increase
mass, we have:
\[
dS(r_k)=S(dr_k)
\leq\mathbf{M}\bigl((p_m)_\#(\widehat T_k+\partial Q)\bigr)
\leq\mathbf{M}(\widehat T_k+\partial Q).
\]
It follows that
\begin{equation}\label{eq:s7-local-comparison}
\mathbf{M}(S_k\llcorner B)
\leq\mathbf{M}\bigl((S_k+\partial Q)\llcorner B\bigr)+d\delta_k.
\end{equation}

Choose an admissible Euclidean ball $B_\rho\Subset B$ and take
$Q=-\llbracket E_k\cap B_\rho\rrbracket$. For almost every $\rho$,
the restriction formula replaces $S_k$ inside $B_\rho$ by the trace of
$E_k$ on $\partial B_\rho$, with unchanged exterior. Choosing also
$\|S_k\|(\partial B_\rho)=0$ and cancelling that exterior in
\eqref{eq:s7-local-comparison} gives
\begin{equation}\label{eq:s7-local-bound}
\mathbf{M}(S_k\llcorner B_\rho)
\leq\mathcal H_g^{n-1}(\partial B_\rho)+d\delta_k.
\end{equation}
The radius can be chosen admissible for every $k$. Enlarging a ball
slightly gives a uniform mass bound on any prescribed compact
set. Compactness gives the convergence
$\chi_{E_k}\to\chi_E$ in $L^1_{\mathrm{loc}}$ along a subsequence.

We now verify the minimizing property of the limit $S$. Since $\chi_{E_k}\to\chi_E$ in $L^1_{\mathrm{loc}}(\mathbb R^n)$,
we may pass to a further subsequence, still indexed by $k$, such that
\[
    \sum_{k=1}^{\infty}
    \int_{B_m}|\chi_{E_k}-\chi_E|\,dV_g<\infty
    \qquad\text{for every }m\in\mathbb N,
\]
where $\{B_m\}_{m\geq1}$ is a fixed increasing exhaustion
of $\mathbb R^n$ by bounded balls. Let $Q$ be a compactly supported integral
$n$-current. Choose $\rho$ with $\operatorname{spt}Q\Subset B_\rho$,
and a fixed larger ball $B$, so that all restrictions are admissible,
$\|S_k\|(\partial B_\rho)=\|S\|(\partial B_\rho)=0$, and the trace
current
\[
Z_k=(\chi_E-\chi_{E_k})\,\partial\llbracket B_\rho\rrbracket
\]
satisfies $\mathbf{M}(Z_k)\to0$. Use
\[
Q_k=Q+\llbracket E\cap B_\rho\rrbracket
       -\llbracket E_k\cap B_\rho\rrbracket
\]
in \eqref{eq:s7-local-comparison} on $B$. Inside $B_\rho$, the modified
boundary is $S+\partial Q$. Outside it is $S_k$, and its additional
spherical term is $Z_k$. 
Since the competitor agrees with $S_k$ outside
$B_\rho$, the almost-minimizing inequality,
after subtracting the common exterior mass, gives:
\[
\mathbf{M}(S_k\llcorner B_\rho)
\leq\mathbf{M}\bigl((S+\partial Q)\llcorner B_\rho\bigr)
   +\mathbf{M}(Z_k)+d\delta_k.
\]
Taking $k\to \infty$ and using lower-semicontinuity of mass gives
\[\mathbf{M}(S\llcorner B_\rho)
\leq\mathbf{M}\bigl((S+\partial Q)\llcorner B_\rho\bigr).\]
Therefore, $S$ is area-minimizing. 

We next verify the Birkhoff property. For every
$z\in\mathbb Z^n$, the quasiperiodicity of $u_k$ gives
\[
\tau_z E_k
=\{u_k\geq \ell_k-r_k\cdot z\}.
\]
Thus $E_k$ and $\tau_z E_k$ are nested, and this
property passes to the limit. Moreover, if
$\alpha\cdot z>0$, then $r_k\cdot z>0$ for all
sufficiently large $k$, so
\[
E_k\subset\tau_z E_k.
\]
Passing to the limit gives $E\subset\tau_z E$.
Similarly, $\tau_z E\subset E$ when
$\alpha\cdot z<0$. Hence $E$ has the Birkhoff
property.

If $S\neq0$, Corollary~4.16 of
\cite{ChambolleGoldmanNovaga2014} implies that
$\operatorname{spt}S$ is connected.
If, in addition, $r_k\to\alpha\neq0$ in the sense of
Definition~1.1, the orientation convention
\eqref{birkoff1}--\eqref{birkoff2} shows that
$S\in\mathcal M_\alpha$.
The proof is complete.
\end{proof}

\begin{proof}[Proof of Theorem~\ref{thm:s7-sufficiency}]
Since $\Delta W_{r_k}\to 0$, we can pass to a subsequence and re-index so that $\omega(r_k)-S(r_k)<1/(2k^3)$. We then
choose the corresponding near-optimal $r_k$-sweep-outs $\Psi_k$ satisfying
\[
\sup_t\mathbf{M}(\Psi_k(t))\leq S(r_k)+k^{-3}.
\]
Applying Lemma~\ref{lem:s7-ordering} with $N=k$, we obtain nested sets $A_{k,j}$,
$0\leq j\leq k$, and put
\[
v_{k,j}=u_{r_k}+\chi_{\pi^{-1}A_{k,j}},\qquad
E_{k,j,\ell}=\{v_{k,j}\geq\ell\}\quad(\ell\in\mathbb Z).
\]
Then
\begin{equation}\label{eq:s7-small-excess}
0\leq\mathbf{M}(T_{v_{k,j}})-S(r_k)\leq k^{-2}.
\end{equation}

For each $k$, all the sets $E_{k,j,\ell}$ lie in a single chain, since
\[
E_{k,0,\ell}\subset\cdots\subset E_{k,k,\ell}
=E_{k,0,\ell-1}.
\]
Every integer translation preserves this relation:
\begin{equation}\label{eq:s7-deck-action}
\tau_zE_{k,j,\ell}=E_{k,j,\ell-r_k\cdot z}.
\end{equation}
If $\pi$ is injective on a ball $U$, then
\begin{equation}\label{eq:s7-volume-jump}
\operatorname{Vol}_g\bigl((E_{k,j,\ell}\setminus E_{k,j-1,\ell})\cap U\bigr)
\leq\operatorname{Vol}_g(A_{k,j}\setminus A_{k,j-1})=\frac Vk.
\end{equation}

Let $\{U_a\}_{a\geq1}$ be a countable basis of balls in $\mathbb R^n$,
each projecting injectively to $\mathbb T^n$. The volume jump bound lets us
choose, for every $k,a$, a set $E_{k,a}$ such that
\begin{equation}\label{eq:s7-bisection}
\left|\operatorname{Vol}_g(E_{k,a}\cap U_a)
       -\frac12\operatorname{Vol}_g(U_a)\right|\leq\frac Vk.
\end{equation}
Lemma~\ref{lem:s7-local}, \eqref{eq:s7-small-excess}, and a diagonal subsequence
extraction give, for all $a$ simultaneously,
\[
\chi_{E_{k,a}}\longrightarrow\chi_{E_a}
\quad\text{in }L^1_{\mathrm{loc}},\qquad
S_a=\partial\llbracket E_a\rrbracket\quad\text{area-minimizing}.
\]
The volume condition passes to the limit:
\[
\operatorname{Vol}_g(E_a\cap U_a)
=\tfrac12\operatorname{Vol}_g(U_a).
\]
Therefore, the relative isoperimetric inequality implies that $S_a$ is
nonzero and that its support meets $U_a$.

All sets $E_{k,a}$ and their integer translates belong to the same nested sequence.
Given any two limiting sets, choose a subsequence along which one of
the two possible orders holds; their $L^1_{\mathrm{loc}}$ limits have
that same order. Hence the $E_a$ and all their integer translates are
pairwise nested.

By Lemma~\ref{lem:s7-local}, each $E_a$ has the directional
Birkhoff property.

Let $\mathscr E$ be the $L^1_{\mathrm{loc}}$ closure of
\[
\{\tau_zE_a:a\geq1,\ z\in\mathbb Z^n\}.
\]
For an exactly minimizing boundary, the comparison producing
\eqref{eq:s7-local-bound} has zero error and gives a mass bound on each
compact set independent of the set. BV compactness makes
$\mathscr E$ compact. Lemma~\ref{lem:s7-local} implies that every $S= \partial\llbracket E\rrbracket$ is connected and area-minimizing, hence belongs to $\mathcal{M}_\alpha$. Nesting and the directional order
also pass to the closure, and $\mathscr E$ is deck-invariant. The maximum principle implies that these
boundaries are pairwise disjoint or coincident.

By compactness of $\mathscr E$ and the uniform two-sided
volume-density estimate for area-minimizing sets, the union
of these boundary supports is closed. Since it meets every
basis ball $U_a$, it equals $\mathbb R^n$.

After removing repetitions, these boundaries form the required
foliation.

Finally, when $\alpha$ is totally irrational, any two members of
$\mathcal M_\alpha$ are disjoint. Therefore, the foliation consists of all the elements of $\mathcal{M}_\alpha$.
\end{proof}

The construction preserves common periodicities. If $z\in\mathbb Z^n$
satisfies $r_k\cdot z=0$ for all sufficiently large $k$, then
\eqref{eq:s7-deck-action} gives $\tau_zE_{k,a}=E_{k,a}$. Hence, this also holds for all $E \in \mathscr E$.
Consequently every leaf of the constructed foliation is
$\tau_z$-invariant. 

\begin{Rem}\label{rem:s7-most-periodic}
The vanishing of the barriers along one approximating sequence does not imply a foliation
by the most periodic minimizers $\mathcal M(\alpha)$, as shown in the example below. However, the vanishing along every sequence does give this conclusion. Indeed, the
subspace $V(\alpha)=\bigl(\operatorname{span}_{\mathbb R}K(\alpha)\bigr)^\perp$
is rational, so one can choose primitive integer classes
$r_k\in V(\alpha)$ converging to $\alpha$.
For this sequence, every element of $K(\alpha)$ is a common period.
If $\Delta W_{r_k}\to 0$, the foliation constructed above consists
of $K(\alpha)$-invariant minimizers, hence of elements of
$\mathcal M(\alpha)$. Therefore
$\mathcal M(\alpha)$ itself is a foliation.
Together with Theorem~\ref{maintheorem3}, this proves
\[
\mathcal M(\alpha)\text{ is a foliation}
\quad\Longleftrightarrow\quad
\begin{gathered}
\Delta W_{r_k}\longrightarrow0\text{ for every}\\
\text{primitive sequence }r_k\to\alpha.
\end{gathered}
\]
\end{Rem}

\begin{Ex}\label{everydirectionexceptone}
On $\mathbb T^n$, with coordinates $(x,y,z_3,\ldots,z_n)$,
consider
\[
    g=f(y)(dx^2+dy^2)+\sum_{i=3}^n dz_i^2,
\]
where $f$ is smooth, positive, and $1$-periodic, with a unique
global minimum modulo $\mathbb Z$ at $y=0$. For example, take
$f(y)=1+\sin^2(\pi y)$. Let $\alpha$ be the class of
$\{y=0\}$, oriented by $dx\wedge dz_3\wedge\cdots\wedge dz_n$.

The closed form
\[
    \Omega_0=\sqrt{f(0)}\,dx\wedge dz_3\wedge\cdots\wedge dz_n
\]
has comass $\sqrt{f(0)/f(y)}\leq1$ and calibrates $\{y=0\}$.
Equality in the calibration inequality forces every minimizing
representative of $\alpha$ to be this coordinate torus.
Consequently, $S(\alpha)=\sqrt{f(0)}$ and, at the level of supports,
\[
    \mathcal M(\alpha)=\bigl\{\{y=j\}:j\in\mathbb Z\bigr\},
\]
which is not a foliation.

Now let $r$ be a primitive class not parallel to $\alpha$, and
choose a connected minimizing representative $\Sigma_r$.
There is an isometric coordinate circle action whose orbits
have nonzero intersection pairing with $r$. Its translates of
$\Sigma_r$ are homologous minimizers, hence are disjoint or
coincide. Thus the distinct translates foliate
$\mathbb T^n$. Traversing the leaf space once gives a fundamental
$r$-sweep-out of constant mass $S(r)$, so
\[
    \Delta W_r=0
    \qquad\text{for every primitive }r\notin\mathbb R\alpha.
\]

Choose a coordinate class $\beta$ independent of $\alpha$ and
set $r_k=k\alpha+\beta$. These classes are primitive, converge
in direction to $\alpha$, and satisfy $\Delta W_{r_k}=0$,
although $\mathcal M(\alpha)$ is not a foliation.
In contrast, the period-preserving conclusion of
Theorem~\ref{suffient}, applied to the constant sequence,
implies $\Delta W_\alpha>0$.
\end{Ex}


\section{Proof of Theorem \ref{maintheorem2}}\label{section8}
We want to construct a metric $g$, a rationally dependent class $\alpha$, and a sequence of integer classes $r_k\to \alpha$ such that:
\begin{enumerate}
    \item There exists a foliation of $\mathbb{R}^n$ by elements in $\mathcal{M}_\alpha$.
    \item The sequence $\Delta W_{r_k}=\omega(r_k)-S(r_k)$ is bounded away from zero.
\end{enumerate}

Roughly speaking, we want a foliation by elements of $\mathcal{M}_\alpha$ containing at least one leaf that cannot be approximated by any sequence of closed minimizers in $r_k$. 

We first need the following construction of Bangert (cf. \cite{bangertgap}, Theorem 3.1) of a metric where no minimal foliation exists.
\begin{Prop}\label{nofoliation}
    Suppose that $g_{flat}$ is the Euclidean metric on $\mathbb{R}^n$ and $V\subset \mathbb{R}^n$ is a $\mathbb{Z}^n$-periodic, non-empty open set. There exists a Riemannian metric $g$ that coincides with $g_{flat}$ outside $V$ such that $\mathbb{R}^n$ does not admit any $C^0$ foliation by minimal hypersurfaces. 
\end{Prop}
\begin{proof}
    (Sketch). The idea is to glue $g_{flat}$ to a metric that becomes very large on a region $V$ while agreeing with $g_{flat}$ elsewhere. 
    
Choose a sufficiently small Euclidean ball
$B(x,2r)\subset V$ whose integer translates
are pairwise disjoint. Let
$\varphi_0:\mathbb R^n\to[0,1]$ be a smooth
cutoff supported in $B(x,2r)$ and equal to
$1$ on $B(x,r)$. Set
\[
\varphi(y)=\sum_{z\in\mathbb Z^n}\varphi_0(y-z).
\]
Then $\varphi$ is smooth and $\mathbb Z^n$-periodic,
with support contained in $V$. Define
\[
g_\lambda=(1-\varphi)g_{\mathrm{flat}}
+\lambda^2\varphi g_{\mathrm{flat}}.
\]
    
    We claim that for $\lambda$ sufficiently large, there is no foliation of $\mathbb{R}^n$ by $g_\lambda$-minimal hypersurfaces. Assume the contrary, and let $\mathcal{F}$ be such a minimal foliation. Using the localized Lipschitz normals to the leaves (cf. \cite{solomon}), one sees that every leaf $\Sigma$ of $\mathcal{F}$ is area-minimizing with respect to $g_\lambda$. 

    Since $\mathcal{F}$ is a foliation, there is one leaf, say $\Sigma$, passing through $x$. By construction, the metric $g_\lambda$ is equal to $\lambda^2 g_{flat}$ inside $B(x,r)$. It follows from the monotonicity formula that
    \[\operatorname{Area}(\Sigma \cap B(x,r))\geq (\lambda r)^{n-1}\omega_{n-1},\]
    where $\omega_{n-1}$ is the volume of the unit ball in $\mathbb{R}^{n-1}$. 
    
    Let $S(x,3r)=\partial B(x,3r)$ be the boundary of the ball, which is divided into two parts $S_1$ and $S_2$ with the same boundary as $\Sigma\cap B(x,3r)$. As $\Sigma$ is area-minimizing, by using $S_1$ and $S_2$ as competitors, we have
    \begin{equation}
    \begin{split}
    \operatorname{Area}(\Sigma\cap B(x,3r)) &\leq \min\{\operatorname{Area}(S_1), \operatorname{Area}(S_2)\} \\
    & \leq \dfrac{1}{2}\operatorname{Area}(S(x,3r))= \dfrac{1}{2}n\omega_n(3r)^{n-1}.
    \end{split}
    \end{equation}
    As a result:
    \[\lambda^{n-1} r^{n-1}\omega_{n-1}\leq \operatorname{Area}(\Sigma\cap B(x,r))\leq \operatorname{Area}(\Sigma\cap B(x,3r))\leq \dfrac{1}{2}n\omega_n(3r)^{n-1}.\]
    By choosing $\lambda$ sufficiently large, we get a contradiction. As a result, no such $g_\lambda$-minimal foliation can exist.
\end{proof}

\textit{We can now finish the proof of Theorem \ref{maintheorem2}}: We will first construct such a metric as a product $g=h\times d\theta^2$ on $\mathbb{T}^3=\mathbb{T}^2\times S^1$. The construction can be extended to higher dimensions. According to Proposition \ref{nofoliation}, we can find a metric $h$ on the two-torus $\mathbb{T}^2$ such that no minimal foliation by geodesics exists.  

Consider the universal cover $\mathbb{R}^3=\mathbb{R}^2\times \mathbb{R}$ and fix a rationally dependent class of the form $\alpha=[\alpha_1,\alpha_2,0]$. For the torus, the homology intersection pairing \[I: H_1(\mathbb{T}^3, \mathbb{R})\times H_2(\mathbb{T}^3, \mathbb{R})\to \mathbb{R}\] is induced from the dot product. This means that for $\beta=[0,0,1]\in H_1(\mathbb{T}^3,\mathbb{Z})$ we have $I(\alpha, \beta)=0$. Therefore, for any integral current $T\in \alpha$, we have $\tau_\beta \Tilde{T}=\Tilde{T}$ for any lift $\Tilde{T}$ of $T$. 

Since the lift $\Tilde{g}$ of $g$ to $\mathbb{R}^3$ is the product of the lift of $h$ and the flat metric $d\theta^2$ on $S^1$, we see that $\Tilde{g}$ is translationally invariant in the $\beta$ direction:
\[\varphi_t^*(\Tilde{g})=\Tilde{g}, \quad \forall \text{ } t\in \mathbb{R}\]
where $\varphi_t(x)=x+t\beta$.

We first show that there is a foliation of $\mathbb{R}^3$ by the elements of \[\mathcal{M}(\alpha)\cup\mathcal{M}(\alpha, \beta).\]

Since the metric is flat in the vertical direction, any minimizer $T\in \mathcal{M}(\alpha)$ must be of the form $\gamma \times \mathbb{R}$ where $\gamma$ is a minimizer in $\mathbb{R}^2$ with respect to the direction $[\alpha_1, \alpha_2]$. Since $h$ is chosen so that there is no minimal foliation on $\mathbb{R}^2$, we see that $\mathcal{M}(\alpha)$ is not a foliation of $\mathbb{R}^3$. 

Consider a gap $G$ bounded by neighboring periodic $\Sigma^-, \Sigma^+\in \mathcal{M}(\alpha)$. By the structure theorem, Theorem \ref{structure}, we can find a heteroclinic minimizer $\Sigma\in \mathcal{M}(\alpha, \beta)$
approaching $\Sigma^-$ in the negative $\beta$ direction and $ \Sigma^+$ in the positive $\beta$ direction. Since the metric is vertically invariant, $\tau_t\Sigma$ is area-minimizing for all $t\in \mathbb{R}$. Furthermore, by the characterization of minimizers in $\mathcal{M}(\alpha, \beta)$, all these translations are disjoint. As a result, they foliate the gap $G$. Repeating this for every gap, we get a foliation of the entire space.

Next, choose a sequence of primitive integer classes
$r_k=[a_k,b_k,0]\to\alpha$ in the rational subspace $V(\alpha)$
as in Remark~\ref{rem:s7-most-periodic}. If $\alpha$ is a rational
direction, take the constant sequence given by its primitive integer
representative. In particular,
$r_k\cdot z=0$ for every $z\in K(\alpha)$ and every $k$.
We claim that
\[
\liminf_{k\to\infty}\Delta W_{r_k}>0.
\]
Otherwise, after passing to a subsequence, we have
\[
\Delta W_{r_k}=\omega(r_k)-S(r_k)\longrightarrow0.
\]
By Theorem~\ref{suffient}, there would then exist a foliation of
$\mathbb R^3$ by area-minimizing hypersurfaces in $\mathcal M_\alpha$.
Moreover, it also follows that
every leaf can be chosen invariant under $\tau_z$, simultaneously
for all $z\in K(\alpha)$. Thus every leaf belongs to
$\mathcal M(\alpha)$. This contradicts the fact established above
that $\mathcal M(\alpha)$ does not foliate $\mathbb R^3$.
Therefore, after discarding finitely many terms, the barriers
$\Delta W_{r_k}$ are uniformly bounded away from $0$. A similar construction extends to any $\mathbb{T}^n$ for $n>3$. This completes the proof of Theorem \ref{maintheorem2}.

\section{Minimal hypersurface inside the gap}\label{section9}

In this section, we prove Theorems~\ref{nonminimizer} and~\ref{areaofnonminimizer}. Our goal is to construct a new minimal hypersurface lying inside the gap formed by two consecutive elements in $\mathcal{M}_{\alpha}$, when $\alpha$ is totally irrational. Assume that two such consecutive minimizers $\Sigma_0$ and $\Sigma_1$ in $\mathcal{M}_\alpha$ bound a gap $G$, and write $\Sigma_i=\partial\llbracket E_i\rrbracket$ with $E_0\subset E_1$. Without loss of generality, assume further that $0\in G$. 

\subsection{If there are no closed minimal hypersurfaces inside the gap}
We now work under the following assumption:
\begin{equation}
\textit{ There is no closed minimal hypersurface inside the gap $G$.}
\tag{$\star$}\label{assumption:star}
\end{equation}

We will remove this assumption at the end using a cutting along stable hypersurfaces trick.

From the analysis in Section \ref{section6}, we know that $G$ projects injectively to $\mathbb{T}^n$ under the covering map and therefore has finite volume. Consequently, $\Sigma_0$ and $\Sigma_1$ are asymptotic outside a sufficiently large compact ball.

Here $B_R(p)=\{x:|x-p|<R\}$, with $B_R=B_R(0)$. Thus the auxiliary
distance $\Tilde d$ mentioned in the introduction is the Euclidean
coordinate distance. Unless a metric is indicated, geometric quantities
are measured using $g$. By Lemma~\ref{asymptotic}, choose increasing radii
$R_k\to\infty$, transverse to $\Sigma_0,\Sigma_1$, such that
\[
a_k:=\mathcal H_g^{n-1}(G\cap\partial B_{R_k})\leq2^{-k}.
\]

The idea is to construct prescribed-boundary min--max hypersurfaces in
larger and larger compact truncations, using auxiliary metrics that agree
with $g$ on fixed compact subsets of $G$, and then take a limit. To do so, we foliate the wall of the compact truncations by a one-parameter family of boundaries $\sigma_t$, then prove that one of them must bound at least two area-minimizing hypersurfaces. Montezuma's prescribed boundary min-max yields a mountain-pass hypersurface $\Gamma^{R_k}$ with the same boundary. We then take the limit as $k\to \infty$. The main
difficulty is to retain a positive excess-area lower bound on a fixed compact set
and exclude collapse onto $\Sigma_0\cup\Sigma_1$. The following proposition
provides the initial width gap.

\begin{Prop}\label{widthgap}
    There is a positive constant $C_G$ that satisfies the following: for every $0<C<C_G$, there exist constants $R_0>0$ and $\eta_0>0$ such that, for every transverse radius $R>R_0$ satisfying
    \[
    \mathcal H_g^{n-1}(G\cap\partial B_R)<\eta_0,
    \]
    and every such one-parameter sweep-out $\{\Sigma^R_t\}_{t=0}^{1}$ with boundary on $G\cap\partial B_R$, connecting $\Sigma_0\llcorner B_R$ to $\Sigma_1\llcorner B_R$, we have:
    \begin{equation}\label{widthgap1}
    \sup_{t\in [0,1]} \mathbf{M}(\Sigma^R_t)> \max_{i=0,1} \mathbf{M}(\Sigma_i\llcorner B_R) +C.
    \end{equation}
\end{Prop}
\begin{proof}
Assume the contrary. We can then select a sequence of transverse radii $R_k\to\infty$ and admissible
sweep-outs such that $a_k:=\mathcal H_g^{n-1}(G\cap\partial B_{R_k})<1/k$
and
\[
\sup_{t\in[0,1]}\mathbf M_g(\Sigma_t^{R_k})
\leq\max_i\mathbf M_g(\Sigma_i \llcorner B_{R_k})+\frac1k.
\]
It follows from minimality of $\Sigma_{0,1}$ and area comparison that:
\[|\mathbf M_g(\Sigma_1\llcorner B_{R_k})-\mathbf M_g(\Sigma_0\llcorner B_{R_k})|\leq a_k.\]

Let $Q_t^k=\llbracket q_t^k\rrbracket$ be such that $
\Sigma_t^{R_k}
=\partial\llbracket\chi_{E_0}+q_t^k\rrbracket$,
with integer densities $q_t^k$ extended by zero outside $G\cap B_{R_k}$.
They are continuous in mass, with $Q_0^k=0$ and
$Q_1^k=\llbracket G\cap B_{R_k}\rrbracket$.
Fix a ball $U\Subset G$, the function
$t\mapsto\operatorname{Vol}_g(U\cap\{q_t^k\geq1\})$ is continuous. For large $k$, $U\subset B_{R_k}$, we can choose $t_k$ such that, by setting $E_k:=E_0\cup\{q_{t_k}^k\geq1\}$, we have:
\[
\operatorname{Vol}_g(E_k\cap U)=\tfrac12\operatorname{Vol}_g(U).
\]
The coarea formula gives
$\mathbf M_g(\partial\llbracket E_k\rrbracket\llcorner B_{R_k})
\leq\mathbf M_g(\Sigma_{t_k}^{R_k})$.
Moreover, $E_k=E_0$ outside $B_{R_k}$, and the additional boundary
on $\partial B_{R_k}$ has mass at most $a_k$.
Therefore, global minimality of $\Sigma_0$ gives, for every fixed ball $B$, every integral $n$-current
$P$ with $\operatorname{spt}P\Subset B$, and all sufficiently large $k$,
\[
\begin{aligned}
&\mathbf M_g(\partial\llbracket E_k\rrbracket\llcorner B)
-\mathbf M_g\bigl((\partial\llbracket E_k\rrbracket+\partial P)\llcorner B\bigr)\\
&\qquad\leq\mathbf M_g(\Sigma_{t_k}^{R_k})+a_k
-\mathbf M_g(\Sigma_0\llcorner B_{R_k})
\leq\frac1k+2a_k.
\end{aligned}
\]
We can repeat the argument in the proof of Lemma~\ref{lem:s7-local} to obtain an $L^1_{\mathrm{loc}}$ limit of $E_k$
between $E_0,E_1$ with area-minimizing boundary. As in the proof of
Theorem~\ref{thm:s7-sufficiency}, we have that a smooth, properly embedded,
area-minimizing component $\Sigma$ of this boundary meets $U$.

The maximum principle gives $\Sigma\subset G$.
If $\Sigma$ did not separate $E_0$ from
$\mathbb R^n\setminus\overline{E_1}$, some connected component of $\mathbb{R}^n\setminus \Sigma$
would lie in $G$. Minimality and Lemma~\ref{asymptotic} would then give,
along admissible radii $s_j\to\infty$,
\[
\mathbf M_g(\Sigma\llcorner B_{s_j})
\leq\mathcal H_g^{n-1}(G\cap\partial B_{s_j})\longrightarrow0,
\]
a contradiction. Let $E$ be the side of $\Sigma$ containing $E_0$,
so that $E_0\subset E\subset E_1$ and, with this orientation,
$\Sigma=\partial\llbracket E\rrbracket$.
Since $\Sigma_0$ and $\Sigma_1$ are two consecutive minimizers in $\mathcal{M}_\alpha$, we have,
for $z\in\mathbb Z^n\setminus\{0\}$:
\[
\begin{cases}
E\subset E_1\subset\tau_zE_0\subset\tau_zE,
&\alpha\cdot z>0,\\
\tau_zE\subset\tau_zE_1\subset E_0\subset E,
&\alpha\cdot z<0.
\end{cases}
\]
Thus $\Sigma\in\mathcal M_\alpha$, contradicting the gap assumption. This completes the proof of the proposition.
\end{proof}

\begin{Rem}\label{constantC}
In fact, it follows from the same interpolation technique in Section \ref{section6} that, for at least one gap $G$ in the lamination $\mathcal{M}_\alpha$, the constant $C_G$ can be chosen to be $C_G\geq\Delta W_\alpha$. Here the radius sequence is chosen with
$\mathcal H_g^{n-1}(G\cap\partial B_{R_k})\to0$, and the resulting spherical
closing errors are retained until the limit. Indeed, if $C_G<\Delta W_\alpha$ for every gap $G$, there exists
such a sequence of radii and a sequence of corresponding sweep-outs inside each gap that satisfy
\[\sup_{t\in [0,1]} \textbf{M}(\Sigma^{R_k}_t)-\max\{\textbf{M}(\Sigma_0\llcorner B_{R_k}), \textbf{M}(\Sigma_1\llcorner B_{R_k})\} <C_0<\Delta W_\alpha,\]
then we can modify to construct sweep-outs in a sequence of homology classes $r_k\to \alpha$ with $\lim\inf \Delta W_{r_k}<\Delta W_\alpha$, contradicting the definition of the Mather energy barrier of $\alpha$. This gap $G$ for which $C_G=\Delta W_\alpha$ will be the one that satisfies Theorem \ref{areaofnonminimizer}.
\end{Rem}

For large $R$, we now replace $B_R$ by a modified region $\widehat B_R$ such that $G\cap \partial\widehat{B}_R$ admits a smooth family of boundaries connecting $\Sigma_0\cap \partial\widehat{B}_R$ and $\Sigma_1\cap \partial\widehat{B}_R$. First, we can express $\Sigma_1$ graphically over $\Sigma_0$, outside a large compact set as $F(\cdot,1)$, where
\[
F(x,t)=\exp_x(tu(x)\nu(x)),\qquad u>0,\qquad
\|u\|_{C^\ell(\Sigma_0\setminus B_R)}\longrightarrow0.
\]
Here \(\nu\) is unit normal pointing into \(G\). Let \(\pi\) denote the normal projection onto \(\Sigma_0\), so that \(\pi(F(x,t))=x\). We define \(\widehat B_R\) using the sublevel sets of \(|\pi(\cdot)|\), extended smoothly across the compact core. Thus, for all sufficiently large \(R\), the artificial boundary \(G\cap\partial\widehat B_R\) is foliated by \(F(\Sigma_0\cap\partial B_R,t)\), \(t\in[0,1]\).
Put $D_R=G\cap\widehat B_R$ and $\Sigma_i^R=\Sigma_i\llcorner\widehat B_R$.
For large $R$, $G\cap B_{R-1}\subset D_R\subset G\cap B_{R+1}$ and
$\Sigma_0^R=\Sigma_0\llcorner B_R$. Define:
\[
\sigma_t^R=F(\cdot,t)_\#\partial\Sigma_0^R,\qquad
C_t^R=F_\#(\llbracket[0,t]\rrbracket\times\partial\Sigma_0^R).
\]
This family of boundaries $\sigma^R_t$ is smooth, with $\partial C_t^R=\sigma_t^R-\sigma_0^R$, and
$\mathbf M_g(C_t^R-C_s^R)\leq a(R)$, where
$a(R)=\mathcal H_g^{n-1}(G\cap\partial\widehat B_R)$.
By Lemma \ref{asymptotic}, along admissible sequence $R_k$ we have
$a_k:=a(R_k)\to0$. 
The proof of Proposition~\ref{widthgap} also applies to these modified truncations, since they exhaust \(G\) and the additional boundary introduced along \(\partial\widehat B_{R_k}\) has mass at most \(a_k\to0\).

Choose \(0<\tau_k<1/2\) and round \(D_{R_k}\) inward, leaving the wall band \(\{\sigma_t^{R_k}:\tau_k\le t\le 1-\tau_k\}\) unchanged, to obtain a smooth compact domain \(M_k\Subset G\). We restrict to the admissible filling class
\[
T=\Sigma_0^{R_k}+C_t^{R_k}+\partial Q,
\qquad
Q\in\mathbf I_n(\mathbb R^n),\quad
\operatorname{spt}Q\subset\overline{D_{R_k}},
\]
with \(\operatorname{spt}T\subset\overline{M_k}\). Thus every admissible path in \(M_k\) is also admissible in the original domain \(D_{R_k}\).

Choose $f_k\geq0$ supported in a thin tubular neighborhood of
$\partial M_k$, so that
$\partial M_k$ is strictly convex with respect to $g_k=e^{2f_k}g$.
We arrange that $M_k$ exhausts $G$, that
$g_k=g$ on every fixed compact subset of $G$
for large $k$, and that $\|f_k\|_\infty$ is
sufficiently small for \eqref{sec9:metric-error}.
Let $\mu_k=\max_i\mathbf M_g(\Sigma_i^{R_k})$.

\begin{Prop}\label{2plateau}
For every $0<C<C_G$, the above convex approximations can be chosen so
that, for all sufficiently large $k$, there exist
$t_k\in(\tau_k,1-\tau_k)$ and a smooth embedded prescribed boundary
$\sigma^{R_k}=\sigma_{t_k}^{R_k}\subset G\cap\partial M_k$ for which:
\begin{enumerate}
\item There are two distinct integral currents $\Gamma_0^{R_k},\Gamma_1^{R_k}$
minimizing $g_k$-mass in the specified component, with boundary $\sigma^{R_k}$.
\item Their supports lie in $\overline{M_k}\subset G$ and are smooth $g_k$-minimal
hypersurfaces in $\operatorname{Int}M_k$, with their integral multiplicities.
\item For $i=0,1$ and some $\epsilon_k\to0$,
\[
\min_j\mathbf M_g(\Sigma_j^{R_k})-a_k
\leq\mathbf M_{g_k}(\Gamma_i^{R_k})
\leq\mu_k+a_k+\epsilon_k.
\]
\item Every admissible fixed-boundary path in $M_k$ between $\Gamma_{0,1}^{R_k}$ satisfies
\begin{equation}\label{widthgap2}
\sup_s\mathbf M_{g_k}(\phi_s)\geq\mu_k+C.
\end{equation}
\end{enumerate}
\end{Prop}

\begin{proof}

Fix $C<C_*<C_G$ and take $k$ sufficiently large
that Proposition~\ref{widthgap} applies for
$C_*$ and $a_k<C/4$.

For each $t\in[0,1]$, consider the relative Plateau
problem with prescribed boundary $\sigma_t^{R_k}$
in $(\overline{D_{R_k}},g)$. After convexification,
we consider the corresponding problem in $(M_k,g_k)$
for $t\in[\tau_k,1-\tau_k]$.

We call two admissible currents $T_0,T_1$
\emph{cheaply connected} if there is an admissible
path $\{\Phi_s\}_{s\in[0,1]}$ joining them such that
\[
\sup_{s\in[0,1]}\mathbf M_g(\Phi_s)<\mu_k+C.
\]
Let $X_{0,k}$ be the cheap component containing
$\Sigma_0^{R_k}$.
By Proposition~\ref{widthgap},
$\Sigma_1^{R_k}\notin X_{0,k}$.
The unique $g$-Plateau minimizers at $t=0,1$
are $\Sigma_0^{R_k},\Sigma_1^{R_k}$, respectively,
by global minimality and unique continuation.
Compactness and short interpolation allow us
to choose $\tau_k>0$ so that the minimizers at
$\tau_k$ and $1-\tau_k$ lie in opposite components.
This separation persists for sufficiently small
rounding and conformal perturbations.
The wall comparison also gives (3).

For the fixed metric $g_k$, let $I_0$ and $I_1$
be the set of $t$ for which some Plateau minimizer with boundary $\sigma^{R_k}_t$
belongs to $X_{0,k}$ or its complement,
respectively. Compactness and interpolation
show that both are closed. Since
$I_0\cup I_1=[\tau_k,1-\tau_k]$, there exists
$t_k\in I_0\cap I_1$.
As a result, we can choose two $g_k$-Plateau minimizers
$\Gamma_0^{R_k}$ and $\Gamma_1^{R_k}$
with
\[
\partial\Gamma_0^{R_k}
=\partial\Gamma_1^{R_k}
=\sigma_{t_k}^{R_k},
\]
belonging to different cheap components.
Since $g_k\ge g$, and admissible paths in $M_k$
belong to the original filling class,
any path violating (4) would cheaply connect
these minimizers, contradicting Proposition \ref{widthgap}.

\end{proof}

We now carry out the localized min--max construction under
assumption~\eqref{assumption:star}.

Fix $0<C<C_*<C_G$. By Proposition~\ref{2plateau},
applied at level $C_*$, there are two distinct
$g_k$-Plateau minimizers with common boundary
$\sigma^{R_k}_{t_k}$ satisfying the width gap
\eqref{widthgap2}.

By coarea formula, their minimizing fillings
may be taken to be sets $A$ and $B$.
The lattice inequality implies that $A\cap B$
is also minimizing. Since the original two
minimizers are not cheaply connected, at least
one is not cheaply connected to their intersection.
Thus, after relabeling, we may write
\[
\Gamma_0^{R_k}=T_{A_-},\qquad
\Gamma_1^{R_k}=T_{A_+},\qquad A_-\subset A_+,
\]
without losing the width gap.

Denote by
\[
A_k^0=\mathbf M_g(\Sigma_0^{R_k}),\qquad
L_k=\mathbf L_{g_k}(\Pi_k),
\]
where $\Pi_k$ is the prescribed-boundary homotopy
class of a comparison family joining these currents.
By Proposition~\ref{2plateau}, we have
\begin{equation}\label{sec9:exact-endpoints}
\mathbf M_{g_k}(\Gamma_i^{R_k})
\leq\mu_k+a_k+o_k(1),\qquad
L_k\geq\mu_k+C_*.
\end{equation}

For large $k$, the width exceeds both endpoint masses. Montezuma's
min-max theorem \cite[Theorem~1.3]{rafael} applies on the strictly
$g_k$-convex $M_k$ and gives $g_k$-minimal integral varifolds
$V_k$ with support $\Gamma^{R_k}, \partial \Gamma^{R_k}=\sigma^{R_k}_{t_k}$ and $\|V_k\|_{g_k}(M_k)=L_k$.
We now show that for sufficiently large $k$, the mass of $\Gamma^{R_k}$ exceeds that of the minimizer by $C$ inside a fixed compact set. 

\begin{claim}\label{keyclaim}
The min-max realizations $\Gamma^{R_k}$ can be selected, after passing
to a subsequence, so that for a fixed $R'$ and all sufficiently large $k$,
\begin{equation}\label{widthgapcompact}
\mathbf M_h(\Gamma^{R_k}\llcorner\widehat B_{R'})
\geq\max_i\mathbf M_g(\Sigma_i^{R'})+C,
\qquad h\in\{g_k,g\}.
\end{equation}
\end{claim}

\begin{proof}[Proof of claim]
Fix $0<C<C_*<C_G$ and choose $\delta>0$ with
$3\delta<C_*-C$. Set
\[
d(R)=C_g\left[a(R)+
\operatorname{Vol}_g(G\setminus\widehat B_R)^{\frac{n-1}{n}}\right].
\]
By Lemma~\ref{asymptotic}, there are arbitrarily large admissible
radii $R$ with $a(R)+2d(R)<\delta$.

We first localize the minimizing sweepouts. Apply the cutting-and-capping construction of \cite[Section~6]{ChambersLiokumovich2020} to the ordered fillings \(A_-\subset A_+\), keeping the prescribed boundary and endpoints fixed. Whenever a slice has exterior mass more than \(2d(R)\) above that of \(\Gamma_0^{R_k}\), replace its exterior by one-sided minimizing fillings. The lattice inequality ensures that the mass saved absorbs the interpolation cost, which is at most \(d(R)\). Thus the modified sweepouts remain in \(\Pi_k\), have maximal mass \(L_k+o(1)\), and satisfy the exterior bound on near-maximal slices. Diagonalizing over the good radii yields a min--max varifold \(V_k\) for which, at every fixed such \(R\), we have
\begin{equation}\tag{L}\label{sec9:tightness}
\|V_k\|_{g_k}(M_k\setminus\overline{\widehat B_R})
\leq
\|\Gamma_0^{R_k}\|_{g_k}(M_k\setminus\widehat B_R)
+2d(R)+o_k(1).
\end{equation}
Closing $\Gamma_0^{R_k}$ across the inner wall and using its
Plateau bound give
\[
\|\Gamma_0^{R_k}\|_{g_k}(M_k\setminus\widehat B_R)
\leq A_k^0-\mathbf M_g(\Sigma_0^R)+a(R)+o_k(1).
\]
Here we choose the conformal factors $f_k$ so that
\begin{equation}\label{sec9:metric-error}
\bigl(e^{(n-1)\|f_k\|_\infty}-1\bigr)
(\mu_k+\Lambda_G+1)\longrightarrow0,
\end{equation}
where the comparison sweep-out gives
$L_k\leq\mu_k+\Lambda_G+o_k(1)$ with $\Lambda_G$
independent of $k$.

Consequently, by \eqref{sec9:exact-endpoints} and (L),
\[
\begin{aligned}
\|V_k\|_{g_k}(\overline{\widehat B_R})
&\ge L_k-A_k^0+\mathbf M_g(\Sigma_0^R)
-a(R)-2d(R)-o_k(1)\\
&\ge \max_{i=0,1}\mathbf M_g(\Sigma_i^R)
+C_*-2\delta-o_k(1),
\end{aligned}
\]
where we used $L_k\ge\mu_k+C_*$, $\mu_k\ge A_k^0$, and
$|\mathbf M_g(\Sigma_1^R)-\mathbf M_g(\Sigma_0^R)|\le a(R)$.

Finally, choose a slightly larger regular radius $R'$ with
$\overline{\widehat B_R}\subset\widehat B_{R'}$ and
\[
\max_i\mathbf M_g(\Sigma_i^{R'})-
\max_i\mathbf M_g(\Sigma_i^R)<\delta.
\]
By \eqref{sec9:metric-error}, the $g_k$- and $g$-masses
of $V_k$ differ by $o_k(1)$ even on these restricted regions.
Because $C_*-3\delta>C$, \eqref{widthgapcompact} follows.
\end{proof}

We now complete the proofs of Theorems~\ref{nonminimizer}
and~\ref{areaofnonminimizer} under hypothesis~\eqref{assumption:star}.
\begin{proof}[Proof of Theorems~\ref{nonminimizer}
and~\ref{areaofnonminimizer} (1):]
Subtracting the exterior lower bound obtained by comparison with
$\Sigma_0$ from the total mass upper bound cancels
$\mathbf M_g(\Sigma_0^{R_k})$ and, with the closing-wall errors retained,
gives the following uniform local mass bound:
\begin{equation}\label{sec9:local-upper}
\|V_k\|_g(\widehat B_R)
\leq\mathbf M_g(\Sigma_0^R)+\Lambda_G+a(R)+3a_k+o_k(1).
\end{equation}

It follows from the work \cite{morsetheory} of Montezuma, Marques, and Neves that $\Gamma^{R_k}$ has index bounded by one. Therefore, we can apply Sharp's compactness theorem \cite{sharp} locally, then passing to a diagonal subsequence to obtain an embedded minimal hypersurface $\Gamma$ inside $G$. The convergence is smooth and graphical away from at most one point.  It follows from Claim~\ref{keyclaim} that:
\begin{equation}\label{estimate}
\mathbf M(\Gamma\llcorner\widehat B_{R'})
\geq\max_{i=0,1}\mathbf M(\Sigma_i^{R'})+C,
\end{equation}
for a fixed $R'$. Therefore, $\Gamma$ is different from $\Sigma_0$ and $\Sigma_1$ and thus is not area-minimizing and belongs to the gap. This proves Theorem \ref{nonminimizer} under the hypothesis (\ref{assumption:star}).

    We now prove Theorem \ref{areaofnonminimizer}, part (1). For any $\epsilon>0$, choose the fixed interior radius $R'$ so that (\ref{estimate}) holds and
    \[
    \mathcal H_g^{n-1}(G\cap\partial \widehat B_{R'})<\epsilon.
    \]
    By Lemma~\ref{asymptotic}, we can
    choose a sequence of radii $s_j\to\infty$, with $s_j>R'+1$, such that
    \[
    \mathcal H_g^{n-1}(G\cap\partial B_{s_j})\longrightarrow0.
    \]
    Since $\Sigma_0$ is area-minimizing, area comparison gives
    \[
    \begin{aligned}
    \mathbf M(\Gamma\llcorner(B_{s_j}\setminus\overline{\widehat B_{R'}}))
    &\geq\mathbf M(\Sigma_0\llcorner(B_{s_j}\setminus\overline{\widehat B_{R'}}))
       -\epsilon-\mathcal H_g^{n-1}(G\cap\partial B_{s_j}).
    \end{aligned}
    \]
    Together with (\ref{estimate}) with $C=C_G-\epsilon$, this yields
    \[
    \mathbf M(\Gamma\llcorner B_{s_j})-\mathbf M(\Sigma_0\llcorner B_{s_j})
    \geq C_G-2\epsilon-\mathcal H_g^{n-1}(G\cap\partial B_{s_j}).
    \]
    Hence
    \[
    \begin{aligned}
    \limsup_{R\to\infty}
    \bigl(\mathbf M(\Gamma\llcorner B_R)-\mathbf M(\Sigma_0\llcorner B_R)\bigr)
    &\geq\limsup_{j\to\infty}
    \bigl(\mathbf M(\Gamma\llcorner B_{s_j})-\mathbf M(\Sigma_0\llcorner B_{s_j})\bigr)\\
    &\geq C_G-2\epsilon.
    \end{aligned}
    \]
    Sending $\epsilon\to0$ gives the bound with $C_G$. It follows from
Remark~\ref{constantC} that for some gap $G$, we have $C_G=\Delta W_\alpha$. The proof under hypothesis (\ref{assumption:star}) is complete.
\end{proof}

\subsection{The general case}
We now turn to the general case of Theorem~\ref{nonminimizer}, without assuming the (\ref{assumption:star}) hypothesis. We repeat the previous Plateau construction to obtain $\Gamma^{R_k}_0, \Gamma_1^{R_k}$ with the same boundary $\sigma^{R_k}:=\sigma^{R_k}_{t_k}$. They converge locally
to $\Sigma_0,\Sigma_1$ and bound regions $G_{R_k}$ exhausting $G$.
Orient them so that
$\partial\llbracket G_{R_k}\rrbracket=\Gamma_1^{R_k}-\Gamma_0^{R_k}$.

The strategy is to remove finitely many stable regions and apply min--max
in the complement. It uses the following stable-cutting procedure. The proof is in Appendix~\ref{appendix:e}.
\begin{Prop}\label{cutstable}
    There exists a finite collection of pairwise disjoint open sets $U_{S_1}, \ldots, U_{S_j}\subset G$ such that for every $1\leq i \leq j, \partial U_{S_i}=S_i$ is a closed, connected, stable minimal hypersurface inside $G$, and such that when $k$ is sufficiently large, for every closed minimal hypersurface $\Sigma$ contained inside the interior of \[\Omega_{R_k}:=G_{R_k}\setminus (\cup_{i=1}^{j} U_{S_i}),\] we have
    \[\mathbf{M}(\Sigma)\geq \mathbf{M}(\Gamma_0^{R_k})+\mathbf{M}(\Gamma_1^{R_k}). \]
\end{Prop}

Let $\mathcal{S}$ be the collection $\{S_1, S_2, \ldots, S_j\}$ of stable minimal hypersurfaces constructed in Proposition \ref{cutstable}. Since $G_{R_k}\to G$, $G_{R_k}$ contains all elements of $\mathcal{S}$ for sufficiently large $k$. For each $R_k$, if there exists a stable hypersurface $X_{R_k}$ with boundary $\sigma^{R_k}$ that does not lie completely on one side (above or below) with respect to every closed hypersurface in $\mathcal{S}$, then we can take the limit of $X_{R_k}$ to produce a new minimal hypersurface inside the gap $G$. As a result, we can assume that every stable hypersurface inside $G_{R_k}$ lies completely on one side with respect to $\mathcal{S}$. 

By replacing $\Gamma_0^{R_k}$ by
\[\sup\{X\subset G_{R_k}: \partial X=\sigma^{R_k}, X \text{ is stable}, X\text{ lies below }\mathcal{S}\} \]
and $\Gamma_1^{R_k}$ by
\[\inf\{X\subset G_{R_k}: \partial X=\sigma^{R_k}, X \text{ is stable}, X\text{ lies above }\mathcal{S}\}, \]
we can assume that there is no stable minimal hypersurface with boundary $\sigma^{R_k}$ lying completely inside $G_{R_k}$.

Orient each inner boundary by
$S_i=-\partial\llbracket U_{S_i}\rrbracket$. Then the relevant filling
relation is the current identity
\[
\partial\llbracket\Omega_{R_k}\rrbracket
=\Gamma_1^{R_k}+\sum_{i=1}^j S_i-\Gamma_0^{R_k}.
\]
For each $k$, we can apply Theorem \ref{camillo} (after a small conformal change of the metric near the boundary as in Subsection 9.1) to \[ [\Gamma_0^{R_k}]= \left[\Gamma_1^{R_k}+\sum_{i=1}^{j} S_i\right]\in H_{n-1}(\Omega_{R_k}, \partial \Omega_{R_k}) \] 
whose representatives have the same boundary $\sigma^{R_k}$ to get another minimal hypersurface $\Gamma^{R_k}$ satisfying the following:

\begin{equation}\label{weak1}
\mathbf{M}(\Gamma^{R_k}) > \mathbf{M}(\Gamma_1^{R_k})+\sum_{i=1}^j\mathbf{M}(S_i).
\end{equation}

\begin{claim}\label{widthbound1}
There exists a positive constant $C_0$ independent of $k$ such that the following upper bound holds: \[\mathbf{M}(\Gamma^{R_k})\leq \mathbf{M}(\Gamma_0^{R_k})+\sum_{i=1}^{j}\mathbf{M}(S_i)+C_0.\]
\end{claim}
\begin{proof}[Proof of claim]
   Fix \(\varepsilon>0\) and choose the interpolation scales as in Step 2 of Section \ref{sec6:section}. We can choose a fixed admissible radius \(R\) such that $\operatorname{Vol}_g(G\setminus\widehat{B}_R)$ and $\mathcal{H}^{n-1}(G\cap \partial \widehat{B}_R)$ are sufficiently small. For all large \(k\), we have \(R_k>R\) and \(a_k\) sufficiently small. Since \(\Sigma_0\) and \(\Sigma_1\) are area-minimizing, the small-piece construction of Section \ref{sec6:section} and Corollary \ref{interpolation} yield an admissible sweep-out $\Phi_t^{\text{ext}}$ of $G\cap (\widehat{B}_{R_k}\setminus \widehat{B}_R)$
    with maximal mass at most
\[
\mathbf M\bigl(\Sigma_0\llcorner
(\widehat B_{R_k}\setminus\widehat B_R)\bigr)+\varepsilon,
\]
where \(\varepsilon\) absorbs the wall and interpolation errors.

Fix an admissible sweepout $\Phi^{\text{int}}_t$ from \(\Sigma_0^R\) to \(\Sigma_1^R\). Gluing it to the $\Phi^{\text{ext}}$ along \(\partial\widehat B_R\cap G\) gives, for all sufficiently large \(k\), an admissible sweepout \(\Phi^k\) from \(\Sigma_0^{R_k}\) to \(\Sigma_1^{R_k}\) satisfying
\[
\sup_{t\in[0,1]}\mathbf M(\Phi^k(t))
\le \mathbf M(\Sigma_0^{R_k})+C_0,
\]
where \(C_0\) depends only on $\Phi^{\text{int}}$, and is independent of \(k\).
    Since $\Gamma_0^{R_k}$ and $\Gamma_1^{R_k}$ converge to $\Sigma_0$ and $\Sigma_1$, respectively, we see that a similar estimate holds for some sweep-outs between them:
    \begin{equation}\label{9.6.1}
    \mathbf{M}(\Phi(t))\leq \mathbf{M}(\Gamma_0^{R_k})+C_0, \quad \forall t\in [0,1].
    \end{equation}
    Fix one such sweep-out $\Psi$. 
    Restrict the filling of $\Psi(t)$ to $U_{S_i}$, obtaining sets
$V_i^t$ with boundary $\Psi(t)\llcorner U_{S_i}$ plus a trace on $S_i$.
We can subtract these boundaries to cancel the portions in the removed regions:
\[
\Psi'(t)=\Psi(t)-\sum_{i=1}^j\partial\llbracket V_i^t\rrbracket.
\]
The endpoints are $\Gamma_0^{R_k}$ and
$\Gamma_1^{R_k}+\sum_i S_i$ with the chosen orientations. Admissible
slicing and the relative interpolation retain these endpoints.
    It follows that
    \[\mathbf{M}(\Psi'(t))\leq \mathbf{M}(\Psi(t))+\sum_{i=1}^{j}\mathbf{M}(S_i).\]
    Since $\Psi'$ represents the chosen prescribed-boundary class, we have
    \begin{equation}\label{9.6.2}
    \mathbf{M}(\Gamma^{R_k})\leq \sup_{t\in [0,1]}\mathbf{M}(\Psi'(t))\leq \sup_{t\in [0,1]} \mathbf{M}(\Psi(t))+\sum_{i=1}^{j}\mathbf{M}(S_i).
    \end{equation}
    
    The desired inequality follows by combining (\ref{9.6.1}) and (\ref{9.6.2}).
\end{proof}

We write the decomposition of $\Gamma^{R_k}$ as:
\begin{equation}\label{indexx}
\Gamma^{R_k}=|W_k|+\sum_{i=1}^j n_i|S_i|
                  +\sum_{\ell=1}^{m}p_\ell|Z_\ell|.
\end{equation}
Here $W_k$ is the oriented current carrying the boundary $\sigma^{R_k}$, i.e., $\partial W_k=\sigma^{R_k}$, and $Z_\ell$ are the
remaining closed interior components and $n_i, p_\ell\in \mathbb{Z}_{\geq 0}$. Since $\Gamma_0^{R_k}$ is the least area hypersurface spanning $\sigma^{R_k}$, we have
$\mathbf M(W_k)\geq\mathbf M(\Gamma_0^{R_k})$.
Therefore, Claim~\ref{widthbound1} implies that for all $\ell$ with $p_\ell>0$, we have
\[
\mathbf M(Z_\ell)\leq\sum_{i=1}^j\mathbf M(S_i)+C_0.
\]
Since $Z_\ell$ is a closed minimal hypersurface in the interior of $\Omega_{R_k}$, the above upper bound violates Proposition \ref{cutstable}, when $k$ is sufficiently large. Therefore, $m=0$ and $\Gamma^{R_k}=W_k+\sum_{i=1}^j n_iS_i$.
Since all $S_i$ are stable, we have $\operatorname{Ind}(W_k)\leq1$. Under the assumption that the metric is bumpy, it follows from the proof of the multiplicity one conjecture for one-parameter min-max families in \cite{MN2021} (see also the introduction of \cite{morsetheory}) that $n_i \in \{0,1\}$ for all $i$.

\begin{claim}
   There exists a positive constant $C$ independent of $k$ such that for a subsequence of $R_k\to \infty$, we have
   \[\mathbf{M}(W_k)\geq \mathbf{M}(\Gamma_1^{R_k})+C.\]
\end{claim}
\begin{proof}[Proof of claim]
    Assume $\mathcal S\ne\varnothing$. If some $n_j=0$ in
\eqref{indexx}, then
\[\mathbf{M}(W_k)\geq \mathbf{M}(\Gamma_1^{R_k})+\mathbf{M}(S_{j})\geq \mathbf{M}(\Gamma_1^{R_k})+\delta_\mathcal{S},\]
where $\delta_\mathcal{S}=\min\{\mathbf{M}(S_i)| S_i\in \mathcal{S}\}$. 
If this occurs along a subsequence, the claim follows. 

Assume now that $n_j=1$ for every $j$ and that $W_k$ and $\Gamma_1^{R_k}$ lie on the same side of $\mathcal{S}$ (otherwise we can take the limit of $W_k$ again). Since $W_k$ is unstable, there exists a small neighborhood of $W_k$ foliated by hypersurfaces $W_k^{t}, t\in [-\epsilon, \epsilon]$,
of the same boundary $\sigma^{R_k}$, such that their mean curvature vectors point away from $W_k$ and:
\[\mathbf{M}(W_k^t)<\mathbf{M}(W_k), \quad 0<|t|\leq\epsilon.\] 

If the claim fails, the positive differences
$\mathbf M(W_k)-\mathbf M(\Gamma_1^{R_k})$ tend to zero. Since the
barriers have equal least mass and $\sum_i\mathbf M(S_i)>0$, for large $k$
we obtain 
\[\mathbf{M}(W_k)<\mathbf{M}(\Gamma_0^{R_k})+\sum_{i=1}^j \mathbf{M}(S_i).\]

Together with the strict stability of $\Gamma_0^{R_k}+\sum_{i=1}^j S_i$, we can apply Theorem \ref{camillo} for $W_k^\epsilon$ and $\Gamma_0^{R_k}+\sum_{i=1}^j S_i$ to obtain another index-one hypersurface $\widetilde{W}_k$ disjoint from $W_k$ (by the maximum principle) with $\partial \widetilde{W}_k=\sigma^{R_k}$.
 \begin{figure}[H]
    \centering
    \includegraphics[width=0.95\linewidth]{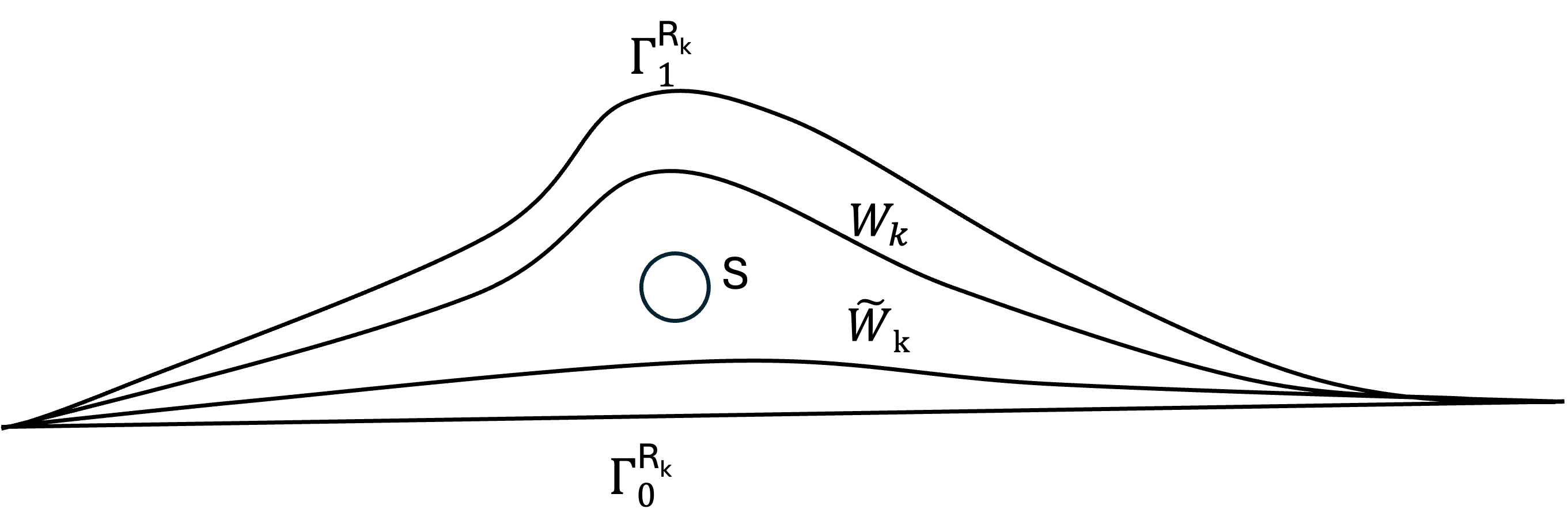}
    \caption{$W_k$ and $\widetilde{W}_k$ act as barriers, forcing the existence of another stable hypersurface with boundary $\sigma^{R_k}$.}
    \end{figure} 
    
Now by using $W_k$ and $\widetilde{W}_k$ as barriers, we can construct a stable hypersurface with the same boundary $\sigma^{R_k}$ lying completely inside $G_{R_k}$, contradicting the choices of $\Gamma_0^{R_k}$ and $\Gamma_1^{R_k}$. This completes the proof of the claim. 
\end{proof}

We can now proceed similarly as in the proof of Claim \ref{keyclaim},
choosing a fixed good interior radius $R'$ beyond all the compact sets
$\overline{U_{S_i}}$ and then taking $k$ large, to obtain an analogue of
the inequality (\ref{widthgapcompact}) with $W_k$ in place of
$\Gamma^{R_k}$: for every $0<C<C_G$, there is $R'>0$ such that for
all sufficiently large $k$ with $R_k>R'$ we have
\[\mathbf{M}(W_k\llcorner \widehat B_{R'})\geq \max \{\mathbf{M}(\Sigma_0^{R'}), \mathbf{M}(\Sigma_1^{R'})\}+C.\]

Note that by the construction, $W_k$ has index at most one and uniformly bounded area on compact subsets. Therefore, Sharp's compactness theorem \cite{sharp} applies to the sequence $\{W_k\}_k$, yielding a limiting hypersurface $\Gamma$ that satisfies the same estimate as above. In particular, $\Gamma$ is different from $\Sigma_0$ and $\Sigma_1$ and lies entirely inside $G$. This concludes the proof of Theorem \ref{nonminimizer}.
\\

We are now in the position to prove the second statement of Theorem \ref{areaofnonminimizer}, for a gap $G$ chosen in Remark \ref{constantC}. Fix a small positive constant $\epsilon$. For each open set $U_{S_i}$ with strictly stable boundary $S_i\in \mathcal{S}$, we can find a one-parameter sweep-out that deforms $S_i$ to the zero current:
\[\Phi_i: [0,1] \to \mathcal{Z}_{n-1}(U_{S_i}; \mathbf{F}, \mathbb{Z}), \quad \Phi_i(0)=S_i, \Phi_i(1)=0\]
such that \[W_\partial(U_{S_i})+\epsilon \geq \sup_{t\in [0,1]} \mathbf{M}(\Phi_i(t)).\]

In the set-up above, for any sweep-out $\Phi^{R_k}$ from $\Gamma_0^{R_k}$ to $\Gamma_1^{R_k}+S_1+\ldots+S_j$, we can combine with $\Phi_1, \ldots, \Phi_j$ to make a sweep-out $\Phi^k$ from $\Gamma_0^{R_k}$ to $\Gamma_1^{R_k}$ of the entire region between them:
\[\Gamma_0^{R_k} \xrightarrow{\Phi^{R_k}} \Gamma_1^{R_k}+S_1+\cdots+S_j \xrightarrow{\Gamma_1^{R_k}+\Phi_1+\cdots+\Phi_j} \Gamma_1^{R_k}.\]
Since $\Gamma_1^{R_k}$ is disjoint from $\Phi_i(t)$ for all $i, t$, we have
\[\max\{\sup \mathbf{M}(\Phi^{R_k}(t)), \sum_{i=1}^{j}\sup \mathbf{M}(\Phi_i(t))+\mathbf{M}(\Gamma_1^{R_k})\}\geq \sup \mathbf{M}(\Phi^k(t)).\]
Choose $\Phi^{R_k}$ such that $\mathbf{M}(\Gamma^{R_k})+\epsilon\geq \sup \mathbf{M}(\Phi^{R_k}(t)).$
It follows that 
\[\max\{\mathbf{M}(\Gamma^{R_k})+\epsilon, \mathbf{M}(\Gamma_1^{R_k})+\sum_{i=1}^{j} W_\partial(U_{S_i})+j\epsilon\} \geq \sup \mathbf{M}(\Phi^k(t)).\]
By Propositions \ref{widthgap} and \ref{2plateau} we have:
\[\sup\mathbf{M}(\Phi^k(t))\geq C+\mathbf{M}(\Gamma_0^{R_k})-\epsilon\] for sufficiently large $k$. Since $\Gamma^{R_k}$ consists of $W_k$ together with some of the $S_i$, each with multiplicity one, it follows that
\[\mathbf{M}(W_k)+\sum_{i=1}^{j} \mathbf{M}(S_i)\geq \mathbf{M}(\Gamma^{R_k}).\]
As a result
\[\max \{\mathbf{M}(W_k)+\sum_{i=1}^{j} \mathbf{M}(S_i)+\epsilon, \mathbf{M}(\Gamma_1^{R_k})+\sum_{i=1}^{j} W_\partial(U_{S_i})+j\epsilon\} \geq C+\mathbf{M}(\Gamma_0^{R_k})-\epsilon. \]
Sending $\epsilon\to0$ and $C\to \Delta W_\alpha$ for the chosen gap $G$ we obtain that
if $\sum_{i=1}^j W_\partial(U_{S_i})<\Delta W_\alpha$, then
\[\limsup_{R\to \infty} \left(\mathbf{M}(\Gamma\llcorner B_R)-\mathbf{M}(\Sigma_0\llcorner B_R)\right)\geq \Delta W_\alpha -\sum_{i=1}^j \mathbf{M}(S_i).\]
This completes the proof of Theorem \ref{areaofnonminimizer}.

\section{Open problems}\label{section11}
We conclude with several open questions that naturally emerge from our work.
\subsection{The function $\Delta W_\alpha$} 
For a totally irrational class $\alpha\in H_{n-1}(\mathbb{T}^n, \mathbb{R})$, we conjecture that the limit \[\Delta W_\alpha=\lim_{r_k\to \alpha} \Delta W_{r_k}=\lim_{r_k\to \alpha}(\omega(r_k)-S(r_k))\] 
exists and is independent of the sequence $r_k$. Equivalently, the asymptotic min-max excess in rational directions approaching $\alpha$ should be an intrinsic quantity associated only with $\alpha$ rather than a particular sequence of rational classes approximating it.  If this is the case, we can then extend the discrete function \[r\rightarrow \Delta W_{r}, \quad r\in H_{n-1}(T^n, \mathbb{Z})\] 
continuously to totally irrational classes, with the property that it vanishes if and only if there is a foliation by minimal hypersurfaces with direction $\alpha$. In the case of the Hamiltonian twist map and geodesics on tori, the corresponding Mather energy barrier and Peierls's barrier were shown to be continuous at irrational frequencies, see \cite{mather} and \cite{mathermodulus}. An even stronger form of the question is whether the extended function admits a modulus of continuity depending only on the geometry of the ambient metric. The same question regarding the Peierls's barrier was analyzed in \cite{mathermodulus}.\\

We also expect that the collection of quantities $\{\Delta W_r\}, r\in H_{n-1}(\mathbb{T}^n, \mathbb{Z})$ may contain global geometric information about the ambient metric. For instance, the simultaneous vanishing of barriers would imply that there exists a foliation of $\mathbb{T}^n$ by closed, homologically area-minimizing hypersurfaces in all primitive homology classes. Bangert \cite{bangerticm} conjectured that this property implies the flatness of the Riemannian metric. 

\subsection{Minimal hypersurfaces inside gaps} Theorem \ref{nonminimizer} shows that for a generic metric, in the totally irrational setting, whenever a minimal lamination by area-minimizing hypersurfaces contains gaps, one can find complete, non-compact minimal hypersurfaces inside these gaps that are not area-minimizing. The genericity of the metric is used in the proof to ensure that the number of disjoint, stable hypersurfaces in the set $S$ that we remove is finite. We conjecture that the same conclusion should remain valid even without the bumpiness assumption.
A natural further question is to investigate the Morse index of such non-compact hypersurfaces constructed by Theorem \ref{nonminimizer} inside gaps. The construction gives a non-area-minimizing hypersurface $\Gamma$ of index at most one, but it is not clear whether $\Gamma$ must actually be unstable. \\

Another closely related open problem is to analyze the gap in minimal lamination in the case where $\alpha$ is rationally dependent. In the proof of Theorem \ref{nonminimizer}, we rely crucially on the fact that when $\alpha$ is rationally independent, the corresponding gaps necessarily have finite volume and hence their boundaries converge to each other uniformly. This phenomenon is not guaranteed in the rationally dependent setting. It would be interesting if one could produce a non-minimizing hypersurface in this case. A natural way to overcome this difficulty is to work not in the universal cover $\mathbb{R}^n$, but in a suitable intermediate normal cover adapted to the rational dependencies of $\alpha$, where the gap has finite volume, and then repeat the constructions carried out here. 

\subsection{The Birkhoff property} The main objects studied in this paper are area-minimizing hypersurfaces that satisfy Conditions (\ref{birkoff1}) and (\ref{birkoff2}). These conditions arise naturally from our goal of studying minimal laminations on the torus. In the non-parametric setting, for variational problems with an elliptic integrand, Bangert raised the question of whether the Birkhoff property can be forced by a weaker geometric assumption on the minimizer. In particular, he proved that minimizers lying entirely between two parallel affine hyperplanes with a totally irrational normal vector are Birkhoff. It would be interesting to investigate whether a similar phenomenon occurs in the parametric setting of minimal hypersurfaces. 

\subsection{Generalization to other manifolds} We expect that many of the arguments developed in this paper can be carried over to a broader class of manifolds with abelian fundamental groups, for example, on the product $\mathbb{T}^n\times S^m$. 

There are also related results in the non-abelian setting. It is shown in \cite{planelike} that when the fundamental group is \textit{residually finite} and there exists a function (called the cocycle function) $\varphi: \pi_1(M)\to \mathbb{R}$ such that
\[\varphi(gg')=\varphi(g)+\varphi(g'),\]
one can still prove the existence of minimal laminations in any direction governed by the corresponding level set of the cocycle function. One natural obstruction for the method to carry over is that, in general, area-minimizing hypersurfaces in a given homology class might not be connected, even when the class is primitive. 

\appendix 
\section{Connectedness of area-minimizing hypersurfaces on the torus}\label{appendix:c}
\begin{Prop}
    Suppose that $3\leq n \leq 7$ and let $g$ be a Riemannian metric on $\mathbb{T}^n$. Then every homologically area-minimizing hypersurface in a primitive integer homology class is connected and of multiplicity one. 
\end{Prop}
\begin{proof}
    Suppose that $\alpha\in H_{n-1}(\mathbb{T}^n, \mathbb{Z})$ is a primitive class and let $\Sigma\in \alpha$ be an area-minimizing hypersurface in $\alpha$, which exists and is smooth by Federer \cite{federerbook}. Assume that $\Sigma=\Sigma_1+\Sigma_2+\cdots+\Sigma_k$ and each $\Sigma_i$ is closed. 

    Let $\alpha_i$ denote the homology class represented by $\Sigma_i$. It follows that $\alpha=\alpha_1+\cdots+\alpha_k$. Since $\Sigma$ is area-minimizing, none of the $\alpha_i$ can be zero (otherwise one could remove $\Sigma_i$ and decrease the mass of $\Sigma$ while staying inside $\alpha$), and we have 
    \[\mathbf{M}(\Sigma)=\sum_{i=1}^{k}\mathbf{M}(\Sigma_i).\]
    
    Also, by Federer, $\Sigma$ is embedded. It follows that the distinct components $\Sigma_i$ are pairwise disjoint. In other words, for $i\neq j$, we have $\Sigma_i\cap \Sigma_j=\emptyset$. Consequently, the homological intersection product $\alpha_i\cdot \alpha_j$ vanishes. On $\mathbb{T}^n$, each $\alpha_i$ can be represented by the projection of an affine plane, and any two affine planes with linearly independent normal vectors intersect. Therefore we must have that $\alpha_i$ and $\alpha_j$ are linearly dependent for all $1\leq i, j \leq k$.

    It follows that there exist $\beta\in H_{n-1}(\mathbb{T}^n, \mathbb{Z})$ and a positive integer $N$ such that $\alpha=N\beta$. However, since $\alpha$ is primitive, we have $N=1$, and it follows that $\Sigma$ has only one connected component of multiplicity one. 
\end{proof}

\section{Minimal foliation in an integer homology class}\label{appendix:a}
\begin{Prop}
    Assume that $g$ is a Riemannian metric on $\mathbb{T}^n$ and $\alpha\in H_{n-1}(\mathbb{T}^n, \mathbb{Z})$ is a primitive integer class. Then the following are equivalent:
    \begin{enumerate}
        \item There is a minimal foliation of $\mathbb{T}^n$ by closed, homologically area-minimizing hypersurfaces in $\alpha$.
        \item $\omega(\alpha)=S(\alpha)$.
    \end{enumerate}
\end{Prop}
\begin{proof}
    If $\mathbb{T}^n$ admits a foliation by a family $\Phi_t$ of closed minimal hypersurfaces in $\alpha$, the following standard calibration argument shows that each leaf $\Phi_t$ minimizes the area in the homology class $\alpha$: consider the following $(n-1)$-form: 
    \[w:=\iota_\nu\text{vol}_g,\]
    where $\nu$ is the globally chosen unit normal vector field to the foliation and $\iota_\nu$ is the interior product. 
    
    A standard computation shows that this is a closed calibration form. By Stokes' theorem, for any integral cycle $T$ with $[T]=\alpha$, we have:
    \[\mathbf{M}(T)\geq \int_Tw=\int_{\Phi_t}w=\mathbf{M}(\Phi_t).\]
    Hence, each $\Phi_t$ minimizes area among all cycles representing $\alpha$.
    Therefore, $\mathbf{M}(\Phi_t)=S(\alpha)$ for all $t$. It follows that
\[S(\alpha)\leq \omega(\alpha)\leq \sup\mathbf{M}(\Phi(t))=S(\alpha)\]
implying $\omega(\alpha)=S(\alpha).$ 

Conversely, assume that $\omega(\alpha)=S(\alpha)$. By the definition of $\omega(\alpha)$, we can find a sequence of continuous maps $\Phi_i: S^1\to \mathcal{Z}_{n-1}(M, \alpha)$ (the space of all cycles in $\alpha$) such that
\[\lim_{i \to \infty} \sup_{t\in S^1} \mathbf{M}(\Phi_i(t)) =\omega(\alpha)=S(\alpha).\]

We want to take the limit of the leaves of these sweep-outs. For any geodesic ball $U\subset \mathbb{T}^n$, since $\Phi_i$ is a one-parameter sweep-out, there exists $t_i\in S^1$ such that the current $\Phi_i(t_i)$ divides $U$ into two regions of equal volume. Consider the sequence $\{\Phi_i(t_i)\}_i$ of integral cycles with uniformly bounded mass. By compactness, there is a subsequence (still denoted by $\Phi_i(t_i)$) that converges in the flat norm to $\Sigma\in \alpha$. By the lower semicontinuity of mass:
\[S(\alpha)\leq \mathbf{M}(\Sigma) \leq \liminf \mathbf{M}(\Phi_i(t_i))=S(\alpha).\]

Therefore, $\Sigma$ is homologically area-minimizing in $\alpha$. Furthermore, it follows from the flat convergence that $\Sigma$ also divides $U$ into two parts of equal volume. Hence, the isoperimetric inequality implies that there is a positive constant $\delta$ such that $\mathbf{M}(\Sigma\llcorner U)\geq \delta>0$. Consequently, $\Sigma$ passes through $U$. Taking the limit, we see that the union of all homologically area-minimizing hypersurfaces in $\alpha$ is $\mathbb{T}^n$. The fact that it is a minimal foliation follows from a result of Hass in \cite{hass}, who proved that two area-minimizing hypersurfaces in the same homology class $\alpha$ do not intersect:
\begin{Th}
    Let $M^n$ be an orientable Riemannian manifold. Let $G_1, G_2$ be hypersurfaces in $M$, each of which is area-minimizing in its homology class and which are homologous up to orientation. Let $F_1$ and $F_2$ be connected components of $G_1$ and $G_2$, respectively. Then either $F_1\cap F_2$ is empty or $F_1$ and $F_2$ coincide. 
\end{Th}
\end{proof}

\section{The uniform continuity of the energy barrier}\label{appendix:b}
\begin{Lemma}\label{widthcontinuous}
    For any metric $g$ on $\mathbb{T}^n$ and any two positive constants $C_1<C_2$, there exists $K=K(g, C_1, C_2)$ such that for any $C_1g\leq g', g''\leq C_2g$ and $r\in H_{n-1}(\mathbb{T}^n, \mathbb{Z})\setminus \{0\}$, we have
    \[|\Delta W_{g'}(r)-\Delta W_{g''}(r)|\leq K |r||g'-g''|_{g}\]
    where $|r|$ denotes the Euclidean norm of the vector $r\in \mathbb{R}^n$ under some identification $H_{n-1}(\mathbb{T}^n, \mathbb{R})=\mathbb{R}^n$.
\end{Lemma}
\begin{proof}
    It follows exactly as in \cite{widthcontinuous} that:
    \begin{equation}\label{wid1}
    |\omega_{g'}(r)-\omega_{g''}(r)|\leq \left(\left(1+C_1^{-1}|g'-g''|_{g}\right)^{\frac{n-1}{2}}-1\right)C_2^{\frac{n-1}{2}}C\omega_{g}(r)
    \end{equation}
    whenever $C_1g\leq g', g'' \leq C_2g$.
    We want to bound $\omega_g(r)$ in terms of $|r|$. Let $g_{flat}$ be the flat metric induced from the Euclidean metric on $\mathbb{T}^n$. Then $g\leq cg_{flat}$ for some $c>0$. 
    Consider the sweep-out $\Phi(t)$ of $\mathbb{T}^n$ by parallel flat sub-tori with normal direction $r$. We have
    \[\mathbf{M}_{g_{flat}}(\Phi(t))=S_{g_{flat}}(r)=|r|.\]
    It follows that 
    \[\mathbf{M}_{g}(\Phi(t))\leq c^{n-1}|r|, \quad \forall \text{ } t.\]
    Hence, by the definition of the width, we have
    \begin{equation}\label{wid2}
    \omega_{g}(r)\leq \sup \mathbf{M}_{g}(\Phi(t)) \leq c^{n-1}|r|.
    \end{equation}
    Combining (\ref{wid1}) and (\ref{wid2}), we have
    \[|\omega_{g'}(r)-\omega_{g''}(r)|\leq K_1 |r||g'-g''|_{g}.\]
    A similar inequality holds for the stable norm:
    \[|S_{g'}(r)-S_{g''}(r)|\leq K_2|r||g'-g''|_g.\]
    The desired inequality now follows from the triangle inequality.
\end{proof}

\section{A decay lemma for a finite-volume gap}\label{appendixd}

\begin{proof}[Proof of Lemma~\ref{asymptotic}]
Fix $p\in\mathbb R^n$, put $r(x)=|x-p|$, and write
$B_s(p)=\{r<s\}$. Uniform metric comparison gives
$|\nabla_g r|_g\leq C_g$ almost everywhere, with $C_g$ independent
of $p$. The coarea formula (see~\cite{simon}) yields, for $T>0$ and
$\ell>0$,
\begin{equation}\label{appD:coarea}
\begin{aligned}
\int_T^{T+\ell}
\mathcal H_g^{n-1}(G\cap\partial B_s(p))\,ds
&=\int_{G\cap\{T<r<T+\ell\}}|\nabla_g r|_g\,dV_g\\
&\leq C_g\operatorname{Vol}_g(G\setminus B_T(p)).
\end{aligned}
\end{equation}
Consequently, a set of radii of measure at least $\ell/2$ in
$(T,T+\ell)$ satisfies
\begin{equation}\label{appD:selection}
\mathcal H_g^{n-1}(G\cap\partial B_s(p))
\leq\frac{2C_g}{\ell}
\operatorname{Vol}_g(G\setminus B_T(p)).
\end{equation}

The restrictions of $r$ to the two smooth boundary hypersurfaces
have a common full-measure set of regular values, by Sard's theorem.
Thus we can impose transversality while retaining
\eqref{appD:selection}. We may also exclude any prescribed null set
of radii. In particular, after a countable family of locally integral
currents has been fixed, we may require all its restrictions and
slices to be admissible and all its mass measures to assign zero
mass to the cutting sphere. These requirements exclude only a
null set, by slicing and local finiteness of the measures.

Now prescribe positive numbers $\eta_j\to0$ and lower bounds $L_j$.
Set $R_0=0$. Since $\operatorname{Vol}_g(G)<\infty$, choose inductively
\[
T_j>\max\{R_{j-1}+1,j,L_j\},\qquad
2C_g\operatorname{Vol}_g(G\setminus B_{T_j}(p))<\eta_j.
\]
Apply \eqref{appD:selection} with $T=T_j$ and $\ell=1$, choosing
$R_j\in(T_j,T_j+1)$ with the required transversality and admissibility.
Then $R_j\to \infty$, $R_j>L_j$, and
\[
\mathcal H_g^{n-1}(G\cap\partial B_{R_j}(p))\leq\eta_j\longrightarrow0.
\]
This proves the lemma.
\end{proof}

\section{Finiteness of stable minimal hypersurfaces in the gap}\label{appendix:e}
In this section, we prove Proposition \ref{cutstable}. We first establish that the volume of an open set bounded by a closed, connected, stable minimal hypersurface inside $G$ cannot be arbitrarily small.
\begin{Lemma}\label{volumebound}
    There exists a positive constant $v_0$ depending only on the metric $g$, such that for any closed, connected, stable minimal hypersurface $S$ inside $G$, if $S=\partial U_S$, then $\operatorname{Vol}_n(U_S)\geq v_0$.
\end{Lemma}
\begin{proof}
    By the standard curvature estimate for stable minimal hypersurfaces \cite{curvatureestimate}, there exists a positive constant $C$ such that for all closed stable minimal hypersurfaces $S$, it holds that \[\sup_{x\in S} |A_S(x)|\leq C.\] 
    
    In particular, the normal curvatures of $S$ are uniformly bounded. It follows from \cite{bishop}, Theorem 1.2, that there exists a positive constant $\rho$, depending only on $C$ and $g$, such that 
    \[\sup_{x\in U_S} d_g(x, S)\geq \rho.\]
    
    In other words, $U_S$ contains a ball of radius $\rho$. Since the metric is uniformly bounded on $\mathbb{R}^n$, there is a uniformly positive lower bound $v_0$ on the volume of a ball of radius $\rho$. We conclude that
    \[\operatorname{Vol}_n(U_S)\geq \inf_{x}\operatorname{Vol}_n(B_g(x, \rho)):=v_0.\]
    
    This proves the lemma.
\end{proof}

We now construct the desired collection of stable hypersurfaces inductively as follows. Start with a closed, connected, stable minimal hypersurface \[\widetilde{S}_1=\partial U_{\widetilde{S}_1}\subset G.\] By Lemma \ref{volumebound}, 
\[\operatorname{Vol}_n(U_{\widetilde{S}_1})\geq v_0>0.\]

We claim that there exists a maximal open set $U_{S_1}$ inside $G$ containing $U_{\widetilde{S}_1}$ such that $S_1$ is a closed, connected, stable minimal hypersurface. Suppose not. Then one can find an infinite sequence of nested open sets \[U_{\widetilde{S}_1}\subset U_{\widetilde{S}_2}\subset U_{\widetilde{S}_3} \subset \cdots \]
such that $\partial U_{\widetilde{S}_i}=\widetilde{S}_i$ is a closed, connected, stable minimal hypersurface. 

As in the argument of \cite{nested}, Lemma 6, there exists $i_0\in \mathbb{N}$ such that for all $j>i_0$, the hypersurface $\widetilde{S}_j$ can be written as a graph locally over $\widetilde{S}_{i_0}$, with uniformly bounded slopes. Consequently, $\operatorname{Vol}_{n-1}(\widetilde{S}_j)$ is uniformly bounded in terms of $\operatorname{Vol}_{n-1}(\widetilde{S}_{i_0})$. 

By the compactness theorem for minimal hypersurfaces with uniformly bounded area and curvature, a subsequence of $\widetilde{S}_i$ converges smoothly (with possible multiplicity) to a closed minimal hypersurface $S_\infty\subset G$. For sufficiently large $k$, we can then write $\widetilde{S}_k$ as a graph of a smooth function $u_k: S_\infty \to \mathbb{R}$, where $u_k\to 0$ in $C^\infty(S_\infty)$. The normalizing sequence 
\[v_k:=\dfrac{u_k}{||u_k||_{C^{2, \alpha}(S_\infty)}}\]
converges and produces a nontrivial Jacobi field along $S_\infty$, contradicting the bumpiness of $g$.\\

Therefore, such an infinite nested sequence cannot exist, and we have the maximal open set $U_{S_1}$ as claimed. We then remove $U_{S_1}$ from $G$ and repeat the same construction to $G\setminus U_{S_1}$, which still has finite volume. Since we remove an open set of volume at least $v_0$ at every step and $G$ has finite volume, the process must terminate after finitely many steps, and we obtain an optimal collection $\{S_1, S_2, \ldots, S_j\}$ of pairwise disjoint closed, connected, stable minimal hypersurfaces, with the corresponding open sets $\{U_{S_1}, U_{S_2}, \ldots, U_{S_j}\}$ such that $S_i=\partial U_{S_i}$.
\begin{claim}
    For any closed, minimal hypersurface $\Sigma$ contained in the interior of $\Omega_{R_k}=G_{R_k}\setminus \left(\cup_{i=1}^{j} U_{S_i}\right)$, we have
    \[\mathbf{M}(\Sigma)\geq \mathbf{M}(\Gamma_0^{R_k})+\mathbf{M}(\Gamma_1^{R_k}).\]
\end{claim}
\begin{proof}[Proof of claim]
    
     Since $H_{n-1}(\mathbb{R}^n, \mathbb{Z})=0$, $\Sigma$ is separating, we can cut along $\Sigma$ and consider the component $\Omega'_{R_k}$ that contains both $\Gamma_0^{R_k}$ and $\Gamma_1^{R_k}$. 
     
     By the maximality property of the collection $S_1, S_2, \ldots, S_j$, the hypersurface $\Sigma$ is unstable. Hence, it admits a neighborhood foliated by hypersurfaces with strictly smaller area than $\Sigma$. Therefore, by minimizing the area in the homology class of $\Sigma$ inside $\Omega'_{R_k}$, we obtain a stable minimal hypersurface $\widetilde{\Sigma}$ such that
     \[\mathbf{M}(\widetilde{\Sigma})<\mathbf{M}(\Sigma).\]
     In particular, $\widetilde{\Sigma}\neq \Sigma$. There are two cases:\\
        
     \textbf{Case 1}: Suppose that $\widetilde{\Sigma}$ contains a stable component $\Sigma'\subset \Omega'_{R_k}$. This minimal hypersurface is disjoint from $S_1, S_2, \ldots, S_j$. Suppose that inside $G$ we have $\Sigma'=\partial U_{\Sigma'}$. Then $U_{\Sigma'}$ is either disjoint from all the $U_{S_i}$, or contains some of the $U_{S_i}$. Nonetheless, both cases contradict the maximality of the collection $\{S_1, S_2, \ldots, S_j\}$. \\
     
     \textbf{Case 2}: If $\widetilde{\Sigma}$ contains no closed, stable minimal hypersurfaces in the interior of $\Omega'_{R_k}$, then it must be the whole boundary of $\Omega'_{R_k}$, minus $\Sigma$. Since $\Gamma_0^{R_k}$ and $\Gamma_1^{R_k}$ are parts of the boundary of $\Omega'_{R_k}$, we have
     \[\mathbf{M}(\widetilde{\Sigma})\geq \mathbf{M}(\Gamma_0^{R_k})+\mathbf{M}(\Gamma_1^{R_k}).\]
     This proves the claim and completes the proof of Proposition \ref{cutstable}.
\end{proof}
\begingroup
\setlength{\emergencystretch}{3em}
\printbibliography
\endgroup

\end{document}